\documentclass[reqno]{amsart}
\usepackage{mathrsfs}
\usepackage{amsmath}
\usepackage{amsthm}
\usepackage{amsfonts}
\usepackage{amssymb}
\usepackage{url}
\usepackage{enumerate}
\usepackage[pdftex,bookmarks=true]{hyperref}
  \usepackage[usenames, dvipsnames]{color}
\usepackage{verbatim}

\usepackage{pdfsync}

\newcommand\Vol{{\operatorname{Vol}}}
\newcommand\rank{{\operatorname{rank}}}

\newcommand\R{{\mathbf{R}}}

\newcommand\Q{{\mathbf{Q}}}

\renewcommand\P{{\mathbf{P}}}
\newcommand\E{{\mathbf{E}}}

\newcommand\Z{{\mathbf{Z}}}

\newcommand\F{{\mathbf{F}}}
\newcommand\I{{\mathbf{I}}}

\newcommand\Prob{\P}

\newcommand\ep{\varepsilon}
\newcommand\al{\alpha}
\newcommand\la{\lambda}

\renewcommand\Pr{{\mathbf P }}

\newcommand\CA{{\mathcal A}}

\newcommand\CE{{\mathcal E}}
\newcommand\CF{{\mathcal F}}

\newcommand\CR{{\mathcal R}}

\newcommand\CT{{\mathcal T}}

\renewcommand\mod{\ \operatorname{mod}\ }
\renewcommand \small {\scriptsize}

\newcommand\supp{\mathbf{supp}}

\newcommand\eps{\varepsilon}

\newcommand\codim{\operatorname{codim}}
\newcommand\lang{\langle}
\newcommand\rang{\rangle}
\newcommand\wh{\widehat}
\newcommand\Sp{\operatorname{Span}}
\renewcommand\a{\alpha}

\newcommand\cok{\mathbf{Cok}}
\newcommand\Aut{\mathbf{Aut}}
\newcommand\bs{\backslash}

\newcommand{\ra}{\rightarrow}
\newcommand\isom{\simeq}
\newcommand\Hom{\operatorname{Hom}}
\newcommand\Sur{\operatorname{Sur}}
\newcommand\sub{\subset}
\newcommand\im{\operatorname{im}}
\newcommand\rk{{\operatorname{rank}}}
\newcommand{\ora}{\overrightarrow}

\newcommand\co{c_1}
\newcommand\Cz{C_0}
\newcommand\Co{C_0}
\newcommand\cz{c_0}
\newcommand\ct{c_2}

\newcommand\ao{d_1}
\newcommand\Ao{D_1}
\newcommand\Aot{D_1}
\newcommand\at{f}
\newcommand\ad{\gamma}

\newcommand\Ro{R}

\newcommand\GL{\operatorname{GL}}
\newcommand\CP{{\mathfrak P}}
\newcommand\CB{{\mathfrak W}}
\newcommand\Con{\mathcal{C}}
\newcommand\Drop{\mathcal{D}}
\newcommand\Good{\mathcal{G}}

\usepackage{fullpage}

\theoremstyle{plain}
  \newtheorem{theorem}[subsection]{Theorem}

  \newtheorem{prop}[subsection]{Proposition}
  
  \newtheorem{lemma}[subsection]{Lemma}
  \newtheorem{corollary}[subsection]{Corollary}
  \newtheorem{cor}[subsection]{Corollary}

  \newtheorem{claim}[subsection]{Claim}

\theoremstyle{definition}
  \newtheorem{definition}[subsection]{Definition}

\begin{document}

\title{Random integral matrices: universality of surjectivity and the cokernel}

\author{Hoi H. Nguyen}
\address{Department of Mathematics\\ The Ohio State University \\ 231 W 18th Ave \\ Columbus, OH 43210 USA}
\email{nguyen.1261@math.osu.edu}

\author{Melanie Matchett Wood}
\address{Department of Mathematics\\
University of Wisconsin-Madison \\ 480 Lincoln Drive \\
Madison, WI 53705 USA}  
\email{mmwood@math.wisc.edu}


\begin{abstract} For a random matrix of entries sampled independently from a fairly general distribution in $\Z$ we study the probability that the cokernel is isomorphic to a given finite abelian group, or when it is cyclic.  This includes the probability that the linear map between the integer lattices given by the matrix is surjective.  We show that these statistics are asymptotically universal (as the size of the matrix goes to infinity), given by precise formulas involving zeta values, and agree with  distributions defined by Cohen and Lenstra, even when the distribution of matrix entries is very distorted. Our method is robust and works for Laplacians of random digraphs and sparse matrices with the probability of an entry non-zero only $n^{-1+\eps}$.
\end{abstract}

\maketitle



\section{Introduction}
 For square matrices $M_{n \times n}$ of random discrete entries, the problem to estimate the probability $p_n$ of $M_{n \times n}$ being singular has attracted quite a lot of attention.  In the  60's Koml\'os \cite{Ko} showed $p_n=O(n^{-1/2})$ for entries $\{0,1\}$ with probability each $1/2$.  This bound was significantly improved by Kahn, Koml\'os, and Szemer\'edi in the 90's to $p_n \le 0.999^n$ for $\pm 1$ entries.
About ten years ago, Tao and Vu \cite{TV} improved the bound for $\pm 1$ entries to $p_n\le (3/4+o(1))^n$. The most recent record is due to and Bourgain, Vu and Wood \cite{BVW} who showed $p_n \le (\frac{1}{\sqrt{2}}+o(1))^n$ for $\pm 1$ entries and gave exponential bounds for more general entries as well.  We also refer the reader to \cite{RV} by Rudelson and Vershynin for implicit exponential bounds. For sparse matrices having  entries  $0$ with probability $1-\alpha_n$, Basak and Rudelson \cite{BR} proved $p_n\leq e^{-c\alpha_n n}$ for $\alpha_n\geq C \log n/n$ and for rather general entries, including the adjacency matrices of sparse Erd\H{o}s-R\'{e}nyi random graphs.

When $M_{n \times n}$ has integral entries, these results imply that with very high probability the linear map $M_{n \times n}:\Z^n\ra \Z^n$ is injective.  Another important property of interest is surjectivity, it seems natural to wonder if with high probability $M_{n \times n}:\Z^n\ra \Z^n$ is surjective (see \cite{S-thesis,NP})?  However, recent results of the second author show that the surjectivity probability goes to $0$ with $n$ (e.g. that is implied by \cite[Corollary 3.4]{W1}).  The main result of this paper will imply that when the matrix has more columns than rows, e.g. $M_{n \times (n+1)}:\Z^{n+1}\ra \Z^n$, we have surjectivity with positive probability strictly smaller than one.

We make the following definition to restrict the types of entries our random matrices will have.
We say a random integer $\xi_n$ is \emph{$\alpha_n$-balanced} if for every prime $p$ we have
\begin{equation}\label{eqn:alpha}
\max_{r \in \Z/p\Z} \P(\xi_n\equiv r \pmod{p}) \le 1-\alpha_n.
\end{equation}
Our main result tells us not only whether $M_{n \times (n+u)}$ is surjective, but more specifically about the  cokernel $\cok(M_{ n \times (n+u)})$, which is the quotient group $\Z^{n}/M_{n \times (n+u)}(\Z^{n+u})$ and gives the  failure of surjectivity.
\begin{theorem}\label{theorem:sur}
For integers $n,u\geq 0$, let $M_{n \times (n+u)}$ be an integral $n \times (n+u)$ matrix with entries 
 i.i.d copies of an {\it $\alpha_n$-balanced} random integer $\xi_n$, with $ \alpha_n \ge n^{-1+\eps}$ and  $|\xi_n|\leq n^T$ for any fixed parameters $0<\eps<1$ and $T>0$ not depending on $n$. 
For any fixed finite abelian group $B$ and $u\geq 0$,
\begin{equation}\label{eqn:sur:fix}
\lim_{n\to \infty}  \P\Big(\cok(M_{ n \times (n+u)}) \simeq B \Big)  = \frac{1}{|B|^u |\Aut(B)|}  \prod_{k=u+1}^\infty \zeta(k)^{-1}.
\end{equation}
\end{theorem}
Here $\zeta(s)$ is the Riemann zeta function.  In particular, as $n\ra\infty $, the map $M_{n \times (n+1)}:\Z^{n+1}\ra \Z^n$ is surjective with probability approaching 
$\prod_{k=2}^\infty \zeta(k)^{-1}\approx 0.4358.$  The one extra dimension mapping to $\Z^n$ brought the surjectivity probability from $0$ to $\approx 0.4358.$ 

Note that the product $\prod_{k=u+1}^\infty \zeta(k)^{-1}$ in~\eqref{eqn:sur:fix} is non-zero for $u\geq 1$, but $\zeta(1)^{-1}=0$.
So Theorem~\ref{theorem:sur} shows that every possible finite cokernel appears with  positive probability when $u\geq 1$.  (Note that when the matrix has full rank over $\R$, the cokernel must be finite.)  Theorem~\ref{theorem:sur} is a universality result because these precise positive probabilities do not depend on the distribution of $\xi_n$, the random entries of our matrices.  As a simple example, if we take an $n\times(n+1)$ random matrix with entries all $0$ or $1$, whether we make entries $0$ with probability $\frac{1}{100}$, $\frac{1}{2}$, or $1-n^{-8/9}$, we obtain the exact same asymptotic probability of the map $\Z^{n+1}\ra \Z^n$ being surjective.  If we take entries from $\{-17,0,6,7\}$ with respective probabilities $\frac{2}{3},\frac{1}{n},\frac{1}{6}-\frac{1}{n},\frac{1}{6}$, the asymptotic probability of surjectivity is unchanged. 
Our theorem allows even more general entries as well.

Further, we prove the following.
\begin{theorem}\label{theorem:cyclic}  Let $M_{n \times (n+u)}$ be as in Theorem~\ref{theorem:sur}. 
We have
 \begin{align*}
\lim_{n\to \infty} \P\Big(\cok(M_{n\times (n+u)}) \mbox{ is cyclic}\Big) = \prod_{p \textrm{ prime}} (1 + p^{-(u+1)}(p-1)^{-1}) \prod_{k=u+2}^\infty \zeta(k)^{-1}.
\end{align*}
 \end{theorem}
Note that even when $u=0$, the limiting probability here is positive. For $u=0$, this probability has been seen in several papers studying the probability that a random lattice in $\Z^n$ is co-cyclic (gives cyclic quotient), in cases when these lattices are drawn from the nicest, most uniform distributions, e.g. uniform on lattices up to index $X$ with $X\ra\infty$ \cite{CKK,NS,Pet}, or with basis with uniform entries in $[-X,X]$ with $X\ra\infty$ \cite{SW}.
Stanley and Wang have asked whether the probability of having cyclic cokernel is universal (see  \cite[Remark 4.11 (2)]{SW} and \cite[Section 4]{Stan}).
 Theorem~\ref{theorem:cyclic} proves this universality, showing that the same probability of cocylicity occurs when the lattice is given by $n$ random generators from a rather large class of distributions, including ones that are rather distorted mod $p$ for each prime $p$.


Moreover, we show the same results hold if we replace $\cok(M_{n\times (n+1)})$ with the total sandpile group of an 
Erd\H{o}s-R\'enyi simple random digraph, proving a conjecture of Koplewitz \cite[Conjecture 1]{S2} (see Theorem~\ref{theorem:sur:L}).  This allows some dependence in the entries of our random matrices, since the diagonal of the graph Laplacian depends on the other entries in the matrix.  In particular, this says that with asymptotic probability $\prod_{k=2}^\infty \zeta(k)^{-1}\approx 0.4358$  an Erd\H{o}s-R\'enyi  random digraph is \emph{co-Eulerian}, which Farrell and Levine \cite{FL} define to be any of several equivalent definitions including a simple condition for when chip-firing configurations on the graph stabilize and the condition that recurrent states in the rotor-router model are in a single orbit.  In contrast to the distribution of sandpile groups of Erd\H{o}s-R\'enyi random graphs, where for each  finite abelian group $B$, the sandpile group is $B$ with asymptotic probability $0$ \cite[Corollary 9.3]{W0}, for Erd\H{o}s-R\'enyi  random \emph{digraphs}, we show that each finite abelian group appears with positive asymptotic probability as the total sandpile group.  
Moreover, the universality in our theorems proves that all of these positive limiting probabilities  do not depend on the edge density of the random graph.

Previous work of the second author \cite[Corollary 3.4]{W1} determined the probabilities of these $\cok(M_{n\times (n+u)})$ having any given Sylow $p$-subgroup for a fixed prime $p$ or finite set of primes $p$.  
The two significant advances of this work over previous work are (1) that we determine the distribution of the entire cokernel, not just the part of it related to a finite set of primes, and (2) that we allow our random matrix entries to be  more distorted mod $p$ as $n$ increases, for example allowing sparse matrices where entries are non-zero with probability $n^{-1+\epsilon}$.  

Our proofs require considering primes in three size ranges separately, and in each range we use different methods. 
Our works builds on methods from previous work, including that of Tao and Vu \cite{TV, TVjohn, TVinverse}, the first author and Vu \cite{NgV}, Maples \cite{M1}, the second author \cite{W0,W1}, and
the first author and Paquette \cite{NP}.  The key ideas original to this paper are in our treatment of large primes, where we prove a result that lifts structured normal vectors from characteristic $p$ to characteristic $0$, crucially for $p$ in a range much smaller than $n^{n/2}$.

\subsection{Further results and connections to the literature}

We also show asymptotic almost sure surjectivity when $u\ra\infty$ with $n$, proving a conjecture of Koplewitz \cite[Conjecture 2]{S-thesis}.

\begin{theorem}\label{theorem:bigu}  Let $M_{n \times (n+u)}$ be as in Theorem~\ref{theorem:sur}. 
Then 
 \begin{equation}\label{eqn:sur:large}
\lim_{\min(u,n)  \to \infty}   \P\Big(\cok(M_{ n \times (n+u)}) \simeq \{id\}\Big)  = 1.
\end{equation}
 \end{theorem}

Theorem~\ref{theorem:sur} has several nice corollaries, including the $u\geq 1$ cases of Theorem~\ref{theorem:cyclic} and
the following (see Lemma~\ref{L:anyset} for why these are corollaries).
\begin{cor}\label{cor:anyset} For any fixed $u\ge 0$
 $$\lim_{n\to \infty} \P\left(M_{ n \times (n+u)}: \Z^{n+u}\ra\Z^n \textrm{ is surjective}\right)  =\prod_{k=u+1}^\infty \zeta(k)^{-1}.$$
 Also, for any fixed $u\ge 1$ 
\begin{align*}
 \lim_{n\to \infty} \P\Big(\det(M_{n\times (n+u)}) \mbox{ is square-free}\Big) = \prod_{p \textrm{ prime}} (1 + p^{-u}(p-1)^{-1}) \prod_{k=u+1}^\infty \zeta(k)^{-1}.
\end{align*}
\end{cor}


To give a heuristic for why inverse zeta values arise in these probabilities, note that $M_{n\times(n+1)}$ is surjective if and only if its reduction to modulo $p$ is surjective for all primes $p$. We then make two idealized heuristic assumptions on $M_{n\times(n+1)}$.  (i) (uniformity assumption) Assume that for each prime $p$ the entries of $M_{n \times (n+1)}$ are uniformly distributed modulo $p$. In this case,  a simple calculation gives the probability for $M_{n\times (n+1)}$ being surjective modulo $p$ is $\prod_{j=2}^n (1-p^{-j}) (1-p^{-n-1})$. (ii) (independence assumption) We next assume that the statistics of $M_{n \times (n+1)}$ reduced to modulo $p$ are asymptotically mutually independent for all primes $p$. Under these assumptions, as $n \to \infty$, the probability that $M_{n \times (n+1)}$ is surjective would be asymptotically the product of all of the surjectivity probability modulo $p$, which leads to the number $\prod_{k=2}^\infty \zeta(k)^{-1}$ as seen.  The matrices in this paper do not have to satisfy either assumption, and indeed they can violate them dramatically.  For example, if the matrix entries only take values $0$ and $1$, then they cannot be uniformly distributed mod any prime $>2$, and the matrix entries mod $3$ are not only not independent from the entries mod $5$, but they are in fact determined by the entries mod $5$.  The work of this paper is in showing that even for rather general random matrices, universality holds and gives the same cokernel distributions as for the simplest random matrices.


For our Theorem~\ref{theorem:sur}, we remark that for $u\geq 1$, the limiting probabilities ${|B|^{-u} |\Aut(B)|^{-1}}  \prod_{k=u+1}^\infty \zeta(k)^{-1}$ in Theorem \ref{theorem:sur} do sum to 1 (use \cite[Corollary 3.7 (i)]{CL} with $s=u$ and $k=\infty$).  This gives, for each $u\geq 1$, a probability distribution on finite abelian groups.  Cohen and Lenstra \cite{CL} introduced these distributions to conjecture that the $u=1$ distribution is the distribution of class groups of real quadratic number fields (except for the Sylow 2-subgroup).
Friedman and Washington \cite{FW} later proved that if $M_{n\times n}$ has independent entries taken from Haar measure on the $p$-adics $\Z_p,$ then  for a finite abelian $p$-group $B$ we have 
$$\lim_{n\ra\infty}\P\left(\cok(M_{ n \times n}) \simeq B\right)=|\Aut(B)|^{-1}  \prod_{k=u+1}^\infty (1-p^{-k}).$$ 
The limit is proven by giving an explicit formula for the probability for each $n$.
A similar argument shows that for $M_{n\times (n+u)}$ with independent entries taken from Haar measure on $\wh{\Z}$, the profinite completion of $\Z$ (these are exactly the matrices with entries that satisfy the two heuristic assumptions above), we have for every finite abelian group $B$ that 
$$\lim_{n\to \infty}  \P\left(\cok(M_{ n \times (n+u)}) \simeq B \right)  = {|B|^{-u} |\Aut(B)|}^{-1}  \prod_{k=u+1}^\infty \zeta(k)^{-1}.$$
 This is because as $\wh{\Z}=\prod_p \Z_p$, this Haar measure is the product of the $p$-adic Haar measures.  Building on work of Ekedahl \cite{Eke}, 
 Wang and Stanley \cite{SW} find that the cokernels (equivalently, the Smith normal form) of random $n\times m$ matrices for \emph{fixed} $n$ and $m$ and independent, uniform random integer entries in $[-X,X]$ as $X\ra\infty$ match those for entries from Haar measure on $\wh{\Z}$.  While this agreement is easy to see for the Sylow subgroups at any finite set of primes (because $X$ will eventually become larger than all of the primes), it was a substantial problem to prove this agreement for all primes at once.


Our approach to proving Theorem~\ref{theorem:sur} and the $u=0$ case of Theorem \ref{theorem:cyclic} involves considering three classes of primes (small, medium, and large) separately, and for each class the argument is rather different.  For small primes, we follow the general approach of \cite{W1}: finding the moments (which are an expected number of surjections) by dividing the a priori possible surjections into classes and obtaining, for each class, a bound on the number of possible surjections in it and a bound of the probability that any of those surjections are realized. 
Our advance over \cite{W1} is that we can allow sparser matrices, and to obtain this improvement we have to both refine the classes into which we divide surjections and the bounds we have for the probabilities for each class.  For medium primes, our starting point is a theorem from \cite{NP}, which carries out ideas of Maples \cite{M1} to show that, modulo a prime, under likely conditions, each time we add a column to our matrix, the probability that the new column is in the span of the previous columns (mod a prime) is near to the probability for a uniform random column.  Our contribution is to show the bounds on ``likely conditions'' and ``nearness of probability'' can be turned into a bound on how the ranks of the final matrix compare to the ranks of a uniform random matrix.  We do this via a rather general approach using a coupling of Markov chains.  For large primes, our approach is  new.  We cannot control whether the rank of our matrix drops by $1$ modulo any particular large prime, but considering columns being added one at a time, once the rank drops by $1$ modulo a prime, we shows that it is not likely to drop again.  We do this by showing that the column spaces are unlikely to have a normal vector with many of its coefficients in a generalized arithmetic progression mod $p$, and then proving a new inverse Erd\H{o}s-Littlewood-Offord result over finite fields for sparse random vectors based on the method from \cite{NgV} by the first author and Vu.  However, the probabilities of structured normal vectors mod $p$ are still too large to add up over all $p$, and a key innovation of the paper is that for primes $>e^{n^{1-\eps/3}}$ we show that having a non-zero structured normal vector mod $p$ is equivalent to having one in characteristic 0.  
Fortunately, the bounds for $p$ up to which we can sum the probabilities of structured normal vectors mod $p$, and the
bounds for $p$ where we can lift structured normal vectors overlap, and
this allows us to control the probability of structured normal vectors at all primes.  

 
In contrast to the result of Corollary~\ref{cor:anyset}, in the $u=0$ case we are unable to determine  $\P\left(\det(M_{ n \times n}) \mbox{ is square-free}\right)$, though from \cite[Corollary 3.4]{W1} it follows that 
$$\limsup_{n\ra\infty} \P\left(\det(M_{ n \times n}) \mbox{ is square-free}\right) \leq \zeta(2)^{-1}\prod_{k\geq 2} %
\zeta(k)^{-1},$$ 
and we would conjecture the limit is equal to this value.  We can obtain the limiting probability that $\cok(M_{n\times n})$ is the product of a given finite abelian group $B$ and a cyclic group (see Theorem~\ref{T:prodcyc}), and these  are currently the most general classes of abelian groups for which we can obtain universality results for $n\times n$ matrices.
Even for nicely distributed matrix entries and fixed $n$, the question of how often $\det(M_{ n \times n})$ is square-free is very difficult (see \cite{Poonen2003,Bhargava2014b}).

 Our main results work for $\alpha_n \ge n^{-1+\eps}$, which is asymptotically best possible, in terms of the exponent of $n$.  If the matrix entries are $0$ with probability at least $1-\log n / (n+u)$, then the matrix $M_{n \times (n+u)}$ has a row of all $0$'s with non-negligible probability, and thus cannot possibly be surjective or even have finite cokernel.  
  We also refer the reader to \cite{NP} for some partial results where $\alpha_n$ is allowed to be as small as $O(\log n/n)$ and $u$ is comparable to $n$.  Much of the previous work that we build upon has required the matrix entries to be non-zero with probability bounded away from $0$ as $n\ra\infty$.  It is perhaps surprising that even as the matrices have entries being $0$ more and more frequently, the asymptotic probability that $M_{n \times (n+1)}$ is surjective does not change from $\approx .4358$ as long as $\alpha_n \ge n^{-1+\eps}$.
  
  Another advantage of our method (compared to existing results in the literature on classical random matrix theory) is that the bound on the matrix entries can be as large as any polynomial $n^T$ of $n$.  This can be relaxed somewhat by letting $T\ra\infty$ slowly, but it cannot be lifted entirely as the example of Koplewitz shows \cite[Section 4.4]{S-thesis} (see also the discussion after Lemma~\ref{lemma:O}).

 
\vskip .05in
We now explain in more detail the  extension of our results to a natural family of random matrices of dependent entries, namely to the Laplacian of random digraphs. More generally, let $M=M_{n \times n} =(x_{ij})_{1\le i,j\le n}$ be a random matrix where $x_{ii}=0$ and its off-diagonal entries are i.i.d. copies of an integral random variable $\xi_n$ satisfying ~\eqref{eqn:alpha}. A special
case here is when $M$ is the adjacency matrix of an Erd\H{o}s-R\'enyi simple random digraph $\Gamma \in \ora{G}(n,q)$  where each directed edge is chosen independently with probability $q$ satisfying $\alpha_n \le q \le 1 -\alpha_n$. 
 Let $L_M=(L_{ij})$ be the Laplacian of $M$, that is 
\[
L_{ij}=\begin{cases}
-x_{ij} & \mbox{ if } i \neq j\\
\sum_{k=1}^n {x_{ki}} & \mbox{ if } i=j.
\end{cases}
\]
We then denote $S_M$ (or $S_\Gamma$ in the case of digraphs) to be the cokernel of $L$ with respect to the group $\Z_0^n$ of integral vectors of zero entry-sum  
$$S_M = \Z_0^n/L_M\Z^n.$$  When $\Gamma$ is a graph, this group has been called the \emph{sandpile group without sink} \cite{FL}
and the 
 \emph{total sandpile group} \cite{S2} of the graph.  The size of this group was has been called the \emph{Pham Index} \cite{FL}, and was introduced by Pham \cite{Pham} in order to count orbits
of the rotor-router operation.
We will show that Theorems~\ref{theorem:sur} and ~\ref{theorem:cyclic} extend to this interesting setting.

\begin{theorem}\label{theorem:sur:L} Let $0<\eps<1$ and $T>0$ be given. Let $M_{n \times n}$ be a integral $n \times n$ matrix with entries 
 i.i.d copies of an {\it $\alpha_n$-balanced} random integer $\xi_n$, with $ \alpha_n \ge n^{-1+\eps}$ and  $|\xi_n|\leq n^T$. Then for any finite abelian group $B$,
\begin{equation}\label{eqn:sur:fix:L}
\lim_{n\to \infty}  \P\Big(S_{M_{n \times n}} \simeq B\Big)  = \frac{1}{|B| |\Aut(B)|}  \prod_{k=2}^\infty \zeta(k)^{-1}
\end{equation}
and
 \begin{equation}\label{eqn:cyclic:L}
 \lim_{n\to \infty}  \P\Big(S_{M_{n \times n}} \mbox{ is cyclic}\Big)  = \prod_{p \textrm{ prime}} (1 + (p^2(p-1))^{-1}) \prod_{k=3}^\infty \zeta(k)^{-1} .
 \end{equation}
\end{theorem} 
In particular, every finite abelian group $B$ appears with frequency
given in \eqref{eqn:sur:fix:L} as a total sandpile of the random digraph
$\ora{G}(n,q)$ with parameter $n^{-1+\eps} \le q\le
1-n^{-1+\eps}$.  In a paper about the sandpile (or chip-firing) and rotor-router models, Holroyd, Levine, M\'esz\'aros, Peres, Propp, and Wilson asked if there was an infinite family of non-Eulerian strongly connected digraphs such that the unicycles are in a single orbit of the the rotor-router operation \cite[Question 6.5]{Hol}.  Pham \cite{Pham} then gave an infinite family with a single orbit, and asked if the probability of a single orbit for an Erd\H{o}s-R\'enyi digraph in fact goes to $1$. Koplewitz \cite{S2} gave an upper bound on this probability.
 We have now shown that the desired graphs with a single rotor-router orbit occur with asymptotic probability $\prod_{k=2}^\infty \zeta(k)^{-1}\approx 43.58\%$ (matching the upper bound from \cite{S2}) among Erd\H{o}s-R\'enyi digraphs.  Moreover, for every $k$, our result gives an explicit positive asymptotic probability for exactly $k$ orbits.  

Farrell and Levine show that this number of orbits is the size of the total sandpile group \cite[Lemma 2.9, Theorem 2.10]{FL}, and coined the term \emph{co-Eulerian} for digraphs where the total sandpile group is trivial.  Farrell and Levine also show that for a strongly connected digraph $\Gamma$ the algebraic condition $S_{\Gamma} = \{id\}$ is equivalent to a more combinatorial condition \cite[Theorem 1.2]{FL}, i.e. in this graph a chip configuration $\sigma$ on $\Gamma$ stabilizes after a finite number of legal firings if and only if $|\sigma| \le |E|-|V|$.  Further, they prove that minimal length of a multi-Eulerian tour depends inversely on the size of the total sandpile group \cite[Theorem 5]{FL2}, showing that $|S_{\Gamma}|$  measures ``Eulerianness'' of the graph.  
\begin{corollary} Let $0<\eps<1$ be given and let $q$ be a given parameter such that $n^{-1+\eps} \le q\le 1-n^{-1+\eps}$. Then
\begin{align*}
\lim_{n \to \infty} \P\Big(\ora{G}(n,q) \mbox{ is co-Eulerian}\Big) &=
\\
\lim_{n \to \infty} \P\Big(\ora{G}(n,q) \mbox{ is strongly connected, non-Eulerian, and co-Eulerian}\Big) &= \prod_{k=2}^\infty \zeta(k)^{-1}.
\end{align*}
\end{corollary}

The corollary follows since 
$\ora{G}(n,q)$ is strongly connected and non-Eulerian asymptotically almost surely. 
Although our general method to prove Theorem~\ref{theorem:sur:L} follows the proof method of Theorems~\ref{theorem:sur} and ~\ref{theorem:cyclic}, here the dependency of the entries in each column vector and the non-identical property of the columns pose new challenges. Among other things, for the medium primes we will need to prove a non-i.i.d.~analog of the result of \cite{NP} that we used in the i.i.d.~case.  For small primes, when $\alpha_n$ is constant, our results specialize to those of Koplewitz \cite{S2}, who determined the asymptotic probabilities of given Sylow $p$-subgroups of these total sandpile groups for finitely many primes $p$.  However, as in our main theorem, we require a further refined method to deal with smaller $\alpha_n$.

Note that for $M_{n \times n}$ with general i.i.d. $\alpha_n$-balanced integer entries the results of \cite{BR} do not apply to bound the singularity probability. However, a recent result by Paquette and the first author \cite{NP}  (following the preprint \cite{M1} of Maples) shows that the singularity probability $p_n$ can also be bounded in this case by $e^{-c \alpha_n n}$ with $\alpha_n \ge C \log n/n$ (see also Theorem~\ref{th:allmed}).  We prove the same bound for the singularity of digraph Laplacians in  Corollary~\ref{cor:singularity:L}.

\subsection{Outline of the paper}
In Section~\ref{section:method}, we state our results for each class of primes, and show that Theorems~\ref{theorem:sur},  \ref{theorem:cyclic}, \ref{theorem:bigu}, and \ref{theorem:sur:L} follow from these results.  We will present our main arguments for the i.i.d. case in Sections ~\ref{section:O}-\ref{section:structures}. The main arguments for the graph Laplacian case are in Section~\ref{section:Laplacian}, building on the treatments in the i.i.d. case.

\subsection{Notation} We write $\P$ for probability and $\E$ for expected value. For an event $\mathcal{E}$, we write $\bar{\mathcal{E}}$ for its complement. We write $\exp(x)$ for the exponential function $e^x$.
We use $[n]$ to denote $\{1,\dots,n\}$.  For a given index set $J \subset [n]$ and a vector $X= (x_1,\dots, x_n)$, we write $X|_J$ to be the subvector of $X$ of components indexed from $J$. Similarly, if $H$ is a subspace then $H|_J$ is the subspace spanned by $X|_J$ for $ X\in H$.  For a vector $w=(w_1,\dots,w_n)$ we let $\supp(w)=\{i\in[n] | w_i\ne 0\}$. We will also write $X \cdot w$ for the dot product $\sum_{i=1}^n x_i w_i$.
We say $w$ is a \emph{normal} vector for a subspace $H$ if $X\cdot w=0$ for every $X\in H$.

 For $0\le u\le n$, the matrix $M_{n\times (n-u)}$ is the submatrix of the first $n-u$ columns of $M_{n\times n}$. Sometimes we will write the Laplacian $L_M$ as $L_{n\times n}$, and so $L_{n \times (n-u)}$ is the submatrix of the first $n-u$ columns of $L_M$. We also write $\Z_0^n/p$ to denote the set of vectors of zero-entry sum in $(\Z/p\Z)^n$. 

For a finite abelian group $G$ and a prime $p$, we write $G_p$ for the Sylow $p$-subgroup of $G$. For a set $P$ of primes, we write $G_P:=\prod_{p\in P}G_p$.

 Throughout this paper $C_i, K_i, c_i, \delta, \eta, \eps, \la$, etc  will denote positive constants. When it does not create confusion, the same
 letter may denote different constants in different parts of the proof. 
 The value of the constants may depend on other constants we have chosen,
but will
never depend on the dimension $n$, which is regarded as an asymptotic parameter going to infinity.
 We consider many functions of $n$ and other parameters, e.g. including
$u, \{\xi_i\}_i, \alpha, \eps,T,d,p,q$.  We say ``$f(n,\dots )\in O_S(g(n,\dots))$,'' where $S$ is a subset of the parameters, to mean for any values $v_1,\dots,v_m$ of the parameters in $S$,
there is exists a constant $K>0$ depending on $v_1,\dots,v_m$, such that for all $n$ sufficiently large given $v_1,\dots,v_m$, and \emph{all} allowed values of the parameters not in $S$, that
$|f(n,v_1,\dots,v_m,\dots )|\leq Kg(n,v_1,\dots,v_m,\dots).$

\section{Organization of the proof of Theorems~\ref{theorem:sur}, \ref{theorem:cyclic}, \ref{theorem:bigu} and \ref{theorem:sur:L}}\label{section:method}

We will be mainly focusing on the i.i.d. case to prove Theorems ~\ref{theorem:sur}, \ref{theorem:cyclic}, and \ref{theorem:bigu}. The results for the Laplacian case will be shown in a similar fashion. We prove Theorems~\ref{theorem:sur}, \ref{theorem:cyclic}, and \ref{theorem:bigu} for $M_{n \times (n+u)}$ by checking if the Sylow-$p$ subgroup of $\cok(M_{ n \times (n+u)})$
is equal to $B_p$ for each prime $p$ (or is cyclic for each prime $p$).   The argument will then break up into considering primes in three size ranges with totally different treatments.  

For small primes, we prove the following generalization of \cite[Corollary 3.4]{W1} to sparser matrices, which requires a refinement of the method of \cite{W1}. 
\begin{prop}[Small Primes]\label{prop:small}
Let $M_{n \times (n+u)}$ be as in Theorem~\ref{theorem:sur}.
Let $B$ be a finite abelian group.  Let $P$ be a finite set of primes including all those dividing $|B|$.
Then
$$
\lim_{n\ra\infty} \P\Big(\cok(M_{n\times (n+u)})_P \simeq B \Big) = \frac{1}{|B|^u|\Aut(B)|}\prod_{p\in P} \prod_{k=1}^\infty (1-p^{-k-u}). 
$$
\end{prop}
Proposition~\ref{prop:small} is a special case of Theorem~\ref{th:allsmall}, which allows the matrices to be even sparser and have non-identical entries. This carries the main term of our estimates.

 For medium primes, we combine a result of \cite{NP} with a comparison theorem on the evolving of the matrix ranks to obtain the following.
 
\begin{prop}[Medium Primes]\label{prop:medium}
There are constants $c_0,\eta>0$ such that the following holds. 
Let $M_{n \times (n+u)}$ be as in Theorem~\ref{theorem:sur}.  Let $p$ be a prime.
Then,
\begin{equation}\label{eqn:medu}
\P\Big(M_{n \times (n+u)} \mbox{ mod $p$ is not full rank} \Big)\leq 2p^{-\min(u+1,\eta n)}+O(e^{-c_0\alpha_n n})
\end{equation}
and
\begin{equation}\label{eqn:med}
\P\Big(M_{n \times (n+u)} \mbox{ mod $p$ has rank $\leq n-2$} \Big)\leq 2p^{-\min(2u+4,\eta n)}+O(e^{-c_0\alpha_n n}).
\end{equation}
\end{prop}


Proposition~\ref{prop:medium} will follow from Theorem~\ref{th:allmed} where we allow the matrices to be sparser.  (The big $O$ allows us to require that $n$ is large enough that
$\alpha_n\geq n^{-1+\eps}\geq C_0\log n/n$.)

 Even such a small error bound cannot be summed over all primes, and so for large primes we present a new approach that considers all large primes together.

\begin{prop}[Large Primes]\label{prop:large} 
Let $d>0$ and let $M_{n \times (n+u)}$ and $\eps$ be as in Theorem~\ref{theorem:sur}.  Then,
\begin{equation}\label{eqn:large:n+1prop}
\P\Big(\forall \mbox{ primes } p\geq e^{d \alpha_n n} : M_{n \times (n+1)} \mbox{ mod $p$ has rank at least $n$}\Big) \ge 1 -O_{d,T,\eps}( n^{-\eps}),
\end{equation}
as well as
\begin{equation}\label{eqn:large:nprop}
\P\Big(\forall \mbox{ primes } p\geq e^{d \alpha_n n} : M_{n \times n} \mbox{ mod $p$ has rank  at least $n-1$}\Big)\geq 1 -O_{d,T,\eps}( n^{-\eps}).
\end{equation}


\end{prop}
Proposition~\ref{prop:large} in proven in Sections~\ref{section:largeprimes} and \ref{section:structures}, and is the source of the lower bound on $\alpha_n$ in our theorems. 


The heart of the paper is proving the three propositions above, as the main theorems follow simply from these, as we now show.

\begin{proof}[Proof of Theorems~\ref{theorem:sur} and \ref{theorem:bigu}]
We first prove Theorem~\ref{theorem:bigu}, the case when $u\ra\infty$, where we need to consider only medium and large primes.  
By Equation~\eqref{eqn:medu} of Proposition~\ref{prop:medium}, there are $c_0,\eta>0$ such that
\begin{align*}\P\Big(M_{n \times (n+u)} \mbox{ mod $p$ not full rank  for some $2\leq p \leq e^{c_0\alpha_n n/2}$}\Big) \le &\sum_{2\leq p \leq e^{c_0\alpha_n n/2}}  (2p^{-\min(u+1,\eta n)} + O(e^{-c_0 \alpha_n n}))  \\=&O (\frac{1}{2^{\min(u,\eta n)}} + e^{-c_0 \alpha_n n/2}).
\end{align*}
Combined with Equation~\eqref{eqn:large:n+1prop} of Proposition~\ref{prop:large} (applied to $d=c_0/2$) we obtain
$$\P\Big(M_{n \times (n+u)} \mbox{ mod $p$ is full rank  for all $p\geq 2$}\Big) \ge 1-  O_{T,\eps}(\frac{1}{2^{\min(u,\eta n)}} +e^{-c_0 \alpha_n n/2}+ n^{-\eps}),$$
completing the proof of Equation~\eqref{eqn:sur:large}.

We next turn to Equation~\eqref{eqn:sur:fix}. Let $k_0$ be fixed and $u\ge 1$ be fixed. By applying Equation~\eqref{eqn:medu} of Proposition~\ref{prop:medium}, for $n$ large enough given $\eta$ and $u$ we have
\begin{align*}
\P\Big(M_{n \times (n+u)} \mbox{ mod $p$ is not full rank  for some $k_0 \le p \le e^{c_0 \alpha_n n/2}$}\Big) &\le \sum_{p=k_0}^{e^{c_0 \alpha_n n/2}}  (2p^{-(u+1)} + O(e^{-c_0 \alpha_n n}))\\  
&= O(\frac{1}{k_0}+e^{-c_0 \alpha_n n/2}).
\end{align*} 
Combined  with Equation~\eqref{eqn:large:n+1prop} of Proposition~\ref{prop:large} we obtain
$$\P\Big(M_{n \times (n+u)} \mbox{ mod $p$ is full rank  for all $ p\geq k_0 $}\Big) \ge 1-  O_{T,\eps}( \frac{1}{k_0}+e^{-c_0 \alpha_n n/2}+ n^{-\eps}).$$
Now let $k_0$ be at least as large as the largest prime divisor of $|B|$, and let $P$ be the collection of primes up to $k_0$.
By Proposition~\ref{prop:small},
\begin{equation}\label{eqn:compB}
 \P\Big(\cok(M_{n\times (n+u)})_P \simeq B \Big) = \frac{1}{|B|^u|\Aut(B)|}\prod_{p\le k_0} \prod_{k=1}^\infty (1-p^{-k-u}) + o_{\{\xi_i\}_i,B,u}(1).
\end{equation}
Putting the two bounds together,
$$ \P\Big(\cok(M_{n\times (n+u)}) \simeq B \Big) \ge  \frac{1}{|B|^u|\Aut(B)|}\prod_{p\le k_0} \prod_{k=1}^\infty (1-p^{-k-u}) - O_{T,\eps}( \frac{1}{k_0} +e^{-c_0 \alpha_n n/2}+ n^{-\eps})
+o_{\{\xi_i\}_i,B,u}(1)
.$$
Taking the limit as $n\ra\infty$, we obtain
$$
\liminf_{n\ra\infty} \P\Big(\cok(M_{n\times (n+u)}) \simeq B \Big) \ge  \frac{1}{|B|^u|\Aut(B)|}\prod_{p\le k_0} \prod_{k=1}^\infty (1-p^{-k-u}) - O_{T,\eps}( \frac{1}{k_0}).
$$
As this is true for any fixed $k_0$, we can take $k_0\ra\infty$ to obtain,
$$ \liminf_{n \to \infty} \P\Big(\cok(M_{n\times (n+u)}) \simeq B\Big) \geq  \frac{1}{|B|^u|\Aut(B)|} \prod_{p} \prod_{k=1}^\infty (1-p^{-k-u}).$$
Since $\P(\cok(M_{n\times (n+u)}) \simeq B) \leq \P(\cok(M_{n\times (n+u)})_P \simeq B_P)$,  Equation~\eqref{eqn:compB} gives
$$\limsup_{n \to \infty} \P\Big(\cok(M_{n\times (n+u)}) \simeq B \Big) \leq \frac{1}{|B|^u|\Aut(B)|}\prod_{p\le k_0} \prod_{k=1}^\infty (1-p^{-k-u}),$$
completing the proof of \eqref{eqn:sur:fix}.
\end{proof}

Finally, to obtain Corollary~\ref{cor:anyset} and the $u\geq 1$ cases of Theorem~\ref{theorem:cyclic}, we only need the following simple observation.

\begin{lemma}\label{L:anyset}
Let $\mu$, and $\mu_n$ (for each positive integer $n$) be probability measures on a countable set $S$.  If for each $B\in S$,
$$
\lim_{n\ra\infty} \mu_n(B)=\mu(B),
$$
then for any subset $T\sub S$, we have $$
\lim_{n\ra\infty} \mu_n(T)=\mu(T).
$$
\end{lemma}
\begin{proof}
Let $T=\{B_1,\dots\}$.  Then
$$
\liminf_{n\ra\infty} \mu_n(T)=\liminf_{n\ra\infty} \sum_{k=1}^{\infty} \mu_n(B_i)\geq \sum_{k=1}^{\infty} \mu(B_i)=\mu(T),
$$
where the inequality is by Fatou's Lemma.  However, the same argument for the complement $\bar{T}$ of $T$ gives $\liminf_{n\ra\infty} \mu_n(T)\leq \mu(T)$.
\end{proof}

The proof of Theorem \ref{theorem:cyclic} is identical to the proof of Theorem~\ref{theorem:sur}, using Equations~\eqref{eqn:med} and \eqref{eqn:large:nprop} in place of Equations~\eqref{eqn:medu} and \eqref{eqn:large:n+1prop}, and the fact that $\cok(M_{n\times n})$ is cyclic if and only if for every prime $p$, the matrix $M_{n\times n}$ mod $p$ has rank at least $n- 1$.  In fact, the proof gives the following.

\begin{theorem}\label{T:prodcyc}  Let $M_{n \times n}$ be as in Theorem~\ref{theorem:sur}. 
Let $B$ be a finite abelian group and let $k_0$ be larger than any prime divisor of $|B|$, and define $C_B=\{B\times C\,|\, C \textrm{ cyclic, } p\nmid |C| \textrm{ for }1<p<k_0 \}$, the set of groups differing from $B$ by a cyclic group with order only divisible by primes at least $k_0$. 
 Then, we have
  $$\lim_{n\to \infty}  \P\Big(\cok(M_{ n \times n})\in C_B\Big)  =  \frac{1}{|\Aut(B)|}   \prod_{\substack{p<k_0\\p \textrm{ prime}}} (1-p^{-1}) \prod_{\substack{p\geq k_0\\p \textrm{ prime}}} (1+(p^2-p)^{-1})
  \prod_{k=2}^\infty \zeta(k)^{-1}.$$
 \end{theorem}

 Now we turn to the Laplacian, where we will follow an almost identical outline (corresponding to the case $u=1$ of our i.i.d. model $M_{n \times (n+u)}$). Indeed we will prove Theorem~\ref{theorem:sur:L} by checking if the Sylow-$p$ subgroup of $S_M$
is equal to $B_p$ for each prime $p$ (or is cyclic for each prime $p$) in three size ranges.  We prove the following proposition in Section~\ref{section:Laplacian}.

\begin{prop}\label{prop:L}
Let $M_{n \times n}$ and $\eps$ be as in Theorem~\ref{theorem:sur:L}. There are constants $c_0, d>0$ such that the following holds.  
\begin{itemize}
\item (Small Primes) Let $B$ be a finite abelian group.  Let $P$ be a finite set of primes including all those dividing $|B|$.
Then
\begin{equation}\label{eqn:small:L}
\lim_{n\ra\infty} \P\Big((S_{M_{n \times n}})_P \simeq B \Big) = \frac{1}{|B||\Aut(B)|}\prod_{p\in P} \prod_{k=1}^\infty (1-p^{-k-1}). 
\end{equation}
\item (Medium primes) Let $p$ be any prime. Then,
\begin{equation}\label{eqn:medu:L}
\P\Big(L_{M_{n \times n}} \mbox{ mod $p$ has rank $\leq n-2$} \Big)\leq 2p^{-2}+O(e^{-c_0\alpha_n n})
\end{equation}
and
\begin{equation}\label{eqn:med:L}
\P\Big(L_{M_{n \times n}} \mbox{ mod $p$ has rank $\leq n-3$} \Big)\leq 2p^{-6}+O(e^{-c_0\alpha_n n}).
\end{equation}
\item (Large primes) We also have
\begin{equation}\label{eqn:large:n:L}
\P\Big(\forall \mbox{ primes } p\geq e^{d \alpha_n n} : L_{M_{n \times n}} \mbox{ mod
  $p$ has full rank in $\Z_0^n/p$}\Big) \ge 1 -O_{d,T,\eps}( n^{-\eps}),
\end{equation}
as well as 
\begin{equation}\label{eqn:large:n-1:L}
\P\Big(\forall \mbox{ primes } p\geq e^{d \alpha_n n} : L_{M_{n \times (n-1)}} \mbox{ mod $p$ has rank at least $n-2$ }\Big) \ge 1 -O_{d,T,\eps}( n^{-\eps}).
\end{equation}
\end{itemize}
\end{prop}

The deduction of Theorem ~\ref{theorem:sur:L} from the above results is similar to the deduction of Theorem ~\ref{theorem:sur} and Theorem ~\ref{theorem:cyclic} from Propositions ~\ref{prop:small},~\ref{prop:medium} and ~\ref{prop:large}, and hence is omitted.

\section{Odlyzko's lemma}\label{section:O}

In this section we give an elementary but extremely useful tool which is a variant of Odlyzko's lemma  \cite{KKS} (also \cite[Lemma 2.2]{M1}). This result will be used in the arguments for small, medium, and large primes. We will  focus on the i.i.d case and refer the reader to Lemma ~\ref{lemma:O:L}  for a similar result regarding the Laplacian.

\begin{lemma}\label{lemma:O}
Let $\F$ be a field.
For a deterministic subspace $V$ of $\F^n$ of dimension $d$ and a random vector $X\in \F^n$ with i.i.d. entries taking any value with probability at most $1-\alpha_n$,
$$\P(X \in V) \le (1-\alpha_n)^{n-d}.$$
\end{lemma}
We give a short proof of this well-known result for completeness.
\begin{proof}
Assume that $V =\Sp(H_1,\dots, H_d)$, where $H_i=(h_{i1},\dots, h_{in})$, and without loss of generality we assume the matrix $(h_{ij})_{1\le i,j\le d}$ has rank $d$. 
Consider the event $X=(x_1,\dots,x_d, x_{d+1},\dots, x_n) \in V$. Because $(h_{ij})_{1\le i,j\le d}$ has rank $d$, there exist unique coefficients $c_1,\dots, c_d \in \F$ such that 
$$(x_1,\dots, x_d) = \sum_i c_i (h_{i1},\dots, h_{id}).$$
Hence conditioning on $(x_1,\dots, x_d)$, if $X=(x_1,\dots,x_d, x_{d+1},\dots, x_n) \in V$ then for all $d+1\le j\le n$
$$x_j = \sum_i c_i h_{ij}.$$
However the probability of each of these events is at most $1-\alpha_n$, and so conditioning on $(x_1,\dots, x_k)$, the event $X \in V$  holds with probability at most $(1-\alpha_n)^{n-d}$.
\end{proof}

\begin{corollary}\label{cor:fullrank} Let $X_1,\dots, X_{n-k}$ be random vectors with i.i.d. entries taking any value with probability at most $1-\alpha_n$. Then the probability that $X_1,\dots, X_{n-k}$
 are linearly independent in $\F^n$ is at least $1 -\a_n^{-1}(1-\a_n)^{k}.$ 
\end{corollary}
\begin{proof}Let $0\le i\le n-k-1$ be minimal such that $X_{i+1} \in span(X_1,\dots,X_i)$. By Lemma ~\ref{lemma:O}, this event is bounded by $(1-\alpha_n)^{n-i}$. Summing over $0\le i\le n-k-1$, the probability under consideration is bounded by $\sum_{i=0}^{n-k-1} (1-\alpha_n)^{n-i} < \a_n^{-1}(1-\alpha_n)^{k}$. 
\end{proof}


In all three arguments, Lemma~\ref{lemma:O} (Odlyzko's lemma) will only suffice for the easy part of the argument, and a stronger, Littlewood-Offord style bound (Lemma~\ref{L:FullFcode}, Theorem~\ref{thm:A1}, Theorem~\ref{theorem:LO}, Theorem~\ref{theorem:ILO}) will be required for the harder part of the argument.  The details of the Littlewood-Offord style bound required are different in each argument, and thus are given in the corresponding sections.  Note that Odlyzko's lemma is too weak to be used alone for our purposes, because it can produce a bound $1-\alpha_n,$  where we require bounds that go to $0$ as $n\ra\infty$.
In this paper, $\alpha_n$ is possibly small.  If, however, the matrix entries take values modulo large primes with probability at most $1-\alpha_n$, and $1-\alpha_n\ra 0$ as $n\ra \infty$, then we expect our arguments can all be considerably simplified and only Odlyzko's lemma would be necessary (and no Littlewood-Offord style bounds  required). 
 For example, such a simplification works to handle the case of entries chosen uniformly in a interval centered at $0$ with size growing at any rate with $n$.

\section{Small Primes}\label{sec:small}
In this section, we prove the following theorem, which generalizes \cite[Corollary 3.4]{W1} to smaller $\alpha_n$ and implies our Proposition~\ref{prop:small}.
The method requires  refinement from that of \cite{W1}, and we discuss the differences below.
 \begin{theorem}\label{th:allsmall}
Let $u$ be a non-negative integer and $\alpha_n$ a function of integers $n$ such that for any constant $\Delta>0$, for $n$ sufficiently large we have
$\alpha_n\geq \Delta \log n/n$.  For every positive integer $n$,
let $M(n)$  be a random matrix valued in $M_{n\times (n+u)}(\Z)$ with independent $\alpha_n$-balanced entries.
Let $B$ be a finite abelian group.  Let $P$ be a finite set of primes including all those dividing $|B|$.
  Then
\begin{align*}
\lim_{n\ra\infty} \P(\cok(M(n))_P\isom B) 
&= \frac{1}{|B|^u|\Aut(B)|}
\prod_{p\in P}
\prod_{k=1}^\infty (1-p^{-k-u}).
\end{align*}
\end{theorem}
Note that the entries of the matrix do not have to be identical.

Throughout the section we write $\Hom(A,B)$ and  $\Sur(A,B)$ for the set of homomorphisms and surjective homomorphisms, respectively, from $A$ to $B$.  We will always use $a$ to denote a positive integer and $R=\Z/a\Z$.  We then study finite abelian groups $G$ whose exponent divides $a$, i.e. $aG=0$.  We write $G^*$ for $\Hom(G,R)$.

\subsection{Set-up}\label{S:mom}
We will study integral matrices by reducing them mod each positive integer.  We let $a$ be a positive integer.
Let $M$ be the random $n\times (n+u)$ matrix with entries in $R$ that is the reduction of $M(n)$ from Theorem~\ref{th:allsmall} modulo $a$.  
We let $X_1,\dots, X_{n+u}\in R^n$ be the columns of $M$, and $x_{ij}$ the entries of $M$ (so that the entries of $X_j$ are $x_{ij}$). 
For a positive integer $n$, we let $V=R^n$ with basis $v_i$ and $W=R^{n+u}$ with basis $w_j$ (these will always implicitly depend on the integers we call $a$ and $n$). 
Note for $\sigma\sub [n]$, $V$ has distinguished submodules $V_{\setminus \sigma}$ generated by the $v_i$ with $i\not\in\sigma$. (So $V_{\setminus \sigma}$ comes from not using the $\sigma$ coordinates.)
 We view $M\in \Hom(W,V)$ and its columns $X_j$ as elements of $V$ so that
$X_j=Mw_j=\sum_i x_{ij}v_i. $  Let $G$ be a finite abelian group with exponent dividing $a$.
We have $\cok (M)=V/MW$.

We know from \cite{W1} that to understand the distribution of $\cok (M)$, it suffices to determine certain moments.
To investigate the moments 
$
\E(\#\Sur(\cok (M), G))
$ (see \cite[Section 3.3]{CLKPW} for more on why these are ``moments''),
we recognize that each such surjection lifts to a surjection $V\ra G$ and so we have
\begin{equation}\label{E:expandF}
\E(\#\Sur(\cok (M), G))=\sum_{F\in \Sur(V,G)}  \P(F(MW)=0).
\end{equation}
By the independence of columns, we have
$$
\P(F(MW)=0)=\prod_{j=1}^m \P(F(X_j)=0).
$$
So we aim to estimate these probabilities $\P(F(X_j)=0)$, which will give us our desired moments.

\subsection{Finding the moments}

We will first estimate $\P(F(X_j)=0)$ for the vast majority of $F$, which satisfy the following helpful property.  

\begin{definition}
We say that $F\in \Hom(V,G)$ is a \emph{code} of distance $w$, if for every $\sigma\sub [n]$ with $|\sigma|<w$, we have $FV_{\setminus \sigma}=G$.
In other words, $F$ is not only surjective, but would still be surjective if we throw out (any) fewer than $w$ of the standard basis vectors from $V$.  (If $a$ is prime so that $R$ is a field, then this is equivalent to whether the transpose map $F: G^* \ra V^*$ is injective and has image $\im(F)\sub V^*$ a linear code of distance $w$, in the usual sense.)
\end{definition}

First we recall a lemma from \cite{W1} that lets us see how a code $F$ acts on a single column from our matrix.  The following statement is slightly stronger than \cite[Lemma 2.1]{W1}, but one can see this statement follows directly from the proof of \cite[Lemma 2.1]{W1}.

\begin{lemma}\label{L:Fcodecolumn}
Let $a,n$ be positive integers, $G$ a finite abelian group of exponent dividing $a$, and $X$ the reduction mod $a$ of a random vector in $\Z^n$ with independent, $\alpha$-balanced entries.  
 Let $F\in \Hom(V,G)$ be a code of distance $w$ and $A\in G$.
We have
$$
\left |  \P(FX=A) - |G|^{-1} \right| \leq \frac{|G|-1}{|G|}  \exp(-\alpha w/a^2).
$$
\end{lemma}

We then will put these estimates for columns together using this simple inequality.

\begin{lemma}[{\cite[Lemma 2.3]{W1}}]\label{L:1tom}
If we have an integer $m\geq 2$ and  real numbers $x\geq 0$ and $y$ such that $|y|/x\leq 2^{1/(m-1)}-1$ and $x+y\geq 0$, then
$$
x^m-2mx^{m-1}|y| \leq (x+y)^m\leq x^m+2mx^{m-1}|y|.
$$
\end{lemma}

The below is a refinement of \cite[Lemma 2.4]{W1} that allows sparse matrices.


\begin{lemma}[Bound for codes]\label{L:FullFcode}
Let $a\geq 1$ and  $u\geq 0$ be  integers, $G$ be a finite abelian group of exponent dividing $a$, the sequence $\{\alpha_n\}_n$ be as in Theorem~\ref{th:allsmall}, and $\delta>0$. 
Then there are $c_1,K_1>0$ such that the following holds.  
 Let  $\bar{M}(n)$ be the reduction modulo $a$ of random matrices $M(n)$ as in Theorem~\ref{th:allsmall}.
For every positive integer $n$, and $F\in \Hom(V,G)$ a code of distance $\delta n$, and  $A\in\Hom(W,G)$, we have
\begin{align*}
 \left|\P(F\bar{M}(n)=A) - |G|^{-n-u}\right| &\leq
 \frac{K_1n^{-c_1}}{|G|^{n+u}}.
\end{align*}
\end{lemma}

\begin{proof}
Choose $\Delta>a^2\delta^{-1}$ and $n$ large enough (depending on $\Delta$ and $\{\alpha_i\}_i$) so that $\alpha_n\geq \Delta\log n/n$.
Then for $n$ large enough given $\delta,\Delta,u,a,|G|$, we have
$$
 \exp(-(\Delta \log n/n) \delta n/a^2)|G|=
 \exp(-\Delta \delta \log n  /a^2)|G|
\leq \frac{\log 2}{n+u-1} \leq 2^{1/(n+u-1)}-1.
$$
So for such $n$ we can combine Lemma~\ref{L:Fcodecolumn} and Lemma~\ref{L:1tom} to obtain
$$
\left |  \P(FM=A) - |G|^{-n-u} \right| \leq 2(n+u)\exp(-\Delta \delta \log n  /a^2)|G|^{-n-u+1}.
$$
We take $c_1 < \Delta \delta   /a^2 -1$ and then for $n$ sufficiently large given $u,\Delta,\delta,a,c_1,|G|,\{\alpha_i\}_i$, we have
$$
\left |  \P(FM=A) - |G|^{-n-u} \right| \leq \frac{2(n+u)}{n^{1+(\Delta \delta   /a^2 -1)}}|G|^{-n-u+1}
\leq n^{-c_1} |G|^{-n-u}.
$$
We choose $K_1$ large enough so that  $\frac{K_1n^{-c_1}}{|G|^{n+u}}\geq 2$ for $n$ that are not as large as needed above, and the lemma follows.
\end{proof}

So far, we have a good estimate for $\P(FM=0)$ when $F$ is a code.  Unfortunately, it is not sufficient to divide $F$ into codes and non-codes.  We need a more delicate division of $F$ based on the subgroups of $G$.  
In \cite{W1}, a notion of \emph{depth} was used to divide the $F$ into classes.  Here we require a slightly finer notion (that we call \emph{robustness}) to deal with the sparser matrices.
Both notions can be approximately understood as separating the $F$ based on what largest size subgroup they are a code for.
For an integer $D$ with prime factorization $\prod_i p_i^{e_i}$, let $\ell(D)=\sum_i e_i$.

\begin{definition}
Given $\delta>0$, we say that $F\in \Hom(V,G)$ is \emph{robust} (or, more precisely, $\delta$-\emph{robust}) for a subgroup $H$ of $G$ if $H$ is minimal such that
$$
\#\{i\in[n] | Fv_i\not\in H \}\leq \ell([G:H])\delta n.
$$
Note that $H=G$ satisfies the above inequality, so every $F\in \Hom(V,G)$ is robust for some subgroup $H$ of $G$.  An $F$ might be robust for more than one subgroup.
\end{definition}

\begin{lemma}\label{L:robustcodes}
Let $\delta>0,$ and $a,n$ be positive integers, and $G$ be a finite abelian group of exponent dividing $a$.
Let $F\in \Hom(V,G)$  be robust for $H$.  Let $\pi:=\{i\in[n] | Fv_i\not\in H \}$.
Then $F$ restricted to $V_{\setminus \pi}$ is a code of distance $\delta n$ in $\Hom (V_{\setminus \pi},H)$.
\end{lemma}
\begin{proof}
Suppose not.  Then there exists a $\sigma\sub [n]\setminus \pi$ such that
$|\sigma|<\delta n$ and $FV_{\setminus (\pi\cup \sigma)}$ lies in some proper subgroup $H'$ of $H$.
In particular, the set of $i$ such that $Fv_i\not\in H'$ is contained in $\pi\cup \sigma$.
Since
$$
|\pi\cup \sigma|\leq \ell([G:H])\delta n +\delta n \leq \ell([G:H'])\delta n,
$$
we then have a contradiction on the minimality of $H$.
\end{proof}

We then bound the number of $F$ that are robust for a certain group $H$, and with certain given behavior outside of $H$.  
The separation of $F$ into classes based on their behavior outside of $H$ did not appear in \cite{W1}, but is necessary here to deal with sparser matrices.

\begin{lemma}[Count of robust $F$ for a subgroup $H$]\label{L:newcount}
Let $\delta>0$, and $a,n\geq 1$  be integers, and $G$ be  finite abelian group of exponent dividing $a$.
Let $H$ be a subgroup of $G$ of index $D>1$ and 
let $H=G_{\ell(D)}\sub \dots \sub G_2\sub G_1\sub G_0=G$ be a maximal chain of proper subgroups.
Let $p_j=|G_{j-1}/G_j|$.
The number of $F\in\Hom(V,G)$ such that $F$ is robust for $H$ 
and for $1\leq j \leq \ell(D)$,  there are $w_j$ elements $i$ of $[n]$ such that $Fv_i \in G_{j-1}\setminus G_j$, is at most
$$
|H|^{n-\sum_j w_j} \prod_{j=1}^{\ell(D)} \binom{n}{w_j}|G_{j-1}|^{w_j} .
$$
\end{lemma}
Note that by the definition of robustness we have that
$w_j\leq \ell([G:H])\delta n,$ or else there are no such $F$.
\begin{proof}
There are at most $\binom{n}{w_j}$ ways to choose the $i$ such that $Fv_i \in G_{j-1}\setminus G_j $ and then
at most $|G_j|^{w_j}$ ways to choose the $Fv_i$.  Then there are $|H|$ choices for each remaining $Fv_i$. 
\end{proof}

Now for $F$ robust for a subgroup $H$, we will get a bound on $\P(FM=0)$, where the larger the $H$, the better the bound.
This is a more delicate bound than \cite[Lemma 2.7]{W1} that it is replacing, and in particular takes into account the behavior of $F$ outside of $H$.

\begin{lemma}[Probability bound for columns given robustness]\label{L:probdepthestimatecolumnNEW}
Let  $\delta>0$, and $a,n\geq 1$ be integers, and $G$ be  finite abelian group of exponent dividing $a$. 
 Let $F\in\Hom(V,G)$ 
be robust for a proper subgroup $H$ of $G$ and let $D:=[G:H]$.
Let $H=G_{\ell(D)}\sub \dots \sub G_2\sub G_1\sub G_0=G$ be a maximal chain of proper subgroups.
Let $p_j=|G_{j-1}/G_j|$.
For $1\leq j \leq \ell(D)$,  let $w_j$ be the number of $i\in[n]$ such that $Fv_i \in G_{j-1}\setminus G_j $.
Let $X\in R^n$ be a a random vector with independent entries that are the reduction mod $a$ of $\alpha$-balanced random integers.  Then for all $n$,
$$\P(FX=0)\leq \Big( D|G|^{-1} +\exp(-\alpha \delta n/a^2) \Big)
\prod_{j=1}^{\ell(D)} 
\Big(p_j^{-1} + \frac{p_j-1}{p_j} \exp(-\alpha w_j /p_j^2)\Big)
.$$
\end{lemma}

\begin{proof} Assume that $X=(x_1,\dots, x_n)$. Let $\sigma_j$ be the collection of indices  $i\in [n]$ such that $Fv_i \in G_{j-1}\setminus G_j$.
Let $\sigma=\cup_{j=1}^{\ell(D)} \sigma_j$.
Then,
\begin{align*}
\P(FX=0) = &\P\Big(\sum_{i\in\sigma_1} (Fv_i)x_i \in G_1\Big)\P\Big(\sum_{i\in\sigma_1 \cup \sigma_2} (Fv_i)x_i \in G_2 \Big| \sum_{i\in\sigma_1} (Fv_i)x_i \in G_1\Big)\times \cdots \\
& \times \P\Big(\sum_{i\in\sigma_1 \cup \cdots \cup \sigma_{\ell(D)}} (Fv_i)x_i \in H \Big| \sum_{i\in\sigma_1 \cup \cdots \cup \sigma_{\ell(D)-1}} (Fv_i)x_i \in G_{\ell(D)-1}\Big)\\
& \times \P\Big(\sum_{i\not\in \sigma} (Fv_i)x_i=-\sum_{i\in\sigma} (Fv_i)x_i \Big| \sum_{i\in\sigma} (Fv_i)x_i \in H  \Big).
\end{align*}
For $1\leq j\leq \ell(D)$, we will bound the $j$th factor above by conditioning on the $x_i$ with $i\in \sigma_1\cup \cdots \cup
\sigma_{j-1}$ and then looking at images in $G_{j-1}/G_j$.
Note for $i\in \sigma_j$, we have that the reduction of $Fv_i$ is non-zero in $G_{j-1}/G_j$.  So $F$ restricted to the $\sigma_j$ coordinates in the reduction to $G_{j-1}/G_j$ is a code of length $w_j$.  We then apply Lemma~\ref{L:Fcodecolumn} to this case to obtain
\begin{align*}
& \P\Big (\sum_{i\in\sigma_1 \cup \cdots \cup \sigma_{j}} (Fv_i)x_i \in G_j \Big| \sum_{i\in\sigma_1 \cup \cdots \cup \sigma_{j-1}} (Fv_i)x_i \in G_{j-1}\Big) \leq  p_j^{-1} + \frac{p_j-1}{p_j}\exp(-\alpha w_j/ p_j^2 ). 
\end{align*}

Note that $\sigma$ is the set of $i$ such that $Fv_i \not \in H$. 
By the definition of robust, $|\sigma|< \ell(D)\delta n$.
By Lemma~\ref{L:robustcodes}, the restriction of $F$ to $V_{\setminus\sigma}$
 is a code of distance $\delta n$ in $\Hom(V_{\setminus\sigma}, H)$.
 So conditioning on the $X_i$ with $i\in\sigma$, we can estimate the conditional probability above using Lemma~\ref{L:Fcodecolumn}:
$$
\P\Big(\sum_{i\not\in\sigma} (Fv_i)x_i =- \sum_{i\in\sigma} (Fv_i)x_i\Big |\sum_{i\in\sigma} (Fv_i)x_i \in H \Big)
\leq |H|^{-1} +\exp(-\alpha \delta n/a^2).  
$$
The lemma follows.
\end{proof}

Now we can combine the estimates we have for $\P(FM=0)$ for various types of $F$ with the bounds we have on the number of $F$ of each type to obtain our main result on the moments of cokernels of random matrices.  



\begin{theorem}\label{T:MomMat}
Let $u\geq0$ be an integer,  $G$ be a finite abelian group, and the sequence $\{\alpha_n\}_n$ be as in Theorem~\ref{th:allsmall}.
Then there are $c_2,K_2$ such that the following holds.  
For every positive integer $n$ and random matrix $M(n)$ as in Theorem~\ref{th:allsmall}, we have
\begin{align*}
\left|\E(\#\Sur(\cok(M(n)),G))  -|G|^{-u} \right| \leq K_2n^{-c_2}.
\end{align*}
\end{theorem}
\begin{proof}
Let $a$ be the exponent of $G$.
By Equation~\eqref{E:expandF}, we need to estimate
$
\sum_{F\in \Sur(V,G)} \P(FM(n)=0).
$
Fix a proper subgroup $H$ of $G$.  
We will apply Lemma~\ref{L:probdepthestimatecolumnNEW} and use the notation from that lemma, along with Lemma~\ref{L:newcount}.  We then have
\begin{align*}
&\sum_{\substack{F\in \Sur(V,G)\\\textrm{$F$ is robust for $H$}}}
\P(FM(n)=0) \\\leq & \sum_{\substack{0\leq w_1,\dots, w_{\ell(D)} \leq \ell(D) \delta n\\w_1\neq 0}} |H|^{n-\sum_j w_j} \prod_{j=1}^{\ell(D)} \binom{n}{w_j}|G_{j-1}|^{w_j} 
\prod_{j=1}^{\ell(D)} 
\Big(p_j^{-1} + \frac{p_j-1}{p_j} \exp(-\alpha_n w_j /p_j^2) \Big)^{n+u}\\
&\times \Big( D|G|^{-1} +\exp(-\alpha_n \delta n/a^2) \Big)^{n+u}\\
= & |H|^n \Big( D|G|^{-1} +\exp(-\alpha_n \delta n/a^2) \Big)^{n+u} 
\prod_{j=1}^{\ell(D)} \sum_{\substack{w_j=0\\w_1\ne 0}}^{\ell(D) \delta n} |H|^{-w_j}  \binom{n}{w_j}|G_{j-1}|^{w_j}  \Big(p_j^{-1} + \frac{p_j-1}{p_j} \exp(-\alpha_n w_j /p_j^2) \Big)^{n+u}.
\end{align*}
We have $w_1\ne 1$ since $F$ is a surjection.
Now we apply Lemma~\ref{L:dealwithonesum} from the Appendix to bound the sums.
The $\Ao,\ao$ from  Lemma~\ref{L:dealwithonesum}, will be $|G_{j-1}|/|H|$ and $p_j^{-1}$ respectively.  
We choose the $\Delta'$ of Lemma~\ref{L:dealwithonesum} so that $\Delta'>2/(1-p_j^{-1})$ for all $j$.
For $n$ sufficiently large (in terms of $\{\alpha_i\}_i,  \Delta',G$), we have $\alpha_n\geq p_j^2 \Delta' \log n/n$ for all $j$.
Lemma~\ref{L:dealwithonesum} then gives us that, for $\delta$ sufficiently small (given $G$), and $n$ sufficiently large (given $G$, $\Delta'$, $\delta$, $\{\alpha_i\}_i,$), we have
\begin{align*}
\sum_{w_j=1}^{\ell(D) \delta n}   \binom{n}{w_j}\left(\frac{|G_{j-1}|}{|H|}\right)^{w_j}  \Big(p_j^{-1} + \frac{p_j-1}{p_j} \exp(-\alpha_n w_j /p_j^2) \Big)^{n+u}
\leq 3n^{-((1-p_j^{-1})\Delta'/2-1)}.
\end{align*}

Let $\Delta>a^2\delta^{-1}$ and $\Delta>2p^3/(p-1)$ for every prime $p\mid a$.
For $n$ also sufficiently large (given $\Delta$ and $\{\alpha_i\}_i$) that $\alpha_n\geq \Delta \log n/n$, and we have
\begin{align*}
|H|^n \Big( D|G|^{-1} +\exp(-\alpha_n \delta n/a^2) \Big)^{n+u}  \leq & |H|^{-u} \Big( 1 +|H| \exp(-\Delta \delta \log n/a^2) \Big)^{n+u}.
\end{align*}

For $n+u\geq 2$, and $n$ sufficiently large (given $\delta,\Delta,u,G$) such that
$$
|H| \exp(-\Delta \delta \log n/a^2)= |H|n^{-\Delta \delta/a^2} \leq \frac{\log 2}{n+u-1} \leq 2^{1/(n+u-1)}-1.
$$
By Lemma~\ref{L:1tom},
$$
\left( 1 +|H| \exp(-\Delta \delta \log n/a^2) \right)^{n+u} \leq 1+ 2({n+u})|H| \exp(-\Delta \delta \log n/a^2).
$$
Putting it altogether we have
\begin{align*}
&\sum_{\substack{F\in \Sur(V,G)\\\textrm{$F$ is robust for $H$}}}
\P(FX=0) \\ \leq  & |H|^n \Big( D|G|^{-1} +\exp(-\alpha_n \delta n/a^2) \Big)^{n+u} 
\prod_{j=1}^{\ell(D)} \sum_{\substack{w_j=0\\w_1\ne 0}}^{\ell(D) \delta n} |H|^{-w_j}  \binom{n}{w_j}|G_{j-1}|^{w_j}  \Big(p_j^{-1} + \frac{p_j-1}{p_j} \exp(-\alpha_n w_j /p_j^2) \Big)^{n+u}
\\ \leq  & |H|^{-u} \Big( 1+ 2({n+u})|H| \exp(-\Delta \delta \log n/a^2) \Big)
3n^{-((1-p_1^{-1})\Delta'/2-1)}
\prod_{j=2}^{\ell(D)} \left(1+ 3n^{-((1-p_j^{-1})\Delta'/2-1)} \right).
\end{align*}

We sum this over proper subgroups $H$ of $G$ to bound,
for $\delta$ sufficiently small (given $G$), and $\Delta'>2/(1-p_j^{-1})$ for all $j$, and $\Delta>a^2\delta^{-1}$, and $\Delta>2p^3/(p-1)$ for every prime $p\mid a$,
for $n$ sufficiently large given $G$, $\delta$, $\Delta$, $\Delta'$, $u$, $\{\alpha_i\}_i,$
\begin{align*}
\sum_{\substack{F\in \Sur(V,G)\\ F\textrm{ not  code of distance $\delta n$}
} }
\P(FM(n)=0)\leq K_3n^{-c_3}.
\end{align*}
where $K_3$ is a constant depending on $G$, $\delta$, $\Delta$, $\Delta'$, $u$ and $c_3>0$ (depending on $a$, $\Delta'$).  

Also, from the proof of \cite[Theorem 2.9]{W1},  we can choose $\delta$ small enough (given $G$) so that we have for all $n$
\begin{align*}
\sum_{\substack{F\in \Sur(V,G)\\ F\textrm{ not  code of distance $\delta n$}
} }
|G|^{-n-u} &\leq  K_4 
 e^{-(\log 1.5)n}
\end{align*}
for some $K_4$ depending on $u,G,\delta$.  
We also have (e.g. see the proof of \cite[Theorem 2.9]{W1}) for all $n$,
 \begin{align*}
\sum_{\substack{F\in \Hom(V,G)\setminus \Sur(V,G)
} }
|G|^{-n-u}&\leq  K_5 
 e^{-\log(2)n}
\end{align*}
for some $K_5$ depending on $G$.  
Using Lemma~\ref{L:FullFcode} we have that for all $n$,
\begin{align*}
\sum_{\substack{F\in \Sur(V,G)\\ F\textrm{  code of distance $\delta n$}
} }
\left| \P(FX=0) -|G|^{-n-u}\right|
&\leq 
 K_1n^{-c_1} . 
\end{align*}
We now make a choice of $\delta$ that is sufficiently small for the two requirements above (given $G$), and we choose $\Delta$ and $\Delta'$ as required above, so that
for all $n$ sufficiently large (given $G$, $\delta$, $\Delta$, $\Delta'$, $u$, $\{\alpha_i\}_i,$)
\begin{align*}
&\Big| \sum_{F\in \Sur(V,G)} \P(FX=0)  -|G|^{-u} \Big| =\Big| \sum_{F\in \Sur(V,G)} \P(FX=0)  -\sum_{F\in \Hom(V,G)} |G|^{-n-u}   \Big|  \\
&\leq
\sum_{\substack{F\in \Sur(V,G)\\ F\textrm{  code of distance $\delta n$}
} }
\Big| \P(FX=0) -|G|^{-n-u}\Big|+
 \sum_{\substack{F\in \Sur(V,G)\\ F\textrm{ not  code of distance $\delta n$}
} }
 \P(FX=0)  + \sum_{\substack{F\in \Hom(V,G)\\ F\textrm{ not  code of distance $\delta n$}
} } |G|^{-n-u} \\
\\
&\leq  K_1n^{-c_1}+K_3n^{-c_3}+K_4e^{-(\log 1.5)n}+K_5 
 e^{-\log(2)n}.
\end{align*}
We choose $c_2\leq \min(c_1,c_3,\log(1.5))$ (which depends on $G,u, \{\alpha_i\}_i$).  We choose $K_2$ so that $K_2\geq K_1+K_3+K_4+K_5$, and also $K_2\geq |G|^n n^{c_2}$ for any $n$ not sufficiently large for the requirements above (so $K_2$ depends on $G,u, \{\alpha_i\}_i$).  The theorem follows.
\end{proof}



We now conclude the proof of Theorem~\ref{th:allsmall}. 
For each fixed $u\ge 1$, we construct a random abelian group according to Cohen and Lenstra's distribution mentioned in the introduction. Independently for each $p$, we have a random finite abelian $p$-group $Y_p$ such that for each $p$-group $B$
$$\P(Y_p = B)= \frac{\prod_{k=1}^\infty (1-p^{-k-u})}{|B|^u |\Aut(B)|}).$$
Let $P$ be a set of primes dividing a given number $a$, we then define a random group $Y$ by taking the group product $\prod_{p\in P} Y_p$. 
\begin{lemma}[{\cite[Lemma 3.2]{W1}}] For every finite abelian group $G$ with exponent dividing $a$ we have 
$$\E (\#\Sur(Y,G))= |G|^{-u}.$$
\end{lemma}
From Theorem~\ref{T:MomMat}, we have seen that $Y$ and $\cok(M(n))$ have asymptotic matching ``moments" with respect to all groups $G$ of exponent dividing $a$. To pass this information back to distribution, we then use the following.

\begin{theorem}[{\cite[Theorem 3.1]{W1}}]\label{th:distribution}Let $X_n$ and $Y_n$ be sequences of random finitely generated abelian groups. Let $a$ be a positive integer and $\CA$ be the set of isomorphism classes of abelian groups with exponent dividing $a$. Suppose that for every $G\in \CA$ we have a number $M_G \le |\wedge^2 G|$ such that $\lim_{n \to \infty} \E( \#\Sur(X_n,G)) = \lim_{n \to \infty}  (\#\Sur(Y_n,G)) = M_G$. Then we have that for every $H \in \CA$
$$\lim_{n \to \infty} \P\big(X_n \otimes (\Z/a\Z) \isom H\big) =\lim_{n \to \infty} \P(Y_n \otimes (\Z/a\Z) \isom H).$$
\end{theorem}

To prove Theorem~\ref{th:allsmall}, assume that the exponent of the group $B$ under consideration has prime factorization $\prod_{p\in P} p^{e_p}$. Theorem \ref{th:distribution}, applied to the sequence $X_n=\cok(M(n))$ and $Y_n=Y$ with $a= \prod_{p\in P} p^{e_p+1}$, implies that  
$$\lim_{n\to \infty} \P\Big(\cok(M(n)) \otimes (\Z/a\Z) \isom B\Big) = \P(Y \otimes (\Z/a\Z) \isom B) = \frac{1}{|B|^u|\Aut(B)|}
\prod_{p\in P}
\prod_{k=1}^\infty (1-p^{-k-u}).$$
The proof is then complete because $\cok(M(n)) \otimes (\Z/a\Z) \isom B$ if and only if $\cok(M(n))_P \isom B$.

\section{Medium Primes}\label{section:mediumprimes}

In this section we prove the following, which we apply to medium primes for the proof of our main results.
\begin{theorem}\label{th:allmed}
There are constants $c_0, \eta,C_0,K_0>0$  such that we have the following.
Let $n,u\geq 0$ be integers, $p$ be a prime, and let $M_{n\times (n+u)}$ be a random matrix $n\times (n+u)$ with independent i.i.d. entries $\xi_n\in \Z/p\Z$.
We further assume we have a real number $\alpha_n$ such that
$$
\max_{r\in\Z/p\Z} \P(\xi_n=r)=1-\alpha_n \leq 1-\frac{C_0 \log n}{n}.
$$
Then we have
$$
\P\Big( \rank (M_{n\times (n+u)})\leq n-1\Big)\leq 2p^{-\min(u+1,\eta n-1)} +K_0e^{-c_0\alpha_n n}
$$
and
$$
\P\Big( \rank (M_{n\times (n+u)})\leq n-2 \Big)\leq 2p^{-\min(2u+4,\eta n-1)} +K_0e^{-c_0\alpha_n n}.
$$
\end{theorem}

The proof of Theorem~\ref{th:allmed} has two main ingredients.  First, we have a result from \cite{NP} that says that the first $n-k$ columns of $M_{n\times(n+u)}$ are likely to generate a subspace $V$ such that the probability of the next column being in $V$ is near to the probability of a uniform random column mod $p$ being in $V$.  (This result was originally stated in \cite{M1} by Maples, but \cite{NP} gives a corrected proof using the ideas of \cite{M1} and \cite{TV}.)

\begin{theorem}[{\cite[Theorems A.1 and A.4]{NP}}]\label{thm:A1} There are constants $c, \eta, C_0,K>0$ such that the following holds.  
Let $n,u\geq 0$ be integers with $u\le \eta n$, $p$ be a prime, and let $M_{n\times (n+u)}$ be a random matrix $n\times (n+u)$ with independent i.i.d. entries $\xi_n\in \Z/p\Z$.
We further assume we have a real number $\alpha_n$ such that
$$
\max_{r\in\Z/p\Z} \P(\xi_n=r)=1-\alpha_n \leq 1-\frac{C_0 \log n}{n}.
$$
For   $-u\leq k\le \eta n$, let $X_{n-k+1}$ be the $(n-k+1)$st column of $M_{n\times (n+u)}$, and 
 $W_{n-k}$ be the subspace  by the first $n-k$ columns of $M_{n\times (n+u)}$.
Then there is an event $\CE_{n-k}$ on the $\sigma$-algebra generated by the first $n-k$ columns of $M_{n\times (n+u)}$, of probability at least $1- 3e^{-c \alpha_n n}$, such that for any $k_0$ with $\max (0,k)\leq k_0 \le \eta n$
$$
\left |\P\Big(X_{n-k+1} \in W_{n-k} | \CE_{n-k} \wedge \codim(W_{n-k})=k_0\Big) - p^{-k_0} \right| \leq K e^{-c \alpha_n n}.$$
\end{theorem}




 We also refer the reader to Theorem~\ref{theorem:corank:L} for a similar statement for the Laplacian with a complete proof. Note that for a uniform random $X\in(\Z/p\Z)^n$, we have $\P(X\in V)=p^{-\codim(V)}.$  Thus, as long as we avoid certain rare bad events, as we consider more and more columns of our random matrices, the probability that the next column is in the span of the previous columns is close to what it would be if we were using uniform random matrices.  The following result, proven in Section~\ref{sec:Gdel} in the Appendix, allows us to use that information to conclude that the rank distribution of our matrices is close to that of uniform random matrices.  This theorem says that if sequences of random variables $x_i$ and $y_i$ have similar transition probabilities going from $x_i$ to $x_{i+1}$ and $y_i$ to $y_{i+1}$, at least under conditions that are likely to be true, then the distributions of $x_n$ and $y_n$ must stay close. 

\begin{theorem}\label{T:allerror}
Let $x_1,\dots,x_n,g_0,\dots,g_{n-1}$ be a sequence of random variables, and let $x_0=0$.
  Let $y_1,\dots,y_n$ be a sequence of  random variables, and let $y_0=0$.   We assume each $x_i,y_i$ takes on at most countably many values, and $g_i\in \{0,1\}$.
  Suppose that  for $0\leq i \leq n-1$, 
  \begin{align*}
&\P(y_{i+1}=s|y_i=r)=\P(x_{i+1}=s |x_i=r\textrm{ and }g_i=1) +\delta(i,r,s) \\
&\textrm{ for all $r$ and $s$
s.t. $\P(y_i=r)\P(x_i=r\textrm{ and }g_i=1)\ne0$}. 
\end{align*}
Then for all $n\geq 0$ and any set $A$ of values taken by $x_n$ and $y_n$, we have
\begin{align*}
&|\P(x_n\in A)-\P(y_n\in A)|\\
&\leq \frac{1}{2}\sum_{i=0}^{n-1} \sum_{r} \sum_s |\delta(i,r,s)| \P(x_i =r)  +
\sum_{i=0}^{n-1} \Pr(g_i\ne 1),
\end{align*}
where $r$ is summed over $\{ r\ |\ \P(x_{i}=r)\ne 0 \textrm{ and } \P(y_{i}=r)\ne 0)\}$ and
$s$ is summed over $\{ s\ |\ \P(x_{i+1}=s)\ne 0 \textrm{ or } \P(y_{i+1}=s)\ne 0)\}.$
\end{theorem}

We remark that our error bounds come from the $\delta$'s and the complement of $g_i=1$. The explicit form here will be extremely useful because in the sparse case $\delta$ and $\P(g_i \neq 0)$ are not small. 

\begin{proof}[Proof of Theorem~\ref{th:allmed}]
We take $\eta, C_0$ as in Theorem~\ref{thm:A1}. 
Since $$\P( \rank (M_{n\times (n+u+1)})\leq  m)\leq \P( \rank (M_{n\times (n+u)})\leq  m),$$ it suffices to prove the theorem for $u\leq \lfloor \eta n  \rfloor-1$. Let $X_m$ be the $m$-th column of $M_{n\times (n+u)}$, and $W_m$ the subspace generated by $X_1,\dots,X_m$.
For $1\leq i \leq \lfloor \eta n \rfloor+u$, define the random variable
$$
x_i=
\begin{cases}
k_0 & \textrm{if } \rk (W_{n-\lfloor \eta n \rfloor+i})=n-\lfloor \eta n \rfloor+i-k_0 \textrm{ and } 0\leq k_0 \leq u+1\\
* & \textrm{if } \rk (W_{n-\lfloor \eta n \rfloor+i})\leq n-\lfloor \eta n \rfloor+i-u-2.
\end{cases}
$$
In other words, $x_i$ measures the deficiency $(n-\lfloor \eta n \rfloor+i) - \rk (W_{n-\lfloor \eta n \rfloor+i})$ if this difference is not larger than $u+1$.
 
Let $y_i$ be analogous function for a uniform random matrix mod $p$ for for $1\leq i \leq \lfloor \eta n \rfloor+u$.
Let $g_0$ the the indicator function of the event that requires both $\rk (W_{n-\lfloor \eta n \rfloor}) =n-\lfloor \eta n \rfloor$ and $\CE_{n-\lfloor \eta n \rfloor}$
from Theorem~\ref{thm:A1}.
Let $g_i$ be the indicator function for the event $\CE_{n-\lfloor \eta n \rfloor+i}$ from Theorem~\ref{thm:A1}, 
 so from that theorem we have for $i\geq 1$ that $\Prob(g_i=1)\geq 1-3 e^{-c\alpha_n n}.$ 
 
 We will apply Theorem \ref{T:allerror} to the sequences $x_i, y_i$ and $g_i$ defined above. For this, we will estimate the error terms $\delta(i,b,a)$ for various values of $i,a$ and $b$.  First, note that if 
$$
\rk (W_{n-\lfloor \eta n \rfloor+i})\leq n-\lfloor \eta n \rfloor+i-u-2,
$$
then
$$
\rk (W_{n-\lfloor \eta n \rfloor+i+1})\leq n-\lfloor \eta n \rfloor+i+1-u-2.
$$
So for $i\geq 1$, 
$$
\P(y_{i+1}=*| y_i=*)=\P(x_{i+1}=*| x_i=* \land g_i=1)=1.
$$
Therefore, for $i\geq 1$ and all $a$ we have 
$$\delta(i,*,a)=0.$$
Next, Theorem~\ref{thm:A1} gives that for $i\geq 1$ and $0\leq k_0 \leq u+1$ (as $u+1\leq \eta n$),
$$
\delta(i,k_0,k_0)=\big |\P(y_{i+1}=k_0| y_i=k_0)-\P(x_{i+1}=k_0| x_i=k_0\land g_i=1) \big | \leq K e^{-c\alpha_n n}.
$$
Furthermore, if $x_i=k_0$, the only possibility for $x_{i+1}$ is either $k_0$ or $k_0+1$ (which should be interpreted as $*$ if $k_0=u+1$). It then follows that for $i\geq 1$ and all $k_0,\ell$, we have 
$$
\delta(i,k_0,\ell)=\big |\P(y_{i+1}=\ell| y_i=k_0)-\P(x_{i+1}=\ell| x_i=k_0\land g_i=1)\big | \leq K e^{-c\alpha_n n}.
$$
To this end, at the initial position $i=0$ we have
$$
\P(y_1=0 |y_0=0)=\prod_{j=0}^{n-\lfloor \eta n \rfloor}(1-p^{-( \lfloor \eta n \rfloor +j)}))
$$
and Theorem~\ref{thm:A1} gives
$$
\P(x_1=0 |x_0=0 \textrm{ and } g_0=1)\geq 1- p^{-\lfloor \eta n \rfloor}-Ke^{-c\alpha_n n} .
$$
Thus for any $\ell$, 
\begin{align*}
\delta(0,0,\ell) \leq & p^{-\lfloor \eta n \rfloor}+Ke^{-c\alpha_n n}  +  1- \prod_{j=0}^{n-\lfloor \eta n \rfloor}(1-p^{-( \lfloor \eta n \rfloor 
+j)}))\\
\leq &p^{-\lfloor \eta n \rfloor}+Ke^{-c\alpha_n n} +   p^{-\lfloor \eta n \rfloor}/(1-p^{-1}).
\end{align*}

We can apply Lemma~\ref{lemma:O} to find the $\Prob(X_{m+1}\not \in W_m | \rk(W_m)=m)$ for all $0\leq m\leq {n-\lfloor \eta n \rfloor}-1$. Taking union bound (see Corollary \ref{cor:fullrank}), we obtain  
$$
\Prob(\rk (W_{n-\lfloor \eta n \rfloor}) =n-\lfloor \eta n \rfloor) \geq  1- \alpha_n^{-1}(1-\alpha_n)^{\lfloor \eta n \rfloor+1}.
$$
So
$$
\Prob(g_0=0)\leq\alpha_n^{-1}(1-\alpha_n)^{\lfloor \eta n \rfloor+1} +3 e^{-c\alpha_n n}.
$$

We now apply Theorem \ref{T:allerror}.  The $n$ from that theorem will be what we call $\lfloor \eta n \rfloor+u$ here.
We conclude that for $k_0=u$ or $u+1$,
\begin{align*}
& \Big |\Prob(x_{\lfloor \eta n \rfloor+u} =k_0)-\Prob(y_{\lfloor \eta n \rfloor+u} =0)\Big |\\
 \leq & \frac{1}{2}(\lfloor \eta n \rfloor+u) Ke^{-c\alpha_n n} \cdot 2 + (\lfloor \eta n \rfloor+u) \cdot 3 e^{-c\alpha_n n} \\
+ & \Big(p^{-\lfloor \eta n \rfloor}+Ke^{-c\alpha_n n}  +   p^{-\lfloor \eta n \rfloor}/(1-p^{-1})\Big)
+ \Big( \alpha_n^{-1}(1-\alpha_n)^{\lfloor \eta n \rfloor+1} +3e^{-c\alpha_n n}  \Big).
\end{align*}
Here the first two terms are from the $i\geq 1$ summands in each sum, the $(\lfloor \eta n \rfloor+u)$ is from the sum over $i$, the sum over $b$ cancels with the $\P(X_i=b)$ terms, and the $2$ is from the sum over $c$ (for each $b$ there are at most $2$ values of $c$ with non-zero $\delta(i,b,c)$).    The second two terms are from the $i=0$ summands.  

Thus for $k_0=u$ or $u+1$, using $u\leq\eta n$,
\begin{align*}
\big |\Prob(x_{\lfloor \eta n \rfloor+u} =k_0)-\Prob(y_{\lfloor \eta n \rfloor+u} =k_0)\big |
\leq  2(K+3)  \eta n  e^{-c\alpha_n n} +3 e^{-c\alpha_n n} + 3 p^{- \eta n +1}+\alpha_n^{-1}(1-\alpha_n)^{ \eta n }.
\end{align*}

Since (e.g. by \cite{FA}) 
$$
\Prob(y_{\lfloor \eta n \rfloor+u} =u)=\prod_{j=1}^{n}(1-p^{-j-u})\geq 1-\sum_{j\geq 1} p^{-j-u} =1-p^{-1-u}/(1-p^{-1}), 
$$
and
\begin{align*}
\Prob(y_{\lfloor \eta n \rfloor+u} \geq u-1)
=&\Big (1+p^{-2-u}\big( \frac{1-p^{-(n-1)}}{1-p^{-1}} \big)\Big)\prod_{j=2}^{n} (1-p^{-j-u})\\
\geq 
& 1-\frac{p^{-4-2u}}{1-p^{-1}}-\frac{p^{-(n-1)-2-u}}{1-p^{-1}}.
\end{align*}
we have
that
\begin{align*}
\Prob(\rk (W_{n+u})\leq n-1)&\leq 2 p^{-1-u} + 2(K+3)  \eta n  e^{-c\alpha_n n} +3 e^{-c\alpha_n n} + 3 p^{- \eta n +1}+\alpha_n^{-1}(1-\alpha_n)^{ \eta n }
\end{align*}
and \begin{align*}
 \Prob(\rk (W_{n+u})\leq n-2) & \le  2 p^{-4-2u}+2p^{-n} + 2(K+3)  \eta n  e^{-c\alpha_n n} +3 e^{-c\alpha_n n} + 3 p^{- \eta n +1}+\alpha_n^{-1}(1-\alpha_n)^{ \eta n }.
\end{align*}
Since  $u\leq \eta n$, for some $K_0$ depending on $K,c,\eta,C_0$, for all $n$ we have
$$
\Prob\big (\rk (W_{n+u})\leq n-1\big )\leq 2 p^{-1-u} +K_0 e^{-\min(c/2,\eta\log(2)/2)\alpha_n n}
$$
and \begin{align*}
& \Prob\big(\rk (W_{n+u})\leq n-2\big )
\leq  2 p^{-4-2u}+K_0 e^{-\min(c/2,\eta\log(2)/2,\log(2))\alpha_n n}.
\end{align*}
The result follows with $c_0=\min(c/2,\eta \log(2)/2,\log(2)).$
\end{proof}

\section{Large primes}\label{section:largeprimes} 

In this section and the next we prove Proposition~\ref{prop:large}.  

{\bf Notation for Sections ~\ref{section:largeprimes} and \ref{section:structures}}:
Throughout this and the next section, we fix $T>0$ and for each positive integer $n$ we let $\xi_n$ be an $\alpha_n$-balanced random integer with $|\xi_n|\leq n^T$.  We define 
$M_{n \times (n+1)}$ to be the integral $n \times (n+1)$ matrix with entries 
 i.i.d copies of $\xi_n$.  We do not make a global assumption on the size of $\alpha_n$, but we will need different assumptions on $\alpha_n$ for the various results in these two sections.  We further fix $d>0$.  
Let  $X_1,\dots, X_{n+1}$ be the columns of $M_{n\times (n+1)}$.
We write $M_{n\times k}$ for the submatrix of $M_{n\times (n+1)}$ composed of the first $k$ columns.
  Let $W_k$ be the submodule of $\Z^n$ spanned by $X_1,\dots,X_k$.
  We write $X_k/p$ and $W_k/p$ for their reductions mod $p$ (and more generally use this notation to denote the reduction of an object from $\Z$ to $\Z/p\Z$).  
     We let $n_0:=n- \lfloor  \frac{3\log n}{\alpha_n}  \rfloor$.

 
 Let 
 $$\CP_n:=\Big\{p \textrm{ prime},   p \geq e^{d \alpha_n n} \Big\}.$$ 
 Let $\CE_{\neq 0}$ be the event that $\det(M_{n \times n}) \neq 0$. As mentioned in the introduction section, from \cite{NP} and also by taking the limit as $p\ra\infty$ in Theorem \ref{th:allmed}, we have
$$\P(\CE_{\neq 0}) \ge 1 -K_0e^{- c_0 \alpha_n n}$$ for absolute constants $c_0,K_0$. 
Our strategy is as follows.  We consider the columns of the matrix one at a time, and check if they are the span of the previous columns modulo $p$ for each prime in $\CP_n$.  We cannot control whether this happens, as $\CP_n$ contains too many primes, but  each $p$ for which this happens is put on a ``watch list'' (called $\CB_k$) and necessarily divides the determinant of $M_{n\times n}$. 
 If the watch list grows too large, since all the primes in the watch list are large, then too a large number divides the determinant, and $M_{n\times n}$ must be singular.  However, we have already bounded the probability of that occurring.  Otherwise, if our watch list is not too large, for each prime in the watch list, we can bound the probability that the next column is in the span of the previous columns mod that prime.


Let $\CB_k$ be the set of primes $p\in \CP_n$ such that $\rk(W_k/p) \leq k-1$.
Let $\Con_k$ be the event that $|\CB_k|\leq  (2T+1) \log n/(2d \alpha_n )$ (the watch list is under control).
Note that any $p\in \CB_k$ for $k\leq n$ must divide $\det(M_{n \times n})$.  By Hadamard's bound, $|\det(M_{n \times n})|\leq n^{n/2}n^{Tn}$, and so in particular, when $\bar{\Con}_k$ occurs (``the watch list is out of control'') then $\det(M_{n \times n}) = 0$.  Let $\Drop_k$ be the event that there is a $p\in \CB_k$ such that 
$\rk(W_k/p) \leq k-2$ (the rank drops), this is the event we want to avoid. 

We will show $\P(\bar{\Con}_{k+1} \lor \bar{\Drop}_{k+1} | \bar{\Con}_k \lor \bar{\Drop}_k)$ is large.
The goal is to conclude that  $\P(\bar{\Con}_{n} \lor \bar{\Drop}_{n})$ is large, and since we know that $\P(\bar{\Con}_{n})$ is small, we can conclude that $\P(\bar{\Drop}_{n})$ is large, as desired.
Note that since $\CB_k\sub \CB_{k+1}$, we have that  $\bar{\Con}_k\sub \bar{\Con}_{k+1}.$
Thus
$$
\P(\bar{\Con}_{k+1} \lor \bar{\Drop}_{k+1} | \bar{\Con}_k )=1.
$$
It remains to estimate $\P(\bar{\Con}_{k+1} \lor \bar{\Drop}_{k+1} | \Con_k \land  \bar{\Drop}_k).$
We condition on the exact values of $X_1,\dots, X_k$ where $\Con_k \land  \bar{\Drop}_k$ holds, and so
there are at most $(2T+1)\log n/(2d \alpha_n )$ primes $p\in \CP_n$ such that $\rk(W_k/p) \leq k-1$ and no prime 
$p\in \CP_n$ such that $\rk(W_k/p) \leq k-2$.  In this case  $\bar{\Drop}_{k+1}$, as long as for each $p\in \CB_k$, we have
$X_{k+1}/p\not\in W_k/p$.  Consider one prime $p\in\CB_k$, and let $V$ be the value of $W_k/p$ that the conditioned $X_1,\dots, X_k$ give.
From Lemma~\ref{lemma:O}, $\P(X_{k+1}/p\in V)\leq (1-\alpha_n)^{n-(k-1)}.$  Thus, 
$$
\P(\bar{\Con}_{k+1} \lor \bar{\Drop}_{k+1} | \Con_k \land  \bar{\Drop}_k)\geq 1- \left(\frac{(2T+1)\log n}{2d \alpha_n } \right)(1-\alpha_n)^{n-(k-1)}.
$$
In particular, we conclude that
$$
\P(\bar{\Con}_{k+1} \lor \bar{\Drop}_{k+1} | \bar{\Con}_k \lor \bar{\Drop}_k)\geq 1- \left(\frac{(2T+1)\log n}{2d \alpha_n } \right)(1-\alpha_n)^{n-(k-1)}.
$$
Then inductively, we have
$$
\P(\bar{\Con}_{k} \lor \bar{\Drop}_{k} )\geq 1- \sum_{i=1}^{k-1}\left(\frac{(2T+1)\log n}{2d \alpha_n } \right)(1-\alpha_n)^{n-(i-1)}=
1- \left(\frac{(2T+1)\log n}{2d \alpha_n } \right)\frac{(1-\alpha_n)^{n-k+2}}{\alpha_n}.
$$
We defined $n_0:=n- \lfloor  \frac{3\log n}{\alpha_n}  \rfloor$ above and so 
if we let $k=n_0$ then if we assume $\alpha_n\geq n^{-1}$, 
 we have that 
$$
\P\big(\bar{\Con}_{n_0 } \lor \bar{\Drop}_{n_0} \big)\geq 1- O_{d,T}\big( n^{-1/2} \big).
$$
Certainly as $k$ gets very close to $n$, $W_k$ has very small codimension, and so Odlyzko's bound will not continue to be strong enough.  Thus for the remaining $k$ we will have to use a different bound.

\subsection{Proof of  Proposition~\ref{prop:large} when $\a_n \ge n^{-1/6+\eps}$}\label{S:nottoosparse}
First, in this section, we will prove Proposition~\ref{prop:large} for the denser case $\a_n \ge n^{-1/6+\eps}$.  For these larger $\alpha_n$ we can present a simpler proof than in the case when $\a_n$ might be as small as $n^{-1+\eps}.$  
Odlyzko's bound is sharp for some spaces, e.g. the hyperplane of vectors with first coordinate $0$, and so if we need to improve on Odlyzko's bound we cannot expect to do it for all spaces at once.  The overall strategy is to see that apart from some bad subspaces, we can improve on Odlyzko's bound, and we can also prove that it is unlikely that $W_k$ is one of those bad spaces.  At this level of generality, this description fits the small and medium primes sections.  However, the specifics are very different, because the small and medium primes sections treat one prime at a time, and now we are in a regime where there are just too many possible primes to add the probability of $W_k$ being bad over all the primes (e.g. adding $\P(\bar{\CE}_{n-k})$ from Theorem~\ref{thm:A1} over all primes up to $n^{n/2+Tn}$ gives too big a result).  On the other hand, we do not need the same strength of improvement over Odlyzko's bound that Theorem~\ref{thm:A1} provides, because the bound on the probability of $X_{k+1}/p\in W_k/p$ only has to be added over the small number of primes in the watch list.  The following lemma balances these requirements, and its proof will be delayed till the end of this subsection.




\begin{lemma}\label{lemma:sqrt} 
Suppose that $\alpha_n\geq 6\log n/n$. 
 Then there is a set of submodules  $\mathcal{S}$ of $\Z^n$ such that 
$$
\Prob(W_{n_0}\in \mathcal{S})\geq 1-e^{-\alpha_n n/8}
$$
and for any prime $p\ge e^{d \alpha_n n}$, and any submodule $H\in\mathcal{S}$, for any proper subspace $H'$ of $(\Z/p\Z)^n$ containing $H/p$,
$$\P\Big( X/p \in H'\Big)=O_{d,T}\left( \frac{ \sqrt{\log n}}{\alpha_n \sqrt{ n}} \right),$$
where $X$ is any column of $M_{n\times n}$.
\end{lemma}

Now, we will also condition on $\Good$, which we define to be the event that $W_{n_0}\in \mathcal{S}$ (i.e., $W_{n_0}$ is $\Good$ood).
We then have
$$
\P\Big((\bar{\Con}_{n_0} \lor \bar{\Drop}_{n_0}) \land 
\Good
\Big)\geq 1-  e^{-\alpha_n n/8} + O_{d,T}(  n^{-1/2}),
$$
Now let $ n_0 \leq k\leq n$.  As before, since $\bar{\Con}_k\sub \bar{\Con}_{k+1},$ we have
$$
\P\Big((\bar{\Con}_{k+1} \lor \bar{\Drop}_{k+1} ) \land 
\Good |  \bar{\Con}_{k} \land 
\Good 
\Big)=1.
$$
It remains to estimate $\P\Big((\bar{\Con}_{k+1} \lor \bar{\Drop}_{k+1} ) \land 
\Good |  \Con_{k} \land \bar{\Drop}_{k} \land 
\Good 
\Big).$  Again, we condition on exact values of $X_1,\dots,X_k$ such that $\Con_{k}, \bar{\Drop}_{k},\Good$ hold.
Then $\bar{\Drop}_{k+1}$ holds unless for some $p\in \CB_k$ we have $X_{k+1}/p\in W_k/p$.  Since $\Con_{k}$ holds, we have a bound on the size of $\CB_k$, and since $\Good$ holds, we can use Lemma~\ref{lemma:sqrt} to bound  the probability that $X_{k+1}/p$ is in $W_k/p$. 
(Note that even if $k=n$, for $p\in \CB_k$, we have that  $W_k/p$ is a proper subspace of $(\Z/p\Z)^n$, and that $n^{-1/6}\geq 2 (\log n)/n$.)
We conclude that 
$$\P\Big((\bar{\Con}_{k+1} \lor \bar{\Drop}_{k+1}) \land 
\Good |  \Con_{k} \land \bar{\Drop}_{k} \land 
\Good 
\Big)\geq 1-\left(\frac{(2T+1)\log n}{2d \alpha_n } \right) O_{d,T}\left( \frac{ \sqrt{\log n}}{\alpha_n \sqrt{ n}} \right).
$$ 
Inductively, starting from $k=k_0$ we then have
$$\P\Big((\bar{\Con}_{n} \lor \bar{\Drop}_{n} ) \land 
\Good 
\Big)\geq 1-\lfloor  \frac{3\log n}{\alpha_n} \rfloor  O_{d,T}\left( \frac{ \log^{1.5} n}{\alpha^2_n \sqrt{ n}} \right)-  e^{-\alpha_n n/8} + O_{d,T}(  n^{-1/2}).
$$
So then, 
\begin{equation}\label{eqn:key16'}
\P(\bar{\Drop}_{n} 
)\geq 1
- O_{d,T}\left( \frac{ \log^{2.5} n}{\alpha^3_n \sqrt{ n}} \right)
-K_0e^{- c_0 \alpha_n n}-  e^{-\alpha_n n/8} + O_{d,T}(  n^{-1/2}).
\end{equation}
To this end, since $\alpha_n\geq n^{-1/6+\eps}$, we have that 
$$\P(\bar{\Drop}_{n} 
)\geq 1- O_{d,T,\eps}(n^{-\eps}),
$$
which is exactly Equation~\eqref{eqn:large:nprop}.  Equation~\eqref{eqn:large:n+1prop} follows similarly, with $\lfloor 3\log n/\alpha_n \rfloor$, the number of steps in the induction, replaced by $\lfloor 3\log n/\alpha_n \rfloor +1.$

The key, of course, is to now verify Lemma
\ref{lemma:sqrt}, which is the heart of the proof of Proposition~\ref{prop:large} for $\alpha_n\geq n^{-1/6+\eps}$.  We need a way to bound the probability that
$X_{k+1}/p\in W_k/p$ that works for all $p\in \CB_k$ and is effective for large $k$.
For this, we introduce a version of the classical Erd\H{o}s-Littlewood-Offord result in $\Z/p\Z$. 

\begin{theorem}[forward Erd\H{o}s-Littlewood-Offord, for non-sparse vectors]\label{theorem:LO} Let $X \in (\Z/p\Z)^n$ be a random vector whose entries are i.i.d. copies of a random variable $\nu_n$ satisfying $\max_{r \in \Z/p\Z} \P(\nu_n =r) \le1-\a_n$. Suppose that $w \in (\Z/p\Z)^n$ has at least $n'$ non-zero coefficients, and $\a_n \ge 4/n'$. Then we have
$$|\P(X \cdot w =r) -\frac{1}{p}| \le \frac{2}{\sqrt{\alpha_n n'}}.$$
\end{theorem}

A proof of this result due to Maples (based on an argument by
Hal\'asz) can be seen in \cite[Theorem 2.4]{M1} (see also \cite[Theorem A.21]{NP}.) 
To use Theorem \ref{theorem:LO}, we need to know it is unlikely that $W_k$ has normal vectors with few non-zero entries.  First, we will see this is true over $\R$.  The approach is standard: there are few sparse vectors and by the Odlyzko's bound each is not that likely to be normal to $W_k$. 

\begin{lemma}[Sparse normal vectors over $\R$ unlikely]\label{lem:sparse} 
Suppose  $\alpha_n\geq \log n/n$ and $k\geq n/2$.
 For $n$ sufficiently large (in an absolute sense), with probability at least $1-e^{-\alpha_n n/8}$, the random subspace $X_1,\dots,X_k$ does not have a non-trivial normal vector with less than $\alpha_n n/(32\log n)$ non-zero entries.
\end{lemma}


\begin{proof}(of Lemma \ref{lem:sparse}) 
Let $l= \lfloor \alpha_n n/(32\log n) \rfloor$. With a loss of a multiplicative factor $\binom{n}{l}$ in probability, we assume that there exists a vector $w=(w_1,\dots, w_l,0,\dots,0)$ which is normal to $X_{1},\dots, X_{k}$. 
Let $M_{l\times k}$ be the matrix with columns given by the first $l$ coordinates of each of $X_1,\dots,X_k$, which has rank at most $l-1$.
With a loss of a multiplicative factor $l$ in probability, we assume that the first row of $M_{l\times k}$ belongs to the subspace $H$ generated by the other $l-1$ rows.  
However, as $H$ has codimension at least $k-l$, 
Lemma~\ref{lemma:O} implies a bound $(1-\alpha_n)^{k-l}$ for this event. Putting together, the event under consideration is bounded by 
\begin{align*}
\binom{n}{l} \times l \times (1-\alpha_n)^{k-l} 
\le
e^{(\log (32\log n/\alpha_n ) +1)\alpha_n n/(32 \log n)} \times e^{\log( \alpha_n n/(32 \log n))} \times e^{-\alpha_n(k-l)}.
\end{align*}
We then have that the exponent of $e$ in the above bound is
$$
\leq \big(\frac{\log\log n}{\log n}+\frac{\log(\alpha_n^{-1})}{\log n}+\frac{\log(32)}{\log n}+\frac{1}{\log n}\big) \alpha_n n/32 +\log(\alpha_n n) -\alpha_n n/4 \le -\alpha_n n /8
$$
for $n$ sufficiently large so that
$$
\big(\frac{\log\log n}{\log n}+\frac{\log(32)}{\log n}+\frac{1}{\log n}\big)\leq 1
$$
and so that $\log (\alpha_n n)\leq \alpha_n n/16$ (which happens for $\alpha_n n\geq 22$, which is implied by $\log n \geq 22$).
\end{proof}

We could prove a similar lemma to Lemma~\ref{lem:sparse} over $\Z/p\Z$ for each $p$, but we could not sum the probabilities $e^{-\alpha_n n/8}$ of sparse normal vectors over any meaningful range of primes $>e^{d\alpha n}$.  
However, now we will prove a deterministic lemma, that lets us lift normal vectors with few non-zero entries
from characteristic $p$,  for large $p$, to characteristic $0$.  Then there is only one bad event to avoid instead of one for each $p$.  This aspect of our argument is unlike previous approaches and uses critically a lower bound on $p$.

\begin{lemma}[Lifting sparse normal vectors from $\Z/p\Z$ to $\R$]\label{lemma:passing}  
Let $k,l,n$ be positive integers, and $M$ a $l\times k$ matrix with integer entries $|M_{ij}|\leq n^T$.
If $p$ is a prime larger than $e^{(k \log k)/2+k T \log n}$, then the rank of $M$ over $\Q$ is equal to the rank of $M/p$ over $\Z/p\Z$.  This has the following corollaries. 
\begin{enumerate}
\item \label{i:lindep} If $Z_1,\dots, Z_k \in \Z^{l}$ are vectors with entries $|Z_{ij}|\leq n^T$, and $Z_1/p,\dots, Z_k/p$ are linearly dependent in $(\Z/p\Z)^{l}$, then $Z_1,\dots, Z_k$ are also linearly dependent in $\Z^{l}$. 
\item \label{i:sparse} Let 
$Z_1,\dots, Z_{i}\in \Z^m$ be vectors with entries $|Z_{ij}|\leq n^T$ . If there is a non-zero vector $w \in (\Z/p\Z)^m$ with at most $k$ non-zero entries that is normal to $Z_1/p,\dots, Z_{l}/p$, then there is a non-zero vector $w' \in \Z^m$ with at most $k$ non-zero entries and normal to $Z_1,\dots, Z_{l}$.
\item \label{i:kernel} The kernel of the map $M: \Z^k \ra \Z^l$ surjects onto the kernel of the map
$M: (\Z/p\Z)^k \ra (\Z/p\Z)^l$.

\end{enumerate}
\end{lemma}

\begin{proof}(of Lemma \ref{lemma:passing}) 
The rank is the greater $r$ such that an $r \times r $ minor has non-zero determinant.  
Since by Hadamard's bound, the determinant of a $r\times r$ minor is at most $e^{(r \log r)/2+r T\log n}$, for primes $p>e^{(k \log k)/2+k T \log n}$ and $r\leq k$, a determinant of an $r\times r$ minor vanishes in $\Z$ if and only if it vanishes mod $p$.  Thus we conclude the main statement of the lemma.  For the first corollary, consider the $\ell \times k$ matrix with $Z_1,\dots, Z_k$ as columns. The vectors are linearly dependent if and only if the matrix has rank less than $k$.
For the second corollary, assume that $\sigma=\supp(w)\subset [m]$ and $Z_1|_\sigma,\dots, Z_{l}|_\sigma$ are the restrictions of $Z_1,\dots, Z_{l}$ over the components in $\sigma$. By definition, the $k$ row vectors of the matrix formed by $Z_1|_\sigma,\dots, Z_{l}|_\sigma$ are dependent when reduced mod $p$, and thus these vectors are dependent over $\Z$.  This gives a non-zero $w' \in \Z^m$ with $|\supp(w')|\leq k$ that is normal to $Z_1,\dots, Z_{l}$.  For the third corollary, we can express $M$ under the Smith normal form $M=S_1D S_2$,
where $S_1\in \GL_l(\Z)$ and $S_2\in \GL_k(\Z)$ and $D$ is an integral diagonal matrix. Then since the ranks of $M$ and $D$ agree over $\Q$ and over $\Z/p\Z$, we conclude that $D$ has the same rank over $\Q$ or $\Z/p\Z$.  This implies that the only diagonal entries of $D$ that are divisible by $p$ are the ones that are $0$.  From this it follows that the kernel of $D :\Z^k \ra \Z^l$ surjects onto the kernel of $D: (\Z/p\Z)^k \ra (\Z/p\Z)^l$.  Multiplication by $S_2^{-1}$ on the left takes these kernels of $D$ to the corresponding kernels of $M$, and the statement follows.
\end{proof}


Putting this all together, we can now prove Lemma \ref{lemma:sqrt}.  The choices of parameters are rather delicate here, e.g. we could obtain more non-zero coordinates of a normal vector than Lemma~\ref{lem:sparse} provides, but then we could not use Lemma~\ref{lemma:passing} to lift those non-zero coordinates.

\begin{proof}[Proof of Lemma \ref{lemma:sqrt}]
We let $k=c\alpha_n n/\log n$, where $c<1/32$ is a sufficiently small constant (in terms of $d,T$) such that 
$$
e^{(k \log k)/2+k T \log n}<e^{d\alpha_n n}.
$$
Since $\alpha_n \geq 6 \log n/n$, it follows that $n_0\geq n/2$, and we can apply Lemma~\ref{lem:sparse} and find that for sufficiently large $n$, 
with probability at least $1-e^{-\alpha_n n/8}$, 
$W_{n_0}$ does not have a normal vector with less than $k$ entries. 
Let $\mathcal{S}$ be the set of submodules of $\Z^n$ that do not have a normal vector with less than $k$ non-zero coordinates. 
 Then by Lemma \ref{lemma:passing}, for each prime $p\geq e^{d\alpha_n n}$, if $W_{n_0}\in\mathcal{S}$, then
the space $W_{n_0}/p$ (and thus any space containing this space) does not have a non-trivial normal vector with less than $k$ non-zero coordinates. 
Since $\alpha_n \geq 6 \log n/n$, for $n$ sufficiently large in terms of $d,c$, we have that $e^{d\alpha_n n}\geq \sqrt{\alpha_n k}$.
Thus by Theorem~\ref{theorem:LO}, 
for any prime $p\geq e^{d\alpha_n n}$ the following holds.
Let $H$ be a subspace of $(\Z/p\Z)^n$ that does not have a non-trivial normal vector with less than $k$ non-zero entries, and then
for any proper subspace $H'$ of $(\Z/p\Z)^n$ containing $H$ and with normal vector $w$,
$$
\P(X\in H')\leq \P(X\cdot w=0)\leq \frac{3}{\sqrt{\alpha_n k}}=\frac{3 \sqrt{\log n}}{\alpha_n \sqrt{ c n}}.
$$
\end{proof}


\subsection{Proof of  Proposition~\ref{prop:large} in general} 
Equation~\eqref{eqn:key16'} shows exactly why
$n^{-1/6+\eps}$ is the threshold exponent for $\alpha_n$ such that the
above method can work, as the error bound has an $\alpha_n^3n^{1/2}$
in the denominator.   Thus, to obtain results that can work for smaller $\alpha_n$, we need a further improvement on Odlyzko's bound, which requires that we consider further bad subspaces besides those with sparse normal vectors.  
We have the following upgrade to Lemma~\ref{lemma:sqrt}, whose proof is rather more involved than that of  Lemma~\ref{lemma:sqrt}, will be completed in the next section, and again, is the heart of the proof.


\begin{lemma}\label{L:upgrade}
There is an absolute constant $\ct>0$ such that the following holds. 
Suppose that $\alpha_n\geq n^{-1+\epsilon}$.  
There is a set $\mathcal{S}'$ of submodules of $\Z^n$, such that
$$\P(W_{n_0}\in \mathcal{S}')\geq 1- e^{-\ct \alpha_n n}$$ 
and for $n$ sufficiently large given $d,\epsilon, T$, 
 any prime $p\ge e^{d \alpha_n n}$,  any submodule $H\in\mathcal{S}'$
 and any proper subspace $H'$ of $(\Z/p\Z)^n$ containing $H/p$,
$$\P\Big( X/p \in H'\Big)\leq n^{-3},
$$
where $X$ is any column of $M_{n\times n}$.
\end{lemma}
The rest of the proof goes the same as after Lemma~\ref{lemma:sqrt}, replacing $\mathcal{S}$ with $\mathcal{S}'$ .  We conclude  that for $\alpha_n\geq n^{-1+\eps}$, we have that  $\P(\bar{\Drop}_{k} 
)\geq 1- O_{d,T,\eps}(n^{-1/2}),
$ for $k=n,n+1$, which proves Proposition~\ref{prop:large}.
We have  not attempted to optimize the error, or even record in Proposition~\ref{prop:large} the error this argument proves (as we wanted to give a weaker statement that could be proved by the simpler argument above when the $\alpha_n$ were not too small).

\section{Proof of Lemma~\ref{L:upgrade}: Enumeration of Structures }\label{section:structures}

Instead of only avoiding  sparse normal vectors, in Lemma~\ref{L:upgrade} we will avoid normal vectors with more general structure.  We now need to make some definitions necessary to describe this structure.

\subsection{Additive structures in abelian groups} Let $G$ be an (additive) abelian group. 
\begin{definition}
A set $Q$ is a \emph{generalized arithmetic progression} (GAP) of
rank $r$ if it can be expressed as in the form
$$Q= \{a_0+ x_1a_1 + \dots +x_r a_r| M_i \le x_i \le M_i' \hbox{ and $x_i\in\Z$ for all } 1 \leq i \leq r\}$$
for some elements $a_0,\ldots,a_r$ of $G$, and for some integers $M_1,\ldots,M_r$ and $M'_1,\ldots,M'_r$.

It is convenient to think of $Q$ as
the image of an integer box $B:= \{(x_1, \dots, x_r) \in \Z^r| M_i \le x_i
\le M_i' \} $ under the linear map
$$\Phi: (x_1,\dots, x_r) \mapsto a_0+ x_1a_1 + \dots + x_r a_r. $$
Given $Q$ with a representation as above
\begin{itemize}
\item the numbers $a_i$ are  \emph{generators} of $Q$, the numbers $M_i$ and $M_i'$ are  \emph{dimensions} of $Q$, and $\Vol(Q) := |B|$ is the \emph{volume} of $Q$ associated to this presentation (i.e. this choice of $a_i,M_i,M_i'$);
\vskip .05in
\item we say that $Q$ is \emph{proper} for this presentation if the above linear map is one to one, or equivalently if $|Q| =|B|$;
\vskip .05in
\item If $-M_i=M_i'$ for all $i\ge 1$ and $a_0=0$, we say that $Q$ is {\it symmetric} for this presentation.
\end{itemize}
\end{definition}

 The following inverse-type idea, which was first studied by Tao and Vu about ten years ago (see for instance~\cite{TVinverse}), will allow us prove bounds much sharper than Theorem \ref{theorem:LO}.

\begin{theorem}[inverse Erd\H{o}s-Littlewood-Offord]\label{theorem:ILO} Let $\eps<1$ and $C$ be positive constants. 
Let $n$ be a positive integer.
Assume that $p$ is a prime that is larger than $C'n^C$ for a sufficiently large constant $C'$ depending on $\eps$ and $C$. 
Let $\nu$ be a random variable taking values in $\Z/p\Z$ which is $\alpha_n$-balanced, that is $\max_{r \in \Z/p\Z} \P(\nu=r) \le 1-\alpha_n$ where $\al_n \ge n^{-1+\eps}$.
 Assume $w=(w_1,\dots, w_n)\in(\Z/p\Z)^n$ such that  
 $$\rho (w) := \sup_{a\in \Z/p\Z} \P(\nu_1 w_1 +\dots + \nu_n w_n=a) \ge  n^{-C},$$
 where $\nu_1,\dots, \nu_n$ are iid copies of  $\nu$.  Then for any $n^{\ep/2} \a_n^{-1} \le n' \le n$ there exists a proper symmetric GAP $Q$ of rank $r=O_{C,\ep}(1)$ which contains all but $n'$ elements of $w$ (counting multiplicity), where 
$$|Q|\le \max\left \{1, O_{C,\ep}(\rho^{-1}/(\al_n n')^{r/2})\right \} .$$ 
\end{theorem}

When  $\a_n$ is a constant, we then recover a variant of \cite[Theorem
2.5]{NgV}. The new, but not too surprising, aspects here are that the result works for small $\a_n$ and for $\Z/p\Z$ for large enough $p$. 
A proof of Theorem~\ref{theorem:ILO}  will be presented in Appendix~\ref{section:ILO} by  modifying the approach of \cite{NgV}.
 We remark that it is in the proof of Theorem~\ref{theorem:ILO} where the  requirement $\a_n n' \ge n^{\eps/2}$ is crucial (which henceforth requires $\a_n$ to be at least $n^{\eps/2-1}$) to guarantee polynomial growth of certain sumsets (see~\eqref{eqn:LO:grow}).  We see that $Q=\{0\}$ includes the special case of sparse $w$.  Theorem~\ref{theorem:ILO} is much sharper than Theorem~\ref{theorem:LO}, and relates the volume of the GAP involved to the bound for $\rho(w)$.
 
We let
$$n':=\lceil n^{\eps/2} \a_n^{-1}\rceil$$ and $m=n-n'$ for the rest of this section, and we will apply Theorem~\ref{theorem:ILO} with this choice of $n'$ and $C=3$.
Thus it will be convenient to let $C_\eps$ be the maximum of $C'$ and the constants from the $O_{C,\eps}$ notation bounding the rank of volume of $|Q|$ in Theorem~\ref{theorem:ILO} applied with $C=3$.  We call a GAP $Q$ \emph{well-bounded} if it is of rank $\leq C_\eps$ and $|Q|\leq C_\eps n^3$.  
We call a vector $w$ \emph{structured} if it is non-zero, and there exists a symmetric well-bounded GAP $Q$ such that all but $n'$ coordinates of $w$ belong to $Q$. Note that it is not always true that $\rho(w) = n^{-O(1)}$ if $w$ is structured in this sense. 


Our general approach is to see that it is not too likely for $W_{n_0}/p$ to have structured normal vectors.  We need to handle the case of $r=0$ separately from the case of $r\geq 1$, as in the latter case we will use the $(\alpha_n n')^{r/2}$ term crucially.  Now, we will give a very different approach to proving $W_{n_0}/p$ is unlikely to have sparse vectors than we used in Lemma~\ref{lem:sparse}, as Lemma~\ref{lem:sparse} is too weak for small $\alpha_n$.  The method of Lemma~\ref{L:supersparse} will actually give better results as $\alpha_n$ gets smaller, while Lemma~\ref{lem:sparse} gets worse.  This method will automatically control sparse vectors for all large primes at once, without any lifting from characteristic $p$ to characteristic $0$.  
Notably, the bound we get from Lemma~\ref{L:supersparse} will be the largest term in our error. 




\begin{lemma}[Extremely sparse normal vectors]\label{L:supersparse}
There are absolute constants $\co, \Co$ such that the following holds. Let $\beta_n:=1- \max_{x\in \Z} \P(\xi_n=x)$, and assume $\beta_n\geq \Co \log n/n$ and $\alpha_n\geq 6 \log n/n$.
For $n\geq 2$,  the following happens with probability at most $e^{-\co \beta_n  n/2}$: for some prime $p>2n^T$,
the space $W_{n_0}/p$ has a non-zero normal vector with at most $144  \beta_n^{-1}$ non-zero coordinates.
\end{lemma}



\begin{proof}(of Lemma \ref{L:supersparse})
In fact, we will show that the following holds with probability at least $1- e^{-\co \beta_n  n/2}$. For any $1\le t \le 144  \beta_n^{-1}$, and any $\sigma \in \binom{[n]}{t}$, there are at least two columns $X_i,X_{j}$ whose restriction $(X_{j}-X_i)|_\sigma$ has exactly one non-zero entry.  We first show that this will suffice to prove the lemma.
Since $(X_{j}-X_i)|_\sigma$ has a unique non-zero entry, and all its entries are at most $2n^T$ in absolute value, for any prime $p>2n^T$ we have that $(X_{j}/p-X_i/p)|_\sigma$ has exactly one non-zero entry. 
Suppose we had a normal vector $w$ to 
$W_{n_0}/p$ with $1\le t \le 144  \beta_n^{-1}$ non-zero entries, and let $\sigma$ be the indices of those entries.  Since $w|_\sigma$ is normal to $(X_{j}/p-X_i/p)|_\sigma$, that would imply that one of the $\sigma$ coordinates of $w$ is zero, which contradicts the choice of $\sigma$.  

Now we prove the claim from the beginning of the proof. Our method is similar to that of \cite[Lemma 3.2]{BR} and \cite[Claim A.9]{NP}. 
 For $k\in \{1, 3,\dots, 2\lfloor (n_0-1)/2 \rfloor +1\}$, consider the vectors $Y_i = X_{k+1}-X_{k}$. The entries of this vectors are iid copies of the symmetrized random variable $\psi =\xi-\xi'$, where $\xi',\xi$ are independent and have distribution $\xi_n$. With $1-\beta_n':=\P(\psi=0)$, then $\beta_n \le \beta_n'  \le 2\beta_n$ as this can be seen by
\begin{equation}\label{eqn:symmetrization}
(1- \beta_n)^2 \le \max_{x} \P(\xi= x)^2 \le \sum_x \P(\xi=x)^2=\P(\psi =0) \le \max_{x} \P(\xi= x) = 1-  \beta_n.
\end{equation}

Now let $p_\sigma$ be the probability that all $Y_i|_\sigma, i\in \{1, 3,\dots, 2\lfloor (n_0-1)/2 \rfloor +1 \}$ fail to have exactly one non-zero entry (in $\Z$), then by independence of the columns and of the entries
$$p_\sigma = (1- t  \beta_n' (1- \beta_n')^{t-1})^{\lfloor (n_0+1)/2 \rfloor} \le (1 - t  \beta_n' e^{- (t-1) \beta_n'})^{n_0/2} \le e^{-n t \beta_n' e^{-(t-1)  \beta_n'}/4}.$$
(Recall since $\alpha_n\geq 6 \log n/n$ we have $n_0\geq n/2$.)
Notice that as $1\le t \le 144  \beta_n^{-1}$, we have $e^{-(t-1)  \beta_n'}/4 \ge \co$ for some positive constant $\co$, and hence 
$$ e^{-n t \beta_n' e^{-(t-1)  \beta_n'}/4} \le (e^{ -\co n  \beta_n' })^t \le n^{-\co \Co t /2} e^{ -\co n  \beta_n/2 },$$
for any $\Co>0$. 
Thus 
$$\sum_{1\le t \le 144  \beta_n^{-1}}\sum_{\sigma \in \binom{[n]}{t}} p_\sigma \le \sum_{1\le t \le 144  \beta_n^{-1}} \binom{n}{t}  n^{-\co \Co t /2} e^{ -\co  n  \beta_n/2 } \le  \sum_{1\le t \le 144  \beta_n^{-1}}  (n^t n^{-\co \Co t /2}) e^{ -\co  n  \beta_n/2} <e^{ -\co  n  \beta_n/2},$$
provided that $n\geq 2$ and $\Co$ is sufficiently large in terms of $\co$.
\end{proof}

The downside of Lemma~\ref{L:supersparse} is that is is rather weak for constant $\alpha_n$.  
So it needs be combined with an improvement of Lemma~\ref{lem:sparse}.  For the improvement, we use Littlewood-Offord (Theorem~\ref{theorem:LO}) in place of Odlyzko's bound.  However, that substitution only makes sense once have have $k$ non-zero coordinates in our normal vector and $\alpha_n k$ is at least a constant.  Luckily, Lemma~\ref{L:supersparse} provides us with exactly that. This strategy is analogous to that used in the proof of \cite[Proposition A.8]{NP}. 


\begin{lemma}[Moderately sparse normal vectors]\label{lemma:sparse:moderate} 
There exist absolute
  constants $\cz,\Cz$ such that the following holds.  
Let $\beta_n:=1- \max_{x\in \Z} \P(\xi_n=x)$.  
  Assume  $\a_n \ge \frac{\Cz \log n}{n}$ and let $p$ be a prime $>2n^T$. 
  The following happens with probability at most  $(2/3)^{n/4}$: 
   the space $W_{n_0}/p$ has a non-zero normal vector $w$ with $144\beta_n^{-1} \leq |\supp(w)| \leq \cz n$.
\end{lemma}
Note Lemma~\ref{lemma:sparse:moderate} only bounds the probability of sparse normal vectors modulo one $p$ at a time, unlike Lemma~\ref{L:supersparse}, which controls sparse normal vectors modulo all sufficiently big primes.

\begin{proof}(of Lemma \ref{lemma:sparse:moderate}) 
For $\sigma \subset [n]$ with $144\beta_n^{-1}\le t=|\sigma| \le \cz n$, consider the event that 
 $W_{n_0}/p$ is normal to a vector $w$ with $\supp(w) =\sigma$ but not to any other vector of smaller support size.
With a loss of a multiplicative factor $\binom{n}{t}$ in probability, we assume that $\sigma=\{1,\dots,t\}$. 
Consider the submatrix $M_{t \times n_0}$ of $M_{n\times n}$ consisting of the first $t$ rows and first $n_0$ columns of $M_{n\times n}$.  Since the restriction $w|_\sigma$ of $w$ to the first $t$ coordinates is normal to all the columns of 
$M_{t \times n_0}/p$, the matrix $M_{t \times n_0}/p$ has rank $t-1$ (if $p=0$, we mean rank over $\R$).  With a loss of a multiplicative factor $\binom{n_0}{t-1}$ in probability, we assume that the column space of $M_{t \times n_0}/p$
is spanned by its first $t-1$ columns.  

Note that for $p>2n^T$, the value of $\xi_n$ is determined by its value mod $p$, and so $\beta_n=1- \max_{x\in \Z/p\Z} \P(\xi_n/p=x)$.   If we fix $X_1,\dots,X_{t-1}$ such that $W_{t-1}|_\sigma/p$ has a normal vector with all $t$ coordinates non-zero, then  
by Theorem \ref{theorem:LO} , the probability that $X_i|_\sigma/p\in W_{t-1}|_\sigma/p$ for all $t\leq i \leq n_0$ is at most
$$(\frac{1}{p} + \frac{2}{\sqrt{ \beta_n  t}})^{n_0-t+1}\le ( \frac{1}{p} + \frac{2}{\sqrt{ \beta_n t}})^{(1-2c_0) n} \le  (\frac{2}{3})^{n/2}  .$$
The first inequality follows as long as $\Cz\geq 3/\cz$ as then we have
 $\cz n\geq 3 \log n/\alpha_n$ and $n_0\geq n-c_0n$.
Thus the total probability of the event in the lemma is at most
$$ \sum_{144  \beta_n^{-1} \le t \le c_0 n}\binom{n}{t}^2 (\frac{2}{3})^{n/2} \le (\frac{2}{3})^{n/4}.$$
provided that $c_0$ is sufficiently small absolutely. 
\end{proof}

Now we will show that the probability of having a structured normal vector for a GAP of rank $r\geq 1$ (which was defined in the discussion following Theorem~\ref{theorem:ILO}) is extremely small.

\begin{lemma}[Structured, but not sparse, normal vectors]\label{L:strucnotsp}
Let $\alpha_n\geq n^{-1+\eps}$.
Let $p$  be a  prime $p\geq  C_\eps n^{3}$.
The following event happens with probability $O_{\eps}(p^{C_\eps} n^{-\eps n/5})$: the space
$W_{n_0}/p$ has a structured normal vector $w$, and 
$W_{n_0}/p$ does not have a non-zero normal vector $w'$ 
such that $|\supp(w')|\leq \cz n$ with $\cz$ from Lemma~\ref{lemma:sparse:moderate}.
\end{lemma}
 Very roughly speaking, aside from the choices of parameters for the GAPs that might contain the most elements of $w$, and of the exceptional elements after applying Theorem~\ref{theorem:ILO}, the key estimate leading to Lemma~\ref{L:strucnotsp} is that 
$$(\rho^{-1}/\sqrt{\alpha_n n'})^n \rho^{n_0} = O(n^{-\eps n /5})$$  
as long as $n^{-O(1)} \le \rho \le O(n^{-\eps/2})$. We now present the details. 
\begin{proof}(of Lemma \ref{L:strucnotsp})
Throughout the proof, we assume $n$ is sufficiently large given $\eps$.
Suppose we have such a $w$.  By Theorem~\ref{theorem:LO} and $|\supp(w)|>\cz n$, as long as $\alpha_n \geq 4/(\cz n)$,
we have $\rho(w)\leq p^{-1} +2/\sqrt{\alpha_n \cz n}\leq (1+2\cz^{-1/2}) n^{-\eps/2}$, since $p\geq n^{\eps/2}$ and $\alpha_n\geq n^{-1+\eps}$.

 Let $Q$ be a symmetric GAP in $\Z/p\Z$ of rank at most $r$ and volume at most $V$, such that for some subset
$\tau\sub [n]$ of size $n'$ we have for $j\in ([n]\setminus\tau)$ that $w_{j}\in Q$.
Let $\Ro_1,\dots, \Ro_n$ denote the rows of the matrix $M$ formed by the columns $X_1/p,\dots,X_{n_0}/p$.
For $n$ sufficiently large (in terms of $\epsilon$) such that $c_0n\geq n'$, we see that the $R_j$ for  $j\in\tau$ must be linearly 
independent (or else there would be a normal vector to $W_{n_0}/p$ with at most $c_0n$ non-zero coefficients).  

First, we will determine how many possible choices there are for the data of $Q$, $\tau$, $\sigma$, and the $w_{j}$ for $j\in\tau$, without any attempt to be sharp.
Then, given those data, we will determine the probability that $X_1,\dots,X_{n_0}$ could produce the situation outlined above with those data.

So we have at most $p^r$ choices of generators for $Q$
and at most $V^r$ choices of dimensions (to obtain a lower rank GAP we just take some dimensions to be $0$).  There are at most $2^n$ choices of $\tau$, and at most $2^n$ choices of $\sigma$.  
There are at most $V^m$ choices of $w_{j}$ for $j\in([n]\setminus\tau)$.

Given $Q$, $\tau$,  $\sigma$, and the $w_{j}$ for $j\in\tau$, we condition on the $X_i$ for $i\in \sigma$.
Then the $\tau$ entries of $w$ are determined by the $w_j$ for $j\in ([n]\setminus\tau)$ and the $X_i$ for $i\in \sigma$ as follows.
From $w\cdot X_i/p=0$ for $i\in\sigma$, it follows that
\begin{equation}\label{eqn:moderate:0}
\sum_{j\in\tau} w_j \Ro_j|_\sigma = -\sum_{j\in([n]\setminus\tau)} w_j \Ro_j|_\sigma.
\end{equation}
Since the $\Ro_j|_\sigma$ (meaning row $\Ro_j$ restricted to the $\sigma$ entries) for $j\in \tau$ are linearly independent and $|\tau|=|\sigma|$, we conclude that the $w_j$ for $j\in ([n]\setminus\tau)$ and $X_i$ for $i\in \sigma$ determine at most one possible choice for the $w_j$ for $j\in\tau$.  

For simplicity, we will proceed in two cases.  First, we will determine how likely it is for $W_{n_0}/p$ to have a normal vector $w$ as in the lemma statement such that $\rho(w)\leq n^{-9}.$  From the well-boundedness of $Q$, we have that
$r\leq C_\eps$ and
 $V\leq C_\eps n^3$.  Thus the total number of choices for $Q$, $\tau$,  $\sigma$, and the $w_{j}$ for $j\in\tau$ is at most 
$
p^{C_\eps} (C_\eps n^3)^{C_\eps+m} 4^n.
$
Once we condition on  the $X_i$ for $i\in \sigma$, the vector $w$ is determined by our choices, and the probability that $w\cdot X_i/p=0$ for $i\in([n_0]\setminus \sigma)$ is at most $n^{-9(n_0-n')}.$  Thus the total probability that $W_{n_0}/p$ has a normal vector $w$ as in the lemma statement such that $\rho(w)\leq n^{-9}$ is at most
$$
p^{C_\eps} (C_\eps n^3)^{C_\eps+m} 4^n n^{-9(n_0-n')}=O_{\eps}(p^{C_\eps} n^{-n}).
$$

Next we will determine how likely it is for $W_{n_0}/p$ to have a normal vector $w$ as in the lemma statement such that $\rho(w)> n^{-9}.$ 
However, instead of counting the $Q$ from the lemma statement, we are going to count the $Q$ provided by Theorem \ref{theorem:ILO}. 
 More specifically, we divide
$[n^{-9},(1+2\cz^{-1/2}) n^{-\eps/2}]$  into dyadic subintervals $I_\ell=[\rho_\ell,
2\rho_{\ell}]$ and we suppose that $\rho(w) \in I_\ell$. 
Let $\rho=\rho(w)$.
We can apply Theorem \ref{theorem:ILO} with $C=3$ for $p\geq  C_\eps n^{3}$. 
Then there exists a  symmetric GAP $Q$ of rank
$r\leq C_\eps$ with $|Q|\leq \max((C_\eps(\rho^{-1}/(\al_n n')^{r/2},1),$ 
 and a subset $\tau\sub [n]$ of $n'$ indices such that for $j\in ([n]\setminus\tau)$,  we have $w_{j}\in Q$.
 Note that $r=0$ would imply that $|\supp(w)|\leq \lceil n^{\eps/2} \a_n^{-1}\rceil$, which contradicts the fact that $|\supp(w)|>\cz n$.  Also, since $\rho^{-1}\geq (1+2\cz^{-1/2})^{-1} n^{\eps/2}$, we have that
 $C'_\eps \rho^{-1}/(\al_n n')^{1/2} \geq 1$ for some constant $C'_\eps\geq C_\eps$ only depending on $\eps$.  
 So $r\geq 1$, and
 \begin{equation}\label{eqn:moderate:Q}
 |Q|=O_{\ep}(\rho^{-1}/(\al_n n')^{1/2}).
 \end{equation}  

Since $\ell$ was chosen so that $\rho(w)\leq 2\rho_{\ell}$, we have that the probability that $w\cdot X_i/p=0$ for $i\in([n_0]\setminus \sigma)$ is at most $(2\rho_{\ell})^{n_0-|\sigma|}.$
Thus the total probability that there is a
$w$ as in the lemma statement such that $\rho(w)> n^{-9}$ is at most
\begin{align}\label{eqn:moderate:1}
\sum_{\ell=1}^{O(\log n)} p^{C_\eps} O_{\ep}(\rho_{\ell}^{-1}/(\al_n n')^{1/2})^{C_\eps+m} 4^n (2\rho_\ell)^{n_0-n'}
&\leq \sum_{\ell=1}^{O(\log n)} p^{C_\eps} e^{O_{\ep}(n)}
 O_{\eps}((\al_n n')^{-1/2})^m (\rho_{\ell}^{-1})^{n-n_0}\nonumber
 \\ 
 &\leq O_{\eps}(p^{C_\eps} n^{-\eps n/5}).
\end{align}
For these inequalities, we use facts including $\rho_\ell^{-1}\leq n^9$ and $n-n_0=\lfloor\frac{3\log n}{\alpha_n} \rfloor\leq 3n^{1-\eps}\log n$, and $(\al_n n')^{-1/2}\leq n^{-\eps/4},$ and $m=n-\lceil n^{\eps/2} \a_n^{-1}\rceil\geq n-\lceil n^{1-\eps/2}\rceil$.


\end{proof}

As good as the bounds in Lemmas~\ref{lemma:sparse:moderate} and \ref{L:strucnotsp} are, they still cannot be summed over all primes $p$ that might divide the determinant of $M_{n\times n}$.  So at some point, we need to lift the structured normal vectors from characteristic $p$ to characteristic $0$.  Unlike in Section~\ref{S:nottoosparse}, when we could lift non-sparse normal vectors for all large primes, our structured vectors here have more noise and we cannot lift until the primes are even larger.  The following lemma does this lifting and is the only place we use that the coefficients of the $X_i$ are bounded.  Instead of counting structured vectors in characteristic $0$ (or modulo a prime $>n^{n/2}$), for which we would need some bound on their coefficients (e.g., see the commensurability results \cite[Lemma 9.1]{Ng} and \cite[Theorem 5.2(iii)]{TV}), we  prove in the following lemma that we can also reduce structured vectors in characteristic $0$ to structured vectors modulo a prime around $e^{n^{1-\eps/3}}$.  This allows us to transfer structured vectors modulo $p$ for our largest range of $p$ to structured vectors for a single prime $p_0$ that is of reasonably controlled size.
 
We say a submodule of $\Z^n$ is \emph{admissible} if it is generated by vectors with coordinates at most $n^T$ in absolute value.  (In particular, $W_k$ is always admissible.) 

\begin{lemma}[Lifting and reducing structured vectors]\label{lemma:passing'}
Let $A$ be an admissible submodule of $\Z^n$, and $p$ be a prime $\geq e^{n^{1-\eps/3}}$,
 and $n$ be sufficiently large given $\eps$ and $T$.
Then $A$ has a structured normal vector (for a GAP with integral generators) if and only if $A/p$ has a structured normal vector.
\end{lemma}

\begin{proof}(of Lemma \ref{lemma:passing'}) 
We will first prove the ``if'' direction.
Assume that the first $m=n-n'$ entries of the normal vector $w=(w_1,\dots, w_n)$ belong to a symmetric well-bounded GAP $Q$ with $r$ generators $a_1,\dots, a_r$ in $\Z/p\Z$, and $w_{j}  = \sum_{l=1}^r  x_{j l} a_l$ for $ 1\le j\le m$. 
Let $M$ be the matrix with entries at most $n^T$ in absolute value whose columns generate $A$.   Let $R_1,\dots, R_n$  be the rows of $M$. We have the equality modulo $p$
$$0 = \sum_{j=1}^m w_{j} R_{j}+   \sum_{j=m+1}^n w_{j} R_{j}  = \sum_{l=1}^r a_l (\sum_{j=1}^m x_{j l} R_{j}) +  \sum_{j=m+1}^n w_{j} R_{j} .$$
Now for $1\leq l\leq r$, let  $Z_l:= \sum_{j=1}^m x_{j l} R_{j}$.
We have $|x_{j l}|\leq |Q|\leq C_\eps n^3$.
  The entries of $Z_l$ are then bounded by $C_\eps n^{T+4}$, which is $\leq n^{T+5}$ for $n$ sufficiently large given $\eps$, while the entries of $R_{{m+1}},\dots, R_{n}$ are bounded by $n^T$. 
Let $M'$ be the matrix whose columns are $Z_1,\dots,Z_r,R_{m+1},\dots R_n$.   
  The above identity then implies that $(a_1,\dots,a_r,w_{m+1},\dots w_n)^t$ is in the kernel of $M'$.
  Lemma \ref{lemma:passing} \eqref{i:kernel} applied to $M'$, with $k=r+n'$ 
  implies that as long as $p\ge  e^{(k \log k)/2 + k(T+5) \log n}$ (which is satisfied because $p\ge e^{n^{1-\eps/3}}$, and $r\leq C_\eps$, and $n'\leq n^{1-\eps/2}+1$, and $n$ is sufficiently large given $\eps$ and $T$), then
  there exist integers $a_l', w_j'$, reducing mod $p$ to $a_l, w_j$, for $1\leq l\leq r$ and $m+1\leq j\leq n$,
   such that
$$\sum_{k=1}^r a_l' Z_l + \sum_{j=m+1}^n w_j' R_{j}=0.$$
Let $w'=(w_1',\dots,w_n')$ where $w_{j}'= \sum_{l=1}^r  x_{j l} a_l'$ for $1\le j\le m$. By definition the $w_{j}'$ for $1\le j\le m$ belong to the symmetric  GAP with generators $a'_l$ and with the same rank and dimensions as $Q$, and   $w'$ is  normal to $A$.  Further $w'$ is non-zero since it reduces to $w$ mod $p$.

The ``only if'' direction appears easier at first---if we start with a structured normal vector, we can reduce the generators of the GAP and the normal vector mod $p$ for any prime $p$.  However, the difficulty is that for general primes $p$ it is possible for the generators $a_l$ of the GAP to be not all $0$ mod $p$, but yet the resulting normal vector  $w$ to be $0$ mod $p$.  
Given $A$, we choose $w$ minimal (e.g. with $\sum_i |w_i|$ minimal) so that the first $m=n-n'$ entries (without loss of generality) of the normal vector $w=(w_1,\dots, w_n)$ to $A$ belong to a symmetric well-bounded GAP $Q$ with $r$ generators $a_1,\dots, a_r$ in $\Z$, and $w_{j}  = \sum_{l=1}^r  x_{j l} a_l $ for  $1\le j\le m$ and $w$ is non-zero. 
 Let $M_x$ be the $n\times (r+n')$ matrix with entries $x_{jl}$ in the first $m$ rows and $r$ columns, the $n'\times n'$ identity matrix in the last $n'$ rows and columns, and zeroes elsewhere.  
So for $a:=(a_1,\dots,a_r,w_{m+1},\dots w_n)^t$, we have $M_x a=w^t.$ 

Certainly by minimality of $w$ at least some coordinate of $w$  is not divisible by $p$ (else we could divide the $a_l$ and $w_j$ all by $p$ and produce a smaller structured normal $w$).  
Suppose, for the sake of contradiction that all of the coordinates of $w$ are divisible by $p$.
 The entries of $M_x$ are bounded by $C_\eps n^3$,
so, as above, for $p\geq e^{n^{1-\eps/3}}$, by Lemma \ref{lemma:passing} \eqref{i:kernel} we have that $\ker M_x|_{\Z^{r+n'}}$ surjects onto $\ker M_x/p.$   So $a/p$ is in the kernel of $M_x/p$, and choose some lift $a':=(a'_1,\dots,a'_r,w'_{m+1},\dots w'_n)^t\in\Z^n$ of $a/p$ in the kernel of $M_x$.  Then $a-a'\in p \Z^n$, and $M_x(\frac{1}{p}(a-a'))=\frac{1}{p}w$.  Note that $\frac{1}{p}w$ is non-zero integral normal vector to $A$, and the equality $M_x(\frac{1}{p}(a-a'))=\frac{1}{p}w$ shows that all but $n'$ of the coordinates of 
$\frac{1}{p}w$ belong to a symmetric well-bounded GAP with integral generators and the same rank and volume as $Q$, contradicting the minimality of $w$.  Thus we conclude that $w/p$ is non-zero and thus a structured normal vector of $A/p$ for GAP $Q/p$.
\end{proof}

We now conclude the main result of this section.

\begin{proof}[Proof of Lemma~\ref{L:upgrade}]
We let $\mathcal{S}'$ be the set of  submodules $H$ of $\Z^n$ such that for all primes $p>e^{d\alpha_n n}$, the vector space $H/p$ has no structured normal vector $w$.
We assume throughout the proof that $n$ is sufficiently large given $\eps, T, d$.
First, we will bound $\P(W_{n_0}\not\in \mathcal{S}')$.  By Lemma~\ref{lemma:passing'}, for $p\geq e^{n^{1-\eps/4}}$, if $W_{n_0}/p$ has a structured normal vector, then $W_{n_0}$ has a structured normal vector, and then
$W_{n_0}/p'$ has a structured normal vector for every prime $p'$ with $e^{n^{1-\eps/3}} \leq p' < e^{n^{1-\eps/4}}$ (of which there is at least 1). 

So it suffices to bound the condition that $W_{n_0}/p$ has a structured normal vector for $p$ is a prime  $C_\eps n^3\leq p< e^{n^{1-\eps/4}}.$ 

We will include in our upper bound the probability that $W_{n_0}/p$ has a non-zero normal vector $w$ with $|\supp(w)|\leq \cz n$ for some prime $p<e^{n^{1-\eps/4}}$, which is at most 
$e^{-\co \alpha_n n/2}+ e^{n^{1-\eps/4}}(2/3)^{n/4}$ by Lemmas~\ref{L:supersparse} and \ref{lemma:sparse:moderate}.  Then, otherwise, by Lemma~\ref{L:strucnotsp}, it is probability at most 
$e^{n^{1-\eps/4}}O_{\eps,T}(e^{C_\eps n^{1-\eps/4}}  n^{-\eps n/5})$ that, for some prime $p<e^{n^{1-\eps/4}}$, the space $W_{n_0}/p$ has a structured normal vector $w$.  We conclude that $\P(W_{n_0}\in \mathcal{S}')\geq 1-e^{-\ct \alpha_n n}$ for some absolute constant $\ct$.

If $H\in \mathcal{S}'$ and $H'$ is a proper subspace of $(\Z/p\Z)^n$ containing $H/p$, then $H'$ has some non-zero normal vector $w$ (also normal to $H/p$).
Let $p>e^{d\alpha_n n}$ be a prime.  If $\rho(w)\leq n^{-3}$, then since $\P(X/p\in H')\leq \P(X/p \cdot  w =0)$ we have $\P(X/p\in H')\leq n^{-3}$.  Otherwise, if $\rho(w)> n^{-3},$ we apply
Theorem~\ref{theorem:ILO} with $C=3$ and find a  symmetric well-bounded GAP containing all but $n'$ coordinates of $w$, which contradicts the definition of $\mathcal{S}'$.
\end{proof}



\section{Laplacian of random digraphs: proof of Theorem ~\ref{theorem:sur:L}}\label{section:Laplacian}
As laid out in Section~\ref{section:method}, it suffices to prove
 Proposition~\ref{prop:L} and this task consists of three parts, in the first part we modify the method of Section ~\ref{sec:small} to justify Equation~\eqref{eqn:small:L} for the small primes, in the second part we provide a complete proof for Equation~\eqref{eqn:medu:L} and~\eqref{eqn:med:L} regarding the medium primes by improving the method of \cite{M1,NP}, and in the last part we modify the method of Sections~\ref{section:largeprimes} and \ref{section:structures} to prove Equation~\eqref{eqn:large:n:L} and ~\eqref{eqn:large:n-1:L} for the large primes.

For $1\le i\le n$, we say that a random vector  $X=(x_1,\dots, x_n) \in \Z_0^n$, the set of vectors of zero entry sum in $\Z^n$, has type $\CT_i$ if $x_i = - (x_1 + \cdots + x_{i-1}+ x_{i+1}+\cdots +x_n)$ and $x_1,\dots, x_{i-1}, x_{i+1},\dots, x_n$ are i.i.d. copies of $\xi_n$ from~\eqref{eqn:alpha}. Recall that $L_{M_{n \times n}}$ is a random
matrix with independent columns $X_i$ sampled from $\CT_i$. Sometimes we will also denote this matrix by $L_{n\times n}$ for short.

 \vskip .2in
{ \bf{I. Proof of Equation~\eqref{eqn:small:L} of Proposition~\ref{prop:L}: treatment for small
    primes}}. In this subsection we modify the approach of Section
~\ref{sec:small} toward the Laplacian setting. 
We first prove the analog of Theorem~\ref{T:MomMat} for the Laplacian.  We will use the same approach as in \cite[Theorem 6.2]{W0} to consider an auxiliary matrix that lets us carry the argument from the i.i.d. case to the Laplacian case.
Let $a$ be the exponent of $G$.  Let $R=\Z/a\Z$ and $V=(\Z/a\Z)^n$.  
We let $M'$ be an $n \times n$ random  matrix with coefficients in $R$ with entries $X_{ij}$ distributed as $(M_{n\times n})_{ij}$ for $i\ne j$ and with $X_{ii}$ distributed uniformly in $R$, with all entries independent.
Let $F_0\in \Hom(V,R)$ be the map that sends each standard basis element to $1$.
Now, $M'$  and $L_{M_{n\times n}}$ do not have the same distribution, as the column sums of
 $M'$ can be anything and the column sums of $L_{M_{n\times n}}$ are zero, i.e. $F_0 L_{M_{n\times n}}=0$. 
However if we condition on $F_0 M'=0$, then we find that this conditioned distribution of $M'$ is the same as the distribution of $L_{M_{n\times n}}$.
Given $M'$ and conditioning on the off diagonal entries, we see that the probability that $F_0 M'=0$ is $a^{-n}$ (for any choice of off diagonal entries).
So any choice of off diagonal entries is equally likely in $L_{M_{n\times n}}$ as in $M'$ conditioned on $F_0 X=0$. 

So for $F\in \Hom(V,G)$, we have
\begin{align*}
\P(FL_{M_{n\times n}}=0)&=\P(FM'=0 | F_0 M'=0)
=\P(FM'=0 \textrm{ and } F_0 M'=0)a^{n}.
\end{align*}
Let $\tilde{F}\in\Hom(V, G\oplus R)$ be the sum of $F$ and $F_0$. Let $Z\sub V$ denote the vectors whose coordinates sum to $0$, i.e. 
$$Z=\{v\in V \ |\ F_0v=0\}.$$
Let $\Sur^* (V,G)$ denote the maps from $V$ to $G$ that are a surjection when restricted to $Z$.
We wish to estimate
\begin{align*}
\E(\#\Sur(S_{M_{n\times n}},G))&=\E(\#\Sur(Z/L_{M_{n\times n}}R^n,G))\\
&=\sum_{F\in \Sur(Z,G)}   \P(FL_{M_{n\times n}}=0)\\
&=\frac{1}{|G|}\sum_{F\in \Sur^*(V,G)} \P(FL_{M_{n\times n}}=0) \\&={|G|^{-1}a^n}\sum_{F\in \Sur^*(V,G)} \P(\tilde{F}M'=0).
\end{align*}
Note that if $F: V \ra G$ is a surjection when restricted to $Z$, then $\tilde{F}$ is a surjection from $V$ to $G\oplus R$.

Now we need a slight variant on Lemma~\ref{L:newcount} to bound $F\in  \Hom(V,G)$ such that $\tilde{F}$ is robust for a subgroup $H$ of $G$. 
\begin{lemma}[Count of robust $F$ for a subgroup $H$]\label{L:newcountL}
Let $\delta>0$, and $a,n\geq 1$  be integers, and $G$ be  finite abelian group of exponent dividing $a$.
Let $H$ be a subgroup of $G\oplus R$ of index $D>1$ and 
let $H=G_{\ell(D)}\sub \dots \sub G_2\sub G_1\sub G_0=G\oplus R$ be a maximal chain of proper subgroups.
Let $p_j=|G_{j-1}/G_j|$.
For $n$ sufficiently large given $G$, the number of $F\in\Hom(V,G\oplus R)$ such that $F$ composed with the projection onto $R$ is $1$ for each standard basis vector, and $F$ is robust for $H$ 
and for $1\leq j \leq \ell(D)$,  there are $w_j$ elements $i$ of $[n]$ such that $Fv_i \in G_{j-1}\setminus G_j $ is at most
$$
a^{-n} |H|^{n-\sum_j w_j} \prod_{j=1}^{\ell(D)} \binom{n}{w_j}|G_{j-1}|^{w_j} .
$$
\end{lemma}
We note that for $n$ sufficiently large in terms of $G$, the condition on the projection onto $R$ implies that $H$ surjects in the projection to $R$, and otherwise the proof of Lemma~\ref{L:newcountL} is analogous to that of Lemma~\ref{L:newcount}.  (See also \cite[Lemma 5.3]{W0}.)
 We can then apply Lemma~\ref{L:probdepthestimatecolumnNEW} as written to the maps $\tilde{F}$ with range $G\oplus R$ and the matrix $M'$.  The proof now follows the proof of Theorem~\ref{T:MomMat}, except that we are estimating
 $${|G|^{-1}a^n}\sum_{F\in \Sur^*(V,G)} \P(\tilde{F}M'=0).$$  The two sums of $|G|^{-n}$ over various $F$ are replaced by sums of $|G|^{-n}a^{-n}$, but proofs of the same bounds can be found in the proof of \cite[Theorem 6.2]{W0}.
We deduce 
\begin{align*}
\left|\E(\#\Sur(S_{M_n\times n},G))  -|G|^{-1} \right| \leq K_2n^{-c_2},
\end{align*}
and then deduce Equation~\eqref{eqn:small:L} of Proposition~\ref{prop:L}, just as we proved Theorem~\ref{th:allsmall} from 
Theorem~\ref{T:MomMat}.

 \vskip .2in
{\bf{II. Proof of  Equations~\eqref{eqn:medu:L} and~\eqref{eqn:med:L} of Proposition~\ref{prop:L}: treatment for  the medium
    primes}.}


  

In this subsection we fix a prime $p$ and will work with $\Z/p\Z$. As such, if not specified otherwise, all of the vectors and subspaces in this subsection are modulo $p$. For brevity, instead of $X_i/p$ or $W_i/p$, we just write $X_i$ or $W_i$. The co-dimensions (coranks) of subspaces, if not otherwise specified, are with respect to $\Z_0^n/p$. Although our main result, Theorem ~\ref{theorem:corank:L}, works for any subspace $W_{n-k}$ generated by $n-k$ columns of $L_{n\times n}$, for simplicity we assume $W_{n-k}=\lang X_{1},\dots, X_{n-k} \rang$. We show the following variant of Theorem ~\ref{thm:A1}. 
\begin{theorem}\label{theorem:corank:L}  There are sufficiently small constants $c, \eta>0$ and sufficiently large constants $C_0,K>0$ such that the following holds. Let $p$ be a prime, and let $L_{n \times n}$ be a random matrix with independent columns $X_i$ sampled from $\CT_i$ respectively, where we assume that 
\begin{equation}\label{eqn:alpha'}
\max_{r \in \Z/p\Z} \P(\xi_n=r) =1-\a_n \le 1 -\frac{C_0 \log n}{n}.
\end{equation}
Then for $1 \le k\le \eta n$ there exists an event $\CE_{n-k}$ on the $\sigma$-algebra generated by $X_{1},\dots,X_{n-k}$, all of probability at least $1- e^{-c\alpha_n n}$, such that for any $k_0$ with $k-1\le k_0\le \eta n$
$$\left | \P_{X_{n-k+1}}\Big(X_{n-k+1} \in W_{n-k} \big| \CE_{n-k} \wedge \codim(W_{n-k})=k_0\Big) - p^{-k_0} \right | \le  K e^{-c \alpha_n n}.$$
\end{theorem}

Combining with Theorem ~\ref{T:allerror} and with
appropriate choices of $c_0$ and $K_0$ we then deduce the part of Proposition~\ref{prop:L} for medium primes, analogous to the proof of Theorem~\ref{th:allmed}.

Now we give a proof of Theorem ~\ref{theorem:corank:L}. Our overall approach is similar to the proof of \cite[Theorems A.1 and A.4]{NP} (which is built on approaches in \cite{M1,TV}), but for the Laplacian we cannot apply these results because the column vectors, as well as the entries in each column, are not identically distributed any more.

 We would like to emphasize that in our argument below the positive constants $c, \beta, \delta, \eta, \la$ are sufficiently small and allowed to depend only on the constant $C_0$ in the bound \eqref{eqn:alpha'} of $\a_n$. We first introduce a version of Lemma ~\ref{lemma:O} and Corollary ~\ref{cor:fullrank} for $\alpha_n$-dense random variables in the Laplacian setting.

\begin{lemma}\label{lemma:O:L}
For a deterministic subspace $V$ of $\Z_0^n/p$ (or $\Z_0^n$) of dimension $d$ and for any $i$
$$\P_{X \in \CT_i}(X \in V) \le (1-\alpha_n)^{n-d-1}.$$
As a consequence, $X_1,\dots, X_{n-k}$ are linearly independent in $\Z_0^n/p$ with probability at least $1 -n(1-\a_n)^{k-1}.$ 
\end{lemma}
\begin{proof} If suffices to verify the first part. But for this we just project the vectors onto the coordinates of indices different from $i$, and then use Lemma ~\ref{lemma:O}.  
\end{proof}

We will also need the following variant of Theorem~\ref{theorem:LO}.

\begin{theorem}[forward Erd\H{o}s-Littlewood-Offord for the Laplacian]\label{theorem:LO:L} Suppose that $w=(w_1,\dots, w_n)\in (\Z/p\Z)^n$ does not have any component $w_j$ with multiplicity larger that $n-m$, then for any $i$
$$\sup_r|\P_{X\in \CT_i}(X \cdot w =r) -\frac{1}{p}| \le \frac{2}{\sqrt{\alpha_n m}}.$$

\end{theorem}
We remark that the classical Erd\H{o}s-Littlewood-Offord in characteristic zero implies that if
$w=(w_1,\dots, w_n)\in \Z^n$ does not have any component $w_j$ with
multiplicity larger that $n-m$, then for any $i$  
$$\sup_{r\in \Z} \P_{X\in \CT_i}(X \cdot w =r)| \le \frac{2}{\sqrt{\alpha_n m}}.$$
\begin{proof}(of Theorem~\ref{theorem:LO:L}) Assume that $X=(x_1,\dots, x_n) \in \CT_i$ for some $1\le i \le n$. Then 
$$x_1w_1+\dots +x_n w_n = x_1 (w_1-w_i)+ \dots + x_{i-1} (w_{i-1}-w_i)+x_{i+1} (w_{i+1}-w_i)+\dots+x_{n} (w_{n}-w_i).$$
By the assumption, at least $m$ entries $w_1-w_i, \dots, w_n-w_i$
are non-zero. Because $x_1,\dots, x_{i-1}, x_{i+1},\dots, x_n$ are i.i.d.,
we then can apply Theorem ~\ref{theorem:LO}.
\end{proof}

\subsection{Sparse subspace} Let $0<\delta, \eta$ be small constants
(independently from $\alpha_n$). Given a vector space $H \le
(\Z/p\Z)^n$, we call $H$ {\it $\delta$-sparse} if there is a non-zero vector $w$ with $|\supp(w)| \le \delta n$ such that $w \perp H$.

\begin{lemma}[random subspaces are not sparse, Laplacian case]\label{lemma:sparse:L} There exist absolute constant $c'$ and $C'$ such that the following holds  with $\a_n \ge \frac{C_0 \log n}{n}$. Let $\eps, \delta, \eta$ be constants such that $0<\eps<1/12$  and  $0 \le \delta, \eta \le \eps$. 

\begin{itemize}
\item (characteristic $p$) For $\xi_n$ satisfying Equation \eqref{eqn:alpha'}, and for $0\le k < \eta n$
$$\P_{X_{1},\dots, X_{n-k}}\left(W_{n-k}/p \mbox{ is not $\delta$-sparse}\right)\ge 1 -e^{-c'\alpha_n n}.$$
\item  (characteristic zero) For $\xi_n$ satisfying Equation ~\eqref{eqn:alpha}, and for $0\le k < \eta n$
$$\P_{X_{1},\dots, X_{n-k}}\left(W_{n-k} \mbox{ is not $\delta$-sparse in $\Z^n$}\right)\ge 1 -e^{-c'\alpha_n n}.$$
\end{itemize}
\end{lemma}
This result is actually a special case of Lemma~\ref{L:supersparse:L} and \ref{lemma:sparse:moderate:L}, which will be discussed in due course. In connection to Theorem ~\ref{theorem:LO:L}, it is more useful to connect
the sparseness property to the one of having an entry of high
multiplicity.

\begin{claim}\label{claim:sparse:multi} Assume that the random
  subspace $W_{n-k}/p$ does not accept any normal vector in $(\Z/p\Z)^n$
 of support size at most $\delta$, then it does not accept any normal vector  with an
 entry of multiplicity between $n-\delta n$ and $n-1$
 either. The same holds in the the characteristic zero case $\Z^n$.
\end{claim}

\begin{proof} This is because of the invariance property that if
  $w=(w_1,\dots,w_n)$ is normal to $W_{n-k}$ then so is any shifted vector $(w_1-w_0,\dots, w_n-w_0)$ to $W_{n-k}$.
\end{proof}

To conclude our treatment for the sparse case, given constants $\eps, \eta, \delta$ and the parameter $\alpha_n$ from \eqref{eqn:alpha'}, let $\CE_{k, dense}= \CE_{k, dense}(\eps, \eta, \delta)$ denote the event in the $\sigma$-algebra generated by $X_1,\dots, X_{n-k}$ considered in Lemma ~\ref{lemma:sparse:L}, then
\begin{equation}\label{eqn:dense:L}
\P(\CE_{k, dense}) \ge 1 - e^{-c' \alpha_n n}.
\end{equation}

As such we can simply condition on this event without any significant loss. In our next move, we will choose $\lambda>0$ to be a sufficiently small constant and show that it is highly unlikely that $W_{n-k}/p$ is some non $\delta$-sparse subspace (module) $V$ of co-dimension $k_0$ with $k-1\le k_0 \le \eta n$ such that  
$$e^{-\la \alpha_n n} <  \max_{i, X \in \CT_i}|\P(X \in V) - \frac{1}{p^{k_0}}|.$$
Let us simply call $V$ {\it bad} if this holds. For motivation, instead of bounding the probability that $W_{n-k}/p$ is bad, let us simplify it to bounding the probability that $X_1,\dots, X_{n-k}$ all belong to a bad subspace $V$. For this we will use the ``swapping method" from \cite{KKS, TV}, and this was also adapted by Maples in \cite{M1} for the modulo $p$ case. Roughly speaking, by letting the random variable $\xi_n$ be lazier at zero, the associated random vector $Y$ with this lazy random variable will stick to $V$ more often than $X$ does (see Lemma~\ref{lemma:swapping:L}), say $|\P(X\in V) - \frac{1}{p^{k_0}}| \le 0.51 |\P(Y\in V) - \frac{1}{p^{k_0}}|$. Hence if $|\P(X\in V) - \frac{1}{p^{k_0}}|$ is large enough, say larger than $\frac{16}{p^{k_0}}$, then we have that $\P(X\in V) \le (2/3)\P(Y \in V)$, which in turn leads to a very useful bound $\P(X_1,\dots, X_{n-k} \in V) \le (2/3)^{n-k} \P(Y_1,\dots, Y_{n-k} \in V)$. In what follows we will try to exploit this crucial exponential gain toward the Laplacian setting and toward the event that $X_1,\dots, X_{n-k}$ actually span a bad subspace.

\subsection{Semi-saturated subspace}\label{sub:semi} Given $0<\alpha, \delta, \la<1$. We call a subspace $V \le \Z_0^n/p$ of co-dimension $k_0 \le \eta n$ (with respect to $\Z_0^n/p$) {\it semi-saturated} (or {\it semi-sat} for short) with respect to these parameters if $V$ is not $\delta$-sparse and 
\begin{equation}\label{eqn:semi-sat:def}
e^{-\la \alpha_n n} <  \max_{i, X \in \CT_i}|\P(X \in V) - \frac{1}{p^{k_0}}| \le  \frac{16}{p^{k_0}}.
\end{equation}
Here we assume 
$$e^{-\la \alpha_n n}< \frac{16}{p^{k_0}}.$$
If this condition is not satisfied (such as when $p$ is sufficiently large), then the semi-saturated case can be omitted. Our main result of this part can be viewed as a structural theorem which says that semi-saturated subspaces can be ``captured" by a set of significantly fewer than $p^n$ vectors. 

\begin{lemma}\label{lemma:inverse:semi} For all $\beta>0$ and $\delta>0$ there exists $0<\la=\la(\beta, \delta)<1$ in the definition of semi-saturation and a deterministic set $\CR \subset (\Z/p\Z)^n$ of non $\delta$-sparse vectors and of size $|\CR| \le (2\beta^\delta)^n p^n$ such that  every semi-saturated $V$ is normal to a vector $R \in \CR$. In fact the conclusion holds for any subspace $V$ satisfying the LHS of \eqref{eqn:semi-sat:def}.
\end{lemma}


\begin{proof}(of Lemma \ref{lemma:inverse:semi}) Without loss of generality, assume that $e^{-\la \alpha_n n} < |\P(X \in V) - \frac{1}{p^{k_0}}| $ where $X\in \CT_1$. Equivalently, with $J=\{2,\dots,n\}$
$$e^{-\la \alpha_n n} < |\P(X|_J \in V|_J) - \frac{1}{p^{k_0}}|.$$
By \cite[Proposition 2.5]{M1} (see also \cite[Lemma A.12]{NP}), there exists a deterministic set $\CR' \subset (\Z/p\Z)^{n-1}$ of non $\delta$-sparse vectors and of size $|\CR'| \le (2\beta^\delta)^{n-1} p^{n-1}$ such that $V|_J$ is normal to a vector $R \in \CR'$. We then define $\CR$ by appending a first coordinate to the vectors of $\CR'$ to make them have zero entry-sum.
\end{proof}

Let $\CF_{n-k,k_0, semi-sat}$ be the event that $\codim(W_{n-k})=k_0$ and $W_{n-k}$ is semi-saturated. 
 
\begin{lemma}[random subspaces are not semi-saturated, Laplacian case]\label{lemma:semi-sat} Let $\beta, \delta>0$ be parameters such that $\beta^\delta <17^{-2}/2$. With $\la=\la(\beta,\delta)$ from Lemma ~\ref{lemma:inverse:semi} we have
$$\P(\CF_{n-k,k_0, semi-sat}) \le e^{-n}.$$
\end{lemma}
In particularly, with $\CE_{n-k, semi-sat}$ denotes the event complements $\wedge_{k-1\le k_0\le \eta n}\overline{\CF}_{n-k, k_0, semi-sat}$ in  the $\sigma$-algebra generated by $X_{1},\dots, X_{n-k}$,  then 
\begin{equation}\label{eqn:semi-sat}
\P(\CE_{n-k, semi-sat}) \ge 1- e^{-n/2}.
\end{equation}

\begin{proof}(of Lemma ~\ref{lemma:semi-sat}) We have
$$\P(\CF_{n-k, k_0, semi-sat}) = \sum_{V semi-sat, \ \codim(V)=k_0} \P(W_{n-k} = V) \le  \sum_{V semi-sat, \ \codim(V)=k_0} \P(X_{1},\dots, X_{n-k} \in V) .$$
Now for each fixed $V\le \Z_0^n/p$ that is semi-saturated of co-dimension $k_0$, by definition 
$$\P(X_{n+k+j} \in V) \le  \max_{i, X \in \CT_i}|\P(X \in V) - \frac{1}{p^{k_0}}| + \frac{1}{p^{k_0}} \le 17 p^{-k_0}.$$
So 
$$\P(X_1,\dots, X_{n-k} \in V) \le 17^{n-k} p^{-k_0(n-k)}.$$
We next use Lemma ~\ref{lemma:inverse:semi} to count the number $N_{semi-sat}$ of semi-saturated subspaces $V$. Each $V$ is determined by its annihilator $V^\perp$ in $\Z_0^n/p$ (of cardinality $p^{k_0}$). For $V^\perp$ we can choose a first vector $v_1 \in \CR$, and then $v_2,\dots, v_{k_0} \in \Z_0^n/p$ (linearly independently). By double counting, we obtain an upper bound
$$N_{semi-sat} =O\Big( (2\beta^\delta)^n p^n \frac{(p^{n-1})^{k_0-1}}{|V^\perp|^{k_0-1}}\Big) = O\Big((2\beta^\delta)^n p^{nk_0 -k_0^2+k_0}\Big).$$
Putting together,
\begin{align*}
\P(\CF_{n-k, k_0, semi-sat}) & \le  \sum_{V semi-sat,\ \codim(V)=k_0} \P(X_{1},\dots, X_{n-k} \in V) =O\Big( (2\beta^\delta)^n p^{nk_0 -k_0^2+k_0} 17^{n-k} p^{-k_0(n-k)} \Big)\\
& = O\Big( 17^{n-k} (2\beta^\delta)^n p^{k_0} p^{k_0(k-k_0)} \Big)= O\Big( 17^{n-k} (2\beta^\delta)^n p^{2k_0} \Big),
\end{align*}
where we noted that $k_0 \ge k-1$. Now recall that $e^{-\la \alpha_n n} \le 16 p^{-k_0}$, and so 
$$\P(\CF_{n-k, k_0, semi-sat}) =O(17^{n-k} (2\beta^\delta)^n p^{2k_0}) = O(17^{n+1-k} (2\beta^\delta)^n e^{2\la \alpha_n n}).$$
We then choose $\beta$ so that $2\beta^\delta <17^{-2}$ and with $\la<1/2$ we have
$$\P(\CF_{n-k, k_0, semi-sat}) \le e^{-n}.$$
\end{proof}

Having worked with subspaces $V$ where $\max_{i, X\in \CT_i}|\P(X \in V) -p^{-k_0}|$ are still small, we now turn to the remaining case to apply the swapping method.

\subsection{Unsaturated subspace}\label{sub:unsat} Let $V$ be a subspace of codimension $k_0$ in $\Z_0^n/p$ for some $k-1 \le k_0 \le \eta n$. We say that $V$ is {\it unsaturated} (or {\it unsat.} for short)  if $V$ is not $\delta$-sparse and 
$$\max( e^{-d \alpha n}, 16 p^{-k_0}) < \max_{i, X\in \CT_i}|\P(X \in V) -p^{-k_0}|.$$
In particularly this implies that 
$$ \max_{i, X\in \CT_i}\P(X \in V)\ge \max\{17 p^{-k_0} , \frac{16}{17} e^{-d \alpha n}\}.$$
In this case, for each $1\le i\le n$ we say that $V$ has {\it type $i$} if 
$$ \P_{X\in \CT_i}(X \in V) =  \max_{1\le j\le n, X\in \CT_j}\P(X \in V).$$
By taking union bound, it suffices to work with unsaturated subspace of type 1. So in what follows $X\in \CT_1$. The following is from \cite[Lemma 2.8]{M1} (see also \cite[Lemma A.15]{NP}).

\begin{lemma}\label{lemma:swapping:L}
There is a $\alpha_n'$-balanced probability distribution $\nu$ on $\Z/p\Z$ with $\alpha_n' = \alpha_n/64$ such that if $Y=(y_1,\dots,y_n) \in (\Z/p\Z)^n$ is a random vector with i.i.d. coefficients $y_2,\dots, y_n$ distributed according to $\nu$ and $y_1=-(y_2+\dots+y_n)$ then for any unsaturated proper subspace $V$ in $\Z_0^n/p$
$$|\P(X\in V) - \frac{1}{p^{k_0}}| \le (\frac{1}{2}+o(1)) |\P(Y\in V) - \frac{1}{p^{k_0}}| .$$ 
\end{lemma}
(To be more precise, \cite[Lemma 2.8]{M1} and \cite[Lemma A.15]{NP} stated for vectors of i.i.d. entries, but for Lemma~\ref{lemma:swapping:L} we just need to truncate the first coordinate from all vectors.) For short, we will say that the vector $Y$ from Lemma \ref{lemma:swapping:L} has type $\CT_1'$. It follows from the definition of unsaturation and from Lemma~\ref{lemma:swapping:L} that 
$$\P_{X\in \CT_1}(X\in V) \le \frac{2}{3} \P_{Y \in \CT_1'}(Y\in V).$$

\begin{definition} Let $V$ be a subspace in $(\Z/p\Z)^n$. Let $d_{comb}\in \{1/n, \dots, n^2/n\}$. We say that $V$ of type 1 has combinatorial codimension $d_{comb}$ if 
\begin{equation}\label{eqn:comb1}
(1- \alpha)^{d_{comb}} \le \P_{X\in \CT_1}(X\in V) \le (1-\alpha_n)^{d_{comb}-1/n}.
\end{equation}
\end{definition}
Now as we are in the unsaturated case, $\P(X \in V)\ge  \frac{16}{17} e^{-\la \alpha_n n}$, and so
\begin{equation}\label{eqn:comb2}
d_{comb} \le  2 \la n.
\end{equation}
In what follows we will fix $d_{comb}$ from the above range, noting that $d$ is sufficiently small, and there are only $O(n^2)$ choices of $d_{comb}$.

Let be fixed any $0<\delta_1<\delta_2 <1/3$ such that 
\begin{equation}\label{eqn:delta12}
16(\delta_2-\delta_1) (1+\log \frac{1}{\delta_2-\delta_1}) < \delta_1.
\end{equation}
Set 
$$r=\lfloor \delta_1 n \rfloor \mbox{ and } s=n-k-\lfloor \delta_2 n\rfloor.$$

Let $Y_1,\dots, Y_r \in \CT_1'$ be random vectors with entries distributed by $\nu$ obtained by Lemma ~\ref{lemma:swapping:L}, and $Z_1,\dots, Z_s \in \CT_1$ bee i.i.d. copies of a type 1 vector generated by $\mu$. Note that in what follows the subspaces $V$ are of given combinatorial dimension $d_{comb}$ as in Equation \eqref{eqn:comb1} and \eqref{eqn:comb2}.

\begin{lemma}[random subspaces are not unsaturated, Laplacian case]\label{lemma:unsat} 
\begin{align*}
\P\Big(X_{1},\dots, X_{n-k} \mbox{ span an unsat.  $V$ of type 1 of dim. between $r+s$ and $n-k$}\Big) \le (3/2)^{-\delta_1 n/4}.
\end{align*}
\end{lemma}
 Note that the event considered here is significantly harder to control than the event discussed in the paragraph preceding Subsection \ref{sub:semi}. This is also the place where \cite{M1} treated incorrectly by relying on \cite[Proposition 2.3]{M1} (although our situation here is more technical with vectors of dependent and extremely sparse entries.) To prove Lemma~\ref{lemma:unsat} we will actually show

\begin{lemma}\label{lemma:swapXY} Assume that $V$ is any subspace of type 1 and of dimension between $r+s$ and $n-k$ and $d_{comb} \le  2\la n$. Then we have
$$\P\Big(X_{1},\dots, X_{n-k} \mbox{ span } V\Big) \le (3/2)^{-r/2} \sum_{(i_1,\dots, i_{n-k-r-s})}\P\Big(Y_1,\dots, Y_r,Z_1,\dots, Z_s, X^{(i_1)},\dots, X^{(i_{n-k-r-s})} \mbox{ span } V\Big),$$
where $(i_1,\dots, i_{n-k-r-s})$ ranges over all subsets of size $n-k-r-s$  of $\{1,\dots,n-k\}$. 
\end{lemma}
To conclude Lemma ~\ref{lemma:unsat} we just use $(3/2)^{-r/2} \binom{n-k}{r+s} \le  (3/2)^{-\delta_1 n/4}$ (basing on Equation \eqref{eqn:delta12}) and  the fact that for each fixed $(i_1,\dots, i_{n-k-r-s})$
$$\sum_{V \le (\Z/p\Z)^n, \mbox{ \small{type 1}} , \ \codim(V) \ge k} \P(Y_1,\dots, Y_r, Z_1,\dots, Z_s,  X^{(i_1)},\dots, X^{(i_{n-k-r-s})} \mbox{ span } V)\le 1.$$

\begin{proof}(of lemma ~\ref{lemma:swapXY}) We use the swapping method from \cite{M1,TV}. First of all, by independence between $X_i,Y_j, Z_l$,
\begin{align}\label{eqn:unsat:1}
&\  \  \ \P\Big(X_{1},\dots, X_{n-k} \mbox{ span } V\Big) \times \P\Big(Y_1,\dots, Y_r, Z_1,\dots, Z_s \mbox{ linearly independent in } V\Big) \nonumber \\
&= \P\Big(X_{1},\dots, X_{n-k} \mbox{ span } V \wedge Y_1,\dots, Y_r, Z_1,\dots, Z_s \mbox{ linearly independent in } V\Big).
\end{align}
 Roughly speaking, the linear independence of  $Y_1, \dots, Z_s$ is to guarantee that we then can add a few other $X_i$ to form a new linear span of $V$, and by this way we can free the other $X_i$ from the role of spanning (see for instance Equation \eqref{eqn:unsat:3}). We next estimate $\P(Y_1,\dots, Y_r, Z_1,\dots, Z_s \mbox{ linearly independent in } V)$. By product rule,
\begin{align*}
& \  \  \ \P\Big(Z_1,\dots, Z_s, Y_1,\dots, Y_r \mbox{ linearly independent in } V\Big) \\
&=\P\Big(Y_r \in V\Big) \times \P\Big(Y_{r-1} \in V, Y_{r-1} \notin \lang Y_r \rang | Y_r \in V\Big) \times \dots \times  \P\Big(Y_1 \in V, Y_1 \notin \lang Y_2,\dots, Y_r \rang | Y_2,\dots, Y_r \mbox{ lin. in $V$}\Big) \times \\
&\times \P\Big(Z_{s} \in V, Z_{s} \notin \lang Y_1,\dots, Y_r \rang | Y_1,\dots, Y_r \mbox{ lin. in $V$} \Big) \times \dots \times \\
&\times  \P\Big (Z_1 \in V, Z_1 \notin \lang Z_2,\dots, Z_r, Y_1,\dots, Y_r \rang | Z_2,\dots, Z_s, Y_1,\dots, Y_r \mbox{ lin. in $V$} \Big).
\end{align*}
We first estimate the terms on $Y_i$. By Lemma ~\ref{lemma:O:L}
\begin{align*}
&\P\Big(Y_i \in V, Y_i \notin \lang Y_{i+1},\dots, Y_r \rang | Y_{i+1},\dots Y_r \mbox{ lin. in $V$} \Big)  \ge \P(Y_i\in V) - (1-\alpha_n')^{n-(r-i)-1}.
\end{align*}
This then can be estimated from below by
\begin{align*}
 \P(Y_i\in V) - (1-\alpha_n')^{n-(r-i)-1} \ge & \frac{3}{2}\P(X_i\in V) - (1-\alpha_n')^{n-(r-i)-1} \ge  \frac{3}{2}(1- \alpha)^{d_{comb}}   - (1-\alpha_n')^{n-(r-i)-1} \\
\ge & \frac{3}{2}(1-\alpha_n)^{d_{comb}} (1- (1-\alpha_n)^{n/256 -d_{comb}}),
\end{align*}
where we used that $\alpha_n' =\alpha_n/64$ and $n-r  \ge (1-\delta_1) n \ge n/2$. Similarly,
\begin{align*}
\P\Big(Z_i \in V, Z_i & \notin \lang Z_{i+1},\dots, Z_s, Y_1,\dots, Y_r \rang |  Z_{i+1},\dots, Z_s, Y_1,\dots, Y_r \Big)  \ge \P(Z_i\in V) - (1-\alpha')^{n-(r+s-i)-1}\\
& \ge (1- \alpha_n)^{d_{comb}}   - (1-\alpha_n')^{n-(r+s-i)-1} \ge (1-\alpha_n)^{d_{comb}}- (1-\alpha_n)^{n/256}
\end{align*}
where we used that $r+s = n-k - (\lfloor \delta_2 n \rfloor - \lfloor \delta_1 n \rfloor) \ge n/2$. Putting together
\begin{align}\label{eqn:unsat:2}
\P\Big (Y_1,\dots,Z_s \mbox{ linearly independent in } V \Big) &\ge (3/2)^r (1-\alpha_n)^{(r+s) d_{comb}} \Big(1- (1-\alpha_n)^{n/256 -d_{comb}}\Big)^{r+s} \nonumber \\
&\ge  (3/2)^{r-1} (1-\alpha_n)^{(r+s) d_{comb}},
\end{align}
where we used $d_{comb} \le 2 \la n$ and $\la$ is sufficiently small. 

Now we estimate the probability $\P(X_{1},\dots, X_{n-k} \mbox{ span } V \wedge Y_1,\dots, Y_r, Z_1,\dots, Z_s \mbox{ linearly independent in } V)$. Since $Y_1,\dots, Y_r, Z_1,\dots, Z_s$ are linearly independent in $V$ and $X_{k+1},\dots, X_n$ span  $V$, there exist $n-k-r-s$ vectors $X^{(i_1)
},\dots, X^{(i_{n-k-r-s})}$, which together with $Y_1,\dots,Y_r, Z_1,\dots, Z_s$, span $V$, and the remaining vectors $X^{i_{n-k-r-s+1}},\dots, X^{i_{n-k}}$ belong to $V$. Thus,
\begin{align}\label{eqn:unsat:3}
&\P\Big(X_{1},\dots, X_{n-k} \mbox{ span } V \wedge Y_1,\dots, Y_r, Z_1,\dots, Z_s \mbox{ linearly independent in } V\Big) \nonumber \\
\le & \sum_{(i_1,\dots, i_{n-k-r-s})}\P\Big(Y_1,\dots, Z_s, X^{(i_1)},\dots, X^{(i_{n-k-r-s})} \mbox{ span } V  \wedge X^{i_{n-k-r-s+1}},\dots, X^{i_{n-k}} \in V \Big)\nonumber \\
\le & \sum_{(i_1,\dots, i_{n-k-r-s})}\P\Big(Y_1,\dots, Z_s, X^{(i_1)},\dots, X^{(i_{n-k-r-s})} \mbox{ span } V\Big)  \P\Big(X^{i_{n-k-r-s+1}},\dots, X^{i_{n-k}}\in V\Big)\nonumber \\
\le & \sum_{(i_1,\dots, i_{n-k-r-s})}  \P\Big(Y_1,\dots, Z_s, X^{(i_1)},\dots, X^{(i_{n-k-r-s})} \mbox{ span } V\Big) (1-\alpha)^{(r+s) (d_{comb}-1/n)},
\end{align}
where in the last step we used the upper bound  $(1-\alpha)^{d_{comb}-1/n}$ for each $\P(X^{(i)} \in V)$.

Putting \eqref{eqn:unsat:1}, \eqref{eqn:unsat:2} and \eqref{eqn:unsat:3} together,  
\begin{align*}\label{eqn:unsat:3}
& \P\Big(X_{1},\dots, X_{n-k} \mbox{ span } V\Big)=\frac{\P\Big(X_{1},\dots, X_{n-k} \mbox{ span } V \wedge Y_1,\dots, Z_s \mbox{ linearly independent in } V\Big)}{\P\Big(Y_1, \dots, Z_s \mbox{ linearly independent in } V\Big)} \nonumber \\
& \le (3/2)^{-r+1} (1-\alpha_n)^{-(r+s) d_{comb}} \times  \nonumber  \\
& \times  \sum_{(i_1,\dots, i_{n-k-r-s})}  \P\Big(Y_1, \dots, Z_s,  X^{(i_1)},\dots, X^{(i_{n-k-r-s})} \mbox{ span } V\ \Big) (1-\alpha_n)^{(r+s) (d_{comb}-1/n)} \nonumber  \\
&\le (3/2)^{-r/2}   \sum_{(i_1,\dots, i_{n-k-r-s})}  \P\Big(Y_1,\dots, Z_s,  X^{(i_1)},\dots, X^{(i_{n-k-r-s})} \mbox{ span } V \Big) .
\end{align*}

\end{proof}

Remark that $r+s= n-k- (\lfloor \delta_2 n \rfloor - \lfloor \delta_1 n\rfloor)< n-k -\eta n$ if $\eta$ is sufficiently small. As a consequence, if we let $\CE_{k, unsat}$ denote the complement of the event in Lemma \ref{lemma:unsat} in the $\sigma$-algebra generated by $X_{1},\dots, X_{n-k}$ then
\begin{equation}\label{eqn:unsat}
\P(\CE_{n-k, unsat}) \ge 1-(3/2)^{-\delta_1 n/4}.
\end{equation}

We now conclude the proof of Theorem ~\ref{theorem:corank:L}. Let $\CE_{n-k, dense}, \CE_{n-k, semi-sat}, \CE_{n-k, unsat}$ be the events introduced in \eqref{eqn:dense:L}, \eqref{eqn:semi-sat}, \eqref{eqn:unsat} and let $\CE_{n-k}$ be their intersection. If we choose $c \le \min \{c',\la\}$ then by definition, on these events, if $\codim(W_{n-k})=k_0$ then for a random vector $X$ of any type $\CT_i$ 
$$|\P(X\in W_{n-k})-\frac{1}{p^{k_0}}| \le e^{-c \alpha_n n},$$
completing the proof.

Finally, we conclude this section with an interesting consequence of Theorem ~\ref{theorem:corank:L} in light of singularity bounds for random matrices from \cite{BR,BVW,KKS, Ko,RV,TV}.
\begin{corollary}[Non-singularity of the Laplacian]\label{cor:singularity:L} There exist absolute constants $c_0, K_0>0$ such that the following holds. Assume that the i.i.d. entries are distributed according to a random variable $\xi_n$ taking integral values and such that for any prime $p$
$$\max_{x\in \Z} \P(\xi_n=x) = 1-\a_n \le 1 - \frac{C_0 \log n}{n}, \mbox{ for a sufficiently large constant $C_0$.}$$ 
Then with probability at least $1-K_0e^{-c_0 \alpha_n n}$ the matrix $L_{n \times (n-1)}$ of any $n-1$ columns of $L_{n\times n}$ has rank $n-1$ in $\R^n$.
\end{corollary}
Note that  we do not require $\xi_n$ to be bounded at all, and our sparseness is almost best possible.
\begin{proof}(of Corollary \ref{cor:singularity:L}) We assume $L_{n \times (n-1)}$ to be the matrix of the
  first $n-1$ columns. Choose a prime $p$ sufficiently large and we will show that $L_{n \times
    (n-1)}/p$ has rank $n-1$ with probability at least $1 -e^{-c
    \alpha_n n}$. By Lemma~\ref{lemma:O:L}, it suffices to
  bound the probability that $L_{n \times (n-1)}/p$ has rank  between
  $n-\eta n$ and $n-2$. For this we can deduce from Theorem ~\ref{theorem:corank:L} that if $1\le k\le \eta n$ for some sufficiently small $\eta$, then 
\begin{equation}\label{eqn:corank':bound}
\P\Big(\rank(L_{n \times (n-1)}) = (n-1)-k \Big) = O\Big(n^k (p^{-k^2} + e^{-c \alpha_n n})\Big).
\end{equation}
Indeed, the event $\rank(L_{n \times (n-1)}) = (n-1)-k$ implies that there exist $k$ column vectors $X_{i_1},\dots, X_{i_k}$ which belong to the subspace of $\Z_0^n/p$ of dimension $(n-1)-k$ generated by the remaining column vectors $X_i, i \neq i_1,\dots, i_k$. With a loss of a factor of $\binom{n-1}{k}$ in probability, we assume that $\{i_1,\dots, i_k\}=\{n-k,\dots, n-1\}$. We then use Theorem ~\ref{theorem:corank:L}
\begin{align*}
&\P\Big(X_{n-k},\dots, X_{n-1} \in W_{n-k-1} \wedge \codim(W_{n-k-1})=k\Big) \\
= & \P\Big(X_{n-k},\dots, X_{n-1} \in W_{n-k-1}  \wedge \CE_{n-k-1} \wedge  \codim(W_{n-k-1})=k\Big)+ O(e^{-c \alpha_n n})\\
\le &  \P\Big(X_{n-k},\dots, X_{n-1} \in W_{n-k-1}  | \CE_{n-k-1} \wedge  \codim(W_{n-k-1})=k\Big) + O(e^{-c \alpha_n n})\\
\le & \Big(p^{-k} + O(e^{-c \alpha_n n})\Big)^k + O(e^{-c \alpha n})=O(p^{-k^2} + e^{-c \alpha n}),
\end{align*}
proving ~\eqref{eqn:corank':bound}, and hence the corollary.
\end{proof}
\vskip .2in 

{\bf{III. Proof of  Equations~\eqref{eqn:large:n:L} and
    ~\eqref{eqn:large:n-1:L} of Proposition~\ref{prop:L}: treatment for large
    primes.}} 
Now we modify the approach of Section~\ref{section:largeprimes} and \ref{section:structures} to the Laplacian setting.    
    Let $d>0$ be a constant and 
 $$\CP_n=\Big\{p \textrm{ prime},   p \geq e^{d \alpha_n n} \Big\}.$$ 
 Let $\CE_{\neq 0}^{(L)}$ be the event that $L_{n \times (n-1)}$ has rank $n-1$ in $\R^n$. It follows from Corollary ~\ref{cor:singularity:L} that 
 $$\P(\CE_{\neq 0}^{(L)}) \ge 1 -K_0e^{- c_0 \alpha_n n}.$$ 
Our strategy is similar to the proof of Proposition~\ref{prop:large}. Recall that $W_{k}$ is the submodule of $\Z_0^n$ spanned by $X_1,\dots,X_k$. Let $\CB_k$ be the set of primes $p\in \CP_n$ such that $\rk(W_k/p) \le k-1$.
Let $\Con_k$ be the event that $|\CB_k|\leq  (2T+1) \log n/(2d \alpha_n )$ (the watch list is not too big).
Note that any $p\in \CB_k$ for $k\leq n$ must divide $\det(L_{n \times (n-1)})$.  By Hadamard's bound, $|\det(L_{n \times (n-1)})|\leq n^{n/2}n^{Tn}$, and so in particular, when $\bar{\Con}_k$ occurs then $\det(L_{n \times (n-1)}) = 0$.  Let $\Drop_k$ be the event that there is a $p\in \CB_k$ such that 
$\rk(W_k/p) \leq k-2$, this is the event we want to avoid for all $p$. 

We will show that $\P(\bar{\Con}_{k+1} \lor \bar{\Drop}_{k+1} | \Con_k \land  \bar{\Drop}_k)$, and hence $\P(\bar{\Con}_{k+1} \lor \bar{\Drop}_{k+1} | \bar{\Con}_k \lor \bar{\Drop}_k)$, are large.
The goal is to conclude that  $\P(\bar{\Con}_{n} \lor \bar{\Drop}_{n})$ is large, and since we know that $\P(\bar{\Con}_{n})$ is small, we can conclude that $\P(\bar{\Drop}_{n})$ is large, as desired.
Now to estimate $\P(\bar{\Con}_{k+1} \lor \bar{\Drop}_{k+1} | \Con_k \land  \bar{\Drop}_k)$ we will condition on the exact values of $X_1,\dots, X_k$ where $\Con_k \land  \bar{\Drop}_k$ holds, and so
there are at most $(2T+1)\log n/(2\la \alpha_n )$ primes $p\in \CP_n$ such that $\rk(W_k/p) \leq k-1$ and no prime 
$p\in \CP_n$ such that $\rk(W_k/p) \leq k-2$.  In this case  $\bar{\Drop}_{k+1}$, as long as for each $p\in \CB_k$, we have
$X_{k+1}/p\not\in W_k/p$.  Consider one prime $p\in\CB_k$, and let $V$ be the value of $W_k/p$ that the conditioned $X_1,\dots, X_k$ give.
From Lemma~\ref{lemma:O:L}, $\P(X_{k+1}/p\in V)\leq (1-\alpha_n)^{(n-1)-(k-1)} = (1-\a)^{n-k}.$  Thus, 
$$
\P(\bar{\Con}_{k+1} \lor \bar{\Drop}_{k+1} | \Con_k \land  \bar{\Drop}_k)\geq 1- \left(\frac{(2T+1)\log n}{2d \alpha_n } \right)(1-\alpha_n)^{n-k}.
$$
In particular, we have the same lower bound for $\P(\bar{\Con}_{k+1} \lor \bar{\Drop}_{k+1} | \bar{\Con}_k \lor \bar{\Drop}_k)$, and then inductively, we have
$$
\P(\bar{\Con}_{k} \lor \bar{\Drop}_{k} )\geq 1- \sum_{i=1}^{k-1}\left(\frac{(2T+1)\log n}{2d \alpha_n } \right)(1-\alpha_n)^{n-i}=
1- \left(\frac{(2T+1)\log n}{2d \alpha_n } \right)\frac{(1-\alpha_n)^{n-k+1}}{\alpha_n}.
$$
Set $n_0:=n-\lfloor 3\log n/\alpha_n \rfloor$. Then by using $\alpha_n\geq n^{-1 +\eps}$ we have that 
$$
\P\big(\bar{\Con}_{n_0} \lor \bar{\Drop}_{n_0} \big)\geq 1- O_{d,T}\big( n^{-1/2} \big).
$$
We then have the following analog of Lemma~\ref{lemma:sqrt}.
\begin{lemma}\label{lemma:sqrt:L} 
 Suppose that $\alpha_n\geq 6\log n/n$. Then there is a set of submodules  $\mathcal{S}_L$ of $\Z_0^n$ such that 
$$
\Prob(W_{n_0}\in \mathcal{S}_L)\geq 1-e^{-\alpha_n n/8}
$$
and for any prime $p\ge e^{d \alpha_n n}$, and any submodule $H\in\mathcal{S}_L$, for any proper subspace $H'$ of $\Z_0^n/p$ containing $H/p$,
$$\P\Big( X/p \in H'\Big)=O_{d,T}\left( \frac{ \sqrt{\log n}}{\alpha_n \sqrt{ n}} \right),$$
where $X$ is any column of $L_{n\times n}$.
\end{lemma} 
Lemma~\ref{lemma:sqrt:L} can be shown exactly the same way Lemma~\ref{lemma:sqrt} was deduced. Indeed,   we can use Lemma~\ref{lemma:passing} to lift the existence of sparse normal vector on any modulo $p$ with $p\ge e^{a \a_n n}$ to the existence of sparse normal vector on characteristic zero, for which we then can use Lemma~\ref{lemma:sparse:L} (or Lemmas~\ref{L:supersparse:L} and~\ref{lemma:sparse:moderate:L}) to show that this event is unlikely. We then apply Theorem~\ref{theorem:LO:L} (combined with Claim~\ref{claim:sparse:multi}) to get the desired probability bound when the normal vectors are non-sparse.

 Now similarly to the iid case, Lemma~\ref{lemma:sqrt:L} allows us to justify Equations~\eqref{eqn:large:n:L} and
    ~\eqref{eqn:large:n-1:L} only for $\a_n \ge n^{-1/6+\eps}$. To extend to $\a_n \ge n^{-1+\eps}$, we will have to need the following analog of Lemma~\ref{L:upgrade}.

\begin{lemma}\label{L:upgrade:L}
There is an absolute constant $\ct>0$ such that the following holds. 
Suppose that $\alpha_n\geq n^{-1+\epsilon}$.  
There is a set $\mathcal{S}_L'$ of submodules of $\Z_0^n$, such that
$$\P(W_{n_0}\in \mathcal{S}'_L)\geq 1- e^{-\ct \alpha_n n}$$ 
and for $n$ sufficiently large given $d,\epsilon, T$, 
 any prime $p\ge e^{d \alpha_n n}$,  any submodule $H\in\mathcal{S}_L'$, and any proper subspace $H'$ of $(\Z_0/p)^n$ containing $H/p$,
$$\max_{i}\P_{X\in \CT_i}\Big( X/p \in H'\Big)\leq n^{-3}.
$$
\end{lemma}

The deduction of Equations~\eqref{eqn:large:n:L} and
    ~\eqref{eqn:large:n-1:L} from this lemma is similar to how Proposition~\ref{prop:medium} was
deduced from Lemma~\ref{L:upgrade}. It remains to verify Lemma ~\ref{L:upgrade:L}.  For this we will make use of Theorem ~\ref{theorem:LO:L} and the following corollary of Theorem ~\ref{theorem:ILO} for vectors from $\CT_i$.

\begin{theorem}[inverse Erd\H{o}s-Littlewood-Offord for the Laplacian]\label{theorem:ILO:L} Let $\eps<1$ and $C$ be positive constants. Assume that $p$ is a prime that is larger than $C'n^C$ for a sufficiently large constant $C'$ depending on $\eps$ and $C$. Let $\xi$ be a random variable taking values in $\Z/p\Z$ which is $\a_n$-balanced with $a_n \ge n^{-1+\eps}$.  Assume $w=(w_1,\dots, w_n) \in (\Z/p\Z)^n$ such that  
 $$\rho (w) = \max_{i, (\xi_1,\dots, \xi_n) \in \CT_i}\sup_{a\in \Z/p\Z} \P(\xi_1 w_1 +\dots + \xi_n w_n=a) \ge  n^{-C},$$
 Then for any $n^{\ep/2} \al^{-1} \le n' \le n$, there exists $1\le i\le n$ and there exists a proper symmetric GAP $Q$ of rank $r=O_{\ep,C}(1)$ which contains all but $n'$ elements of $\{w_1-w_i,\dots, w_n-w_i\}$ (counting multiplicity), where 
$$|Q| \le \max\left \{1, O_{C,\ep}(\rho^{-1}/(\al n')^{r/2})\right \}.$$ 
\end{theorem}

\subsection{Proof of Lemma ~\ref{L:upgrade:L}} Our method is similar to Section~\ref{section:structures}, so we will be brief. First we need an analog of Lemma~\ref{L:supersparse} to estimate the probability that for some large prime $p$ the module $W_{n_0}/p$ of $\Z_0^n/p$ accepts an extremely sparse normal vector.

\begin{lemma}\label{L:supersparse:L}
There are absolute constants $\co, \Co$ such that the following holds. Let $\beta_n:=1- \max_{x\in \Z} \P(\xi_n=x)$, and assume $\beta_n\geq \Co \log n/n$ and $\alpha_n\geq 6 \log n/n$.
For $n\geq 2$,  the following happens with probability at most $e^{-\co \beta_n  n/2}$: for some prime $p>2n^T$,
the subspace $W_{n_0}/p$ has a non-zero normal vector with at most $144  \beta_n^{-1}$ non-zero coordinates.
\end{lemma}
We also need an analog of Lemma~\ref{lemma:sparse:moderate}, which will allow us to control the event that for some prime $p$ of order $e^{o(n)}$ or $p=0$ the subspace $W_{n_0}/p$ accepts a normal vector of $o(n)$ non-zero entries.

\begin{lemma}\label{lemma:sparse:moderate:L} 
There exist absolute
  constants $\cz,\Cz$ such that the following holds.  
Let $\beta_n:=1- \max_{x\in \Z} \P(\xi_n=x)$.  
  Assume  $\a_n \ge \frac{\Cz \log n}{n}$ and let $p$ be a prime $>2n^T$. 
  The following happens with probability at most  $(2/3)^{n/4}$: 
   the subspace $W_{n_0}/p$ has a non-zero normal vector $w$ with $144\beta_n^{-1} \leq |\supp(w)| \leq \cz n$.
\end{lemma}

Lemmas ~\ref{L:supersparse:L} and \ref{lemma:sparse:moderate:L} will be verified in Appendix~\ref{appendix:sparse} by following the proofs of Lemmas~\ref{L:supersparse} and \ref{lemma:sparse:moderate}. 

We next discuss an analog of Lemma~\ref{L:strucnotsp} on the existence of structured but not sparse normal vectors of $W_{n_0}/p$. Similarly to Section \ref{section:structures}, we let $n'=\lceil n^{\eps/2} \a_n^{-1}\rceil$ and $m=n-n'$, and we will apply Theorem~\ref{theorem:ILO:L} with this choice of $n'$ and $C=3$. By replacing $w=(w_1,\dots,w_n)$ by $(w_1-w_{i},\dots, w_n-w_{i})$ if needed (note that this shifted vector is again a normal vector of $W_{n_0}/p$ because this subspace consists of vectors of zero entry sum modulo $p$), we can simply say that Theorem~\ref{theorem:ILO:L} implies structure for $w_1,\dots, w_n$. 
We call a GAP $Q$ well-bounded if it is of rank $\leq C_\eps$ and $|Q|\leq C_\eps n^3$, where $C_\eps$ is the maximum of $C'$ and the constants from the $O_{C,\eps}$ bounding the rank of volume of $|Q|$. Motivated by this, and similarly to Section \ref{section:structures}, we call a vector $w$ structured if it is non-zero, and there exists a symmetric well-bounded GAP $Q$ such that all but $n'$ coordinates of $w$ belong to $Q$.

\begin{lemma}\label{L:strucnotsp:L}
Let $\alpha_n\geq n^{-1+\eps}$.
Let $p$  be a  prime $p\geq  C_\eps n^{3}$.
The following event happens with probability $O_{\eps}(p^{C_\eps} n^{-\eps n/5})$: the space
$W_{n_0}/p$ has a structured normal vector $w$, and 
$W_{n_0}/p$ does not have a non-zero normal vector $w'$ 
such that $|\supp(w')|\leq \cz n$ with $\cz$ from Lemma~\ref{lemma:sparse:moderate:L}.
\end{lemma}
Lemma~\ref{L:upgrade:L} then can be shown by combining Lemmas~\ref{L:supersparse:L},~\ref{lemma:sparse:moderate:L} and~\ref{lemma:passing'} the way Lemma~\ref{L:upgrade} was concluded in the end of Section~\ref{section:structures}. Finally, the proof of Lemma~\ref{L:strucnotsp:L} is almost identical to that of Lemma~\ref{L:strucnotsp}, the only difference is that we need to apply Theorems~\ref{theorem:LO:L} and ~\ref{theorem:ILO:L} instead of Theorems~\ref{theorem:LO} and~\ref{theorem:ILO} in the argument leading to Equations~\eqref{eqn:moderate:Q} and ~\eqref{eqn:moderate:1}, and thus we again omit the detail.

\section*{Acknowledgements}
We would like to thank Wesley Pegden and Philip Matchett Wood for helpful conversations. We are also grateful to Nathan Kaplan, Lionel Levine and Philip Matchett Wood for valuable comments on an earlier version of this manuscript. The first author is partially supported by National Science Foundation grants DMS-1600782 and DMS-1752345.
The second author is partially supported by a Packard Fellowship for Science and Engineering, a Sloan Research Fellowship,  National Science Foundation grants DMS-1301690 and DMS-1652116, and a Vilas Early Career Investigator Award.

\appendix

\section{Inequality Lemmas}

We have a straightforward inequality (easily checked by considering the cases when $A\leq 1$ and $A>1$).
\begin{lemma}\label{L:twoways}
Let $0<\ao<1$ and $A>0$, and  $n\geq 1$ be an integer.  Then we have
$$
\left(\ao+(1-\ao)\exp(-A) \right)^n \leq \max (\exp(-(1-\ao)An/2), \left((1+\ao)/2\right)^n).
$$ 
\end{lemma}

The following is a standard estimate with binomial coefficients.
\begin{lemma}\label{L:sumexpbound}
Let $\Ao>0$.  For every $\at >0$, for all $\delta>0$ sufficiently small (given $\Ao,\at $), we have that for all
sufficiently large $n$ (given $\Ao,\at ,\delta$), that
$$
\sum_{k=0}^{\lfloor \delta n \rfloor} \binom{n}{k} \Ao^k \leq \exp(\at  n).
$$
\end{lemma}

We put these lemma together to obtain the inequality below.
\begin{lemma}\label{L:dealwithonesum}
Let $\Aot>0$ and $0<\ao<1$.  
Then for positive $\ad$ sufficiently small (given $\Aot$ and $\ao$), the following holds.
Let $\Delta'> 2/(1-\ao)$.
For $n$ sufficiently large (given $\Aot,\ao,\ad,\Delta'$)
 we have
$$
\sum_{k=1}^{\lfloor \ad n \rfloor} \binom{n}{k} \Aot^k (\ao+(1-\ao)\exp( -(\Delta' \log n/n) k) )^n\leq  3 n^{-((1-\ao)\Delta'/2-1)}.
$$
\end{lemma}
\begin{proof}
For $n\geq 2$, we have that $(\Delta' \log n/n) k>0$.  So by Lemma~\ref{L:twoways} with $A=(\Delta' \log n/n) k$, we have
\begin{align*}
&\sum_{k=1}^{\lfloor \ad n \rfloor} \binom{n}{k} \Aot^k (\ao+(1-\ao)\exp( -(\Delta' \log n/n) k) )^n
\\ 
\leq 
&\sum_{k=1}^{\lfloor \ad n \rfloor} \binom{n}{k} \Aot^k\exp(-(1-\ao)An/2) + \sum_{k=1}^{\lfloor \ad n \rfloor} \binom{n}{k} \Aot^k \left((1+\ao)/2\right)^n\\
=& \sum_{k=1}^{\lfloor \ad n \rfloor} \binom{n}{k} \Aot^k\exp(-(1-\ao)\Delta' (\log n) k /2) + \sum_{k=1}^{\lfloor \ad n \rfloor} \binom{n}{k} \Aot^k \left((1+\ao)/2\right)^n.
\end{align*}
We have
\begin{align*}
r \cdot \binom{n}{k} \Aot^k\exp(-(1-\ao)\Delta' (\log n) k /2)\geq \binom{n}{k+1} \Aot^{k+1}\exp(-(1-\ao)\Delta' \log n (k+1) /2)
\end{align*}
if and only if
\begin{align*}
r \geq \frac{n-k}{k+1} \frac{\Aot}{n^{(1-\ao)\Delta'/2} }.
\end{align*}
If $(1-\ao)\Delta'/2>1$, then for $n$ sufficiently large given $\Aot,\ao$ and $\Delta'$, we have
$$
\frac{n-k}{k+1} \frac{\Aot}{n^{(1-\ao)\Delta'/2} } \leq \frac{\Aot n}{2n^{(1-\ao)\Delta'/2} } 
= \frac{\Aot}{2n^{(1-\ao)\Delta'/2-1} } \leq
\frac{1}{2}.
$$
So, if $(1-\ao)\Delta'/2>1$, then for $n$ sufficiently large given $\Aot,\ao$ and $\Delta'$, we have
\begin{align*}
\frac{1}{2} \cdot \binom{n}{k} \Aot^k\exp(-(1-\ao)\Delta' (\log n) k /2)\geq \binom{n}{k+1} \Aot^{k+1}\exp(-(1-\ao)\Delta' \log n (k+1) /2).
\end{align*}
In particular, that implies
\begin{align*}
\sum_{k=1}^{\lfloor \ad n \rfloor} \binom{n}{k} \Aot^k\exp(-(1-\ao)\Delta' (\log n) k /2) \leq
 &2 n^{-((1-\ao)\Delta'/2-1)}.
\end{align*}

For the second sum, let $\at =- \log(1+\ao/2)/2>0$ and note that $\at  +\log((1+\ao/2))<0$.  
Then by Lemma~\ref{L:sumexpbound}, we have that for $\ad$ sufficiently small given $\Aot, \ao$ that
for all $n$ sufficiently large (given $\Aot,\ao,\ad$)
$$
\sum_{k=1}^{\lfloor \ad n \rfloor} \binom{n}{k} \Aot^k \left((1+\ao)/2\right)^n \leq \exp(\at  n + \log(1+\ao/2) n )=
\exp( \log(1+\ao/2) n/2 ).
$$
 For $n$ sufficiently large given $\ao,\Delta'$, we have
 $$
 \exp( \log(1+\ao/2) n/2 ) \leq n^{-((1-\ao)\Delta'/2-1)}.
 $$
The lemma follows.
\end{proof}

\section{Approximate transition probabilities: proof of Theorem~\ref{T:allerror}}\label{sec:Gdel}

 First, we prove the following simplified version of the theorem.  Using a coupling, we show that if the transition probabilities of two sequences of random variables are close, then the distribution of the random variables must be close.

\begin{lemma}\label{L:allgooderrorwalk}
 Let $w$ and $w'$ be sequences of random variables with $w_0=w'_0=0$, for each $i\geq 0$
\begin{align*}
&\P(w'_{i+1}=a|w'_i=b)=\P(w_{i+1}=a|w_i=b) +\delta(i,b,a) \\
&\textrm{ for all $a$ and $b$ such that $\P(w'_i=b)=\P(w_i=b)\ne0$ },
\end{align*}
and $w_i$ and $w'_{i}$ only take on countably many values.
Then for any $n\geq0$ and any set $A$ of values taken by $w_n$ or $w'_n$, we have
$$
|\P(w_n\in A)-\P(w'_n\in A)|\leq \frac{1}{2}\sum_{0\leq i\leq n-1} \sum_{b} \sum_c |\delta(i,b,c)| \P(w_i =b) \leq  \frac{1}{2} \sum_{0\leq i\leq n-1} \max_b \sum_c |\delta(i,b,c)|,
$$
where 
$b$ is summed over $\{ b\ |\ \P(w_{k}=b)\ne 0 \textrm{ and } \P(w'_{k}=b)\ne 0)\}$ and
$c$ is summed over $\{ c\ |\ \P(w_{i+1}=c)\ne 0 \textrm{ or } \P(w'_{i+1}=c)\ne 0)\}.$
\end{lemma}

\begin{proof}
Let $S_i$ be the set of values taken on by $w_i$ and $w'_i$. Let $\mu$ be Lebesgue measure on the interval $[0,1]$.
For each $i\geq 0$ and $b\in S_i$, we can choose measurable functions $\phi_{i,b}: [0,1]\ra S_{i+1}$ and $\phi'_{i,b}: [0,1]\ra S_{i+1}$
such that for all $c\in S_i,$
$$
\P(w_{i+1}=c|w_i=b)= \mu(\phi_{i,b}^{-1}(c)) \quad \textrm{and} \quad \P(w'_{i+1}=c|w'_i=b)= \mu((\phi')_{i,b}^{-1}(c)),
$$
and $\mu(\{x\in[0,1] | \phi_{i,b}(x)\ne \phi'_{i,b}(x)\})=\frac{1}{2}\sum_{c\in S_{i+1}} |\delta(i,b,c)|.$  (If $b$ isn't a value taken by one of the variables, we will just take $\phi_{i,b}=\phi'_{i,b}$.)  Then we construct Markov chains $x_i$ and $x_i'$, with $x_0=x'_0=0$, and to determine $x_{i+1}$ and $x'_{i+1}$, we pick a random $x\in[0,1]$, and then let $x_{i+1}=\phi_{i,x_i}(x)$ and $x_{i+1}=\phi'_{i,x_i}(x).$
Note that  for all $i\geq0$ and all $b\in S_i$ and $c\in S_{i+1}$, we have
$$
\P(x_{i+1}=c|x_i=b)= \P(w_{i+1}=c|w_i=b) \quad \textrm{and} \quad \P(x'_{i+1}=c|x'_i=b)= \P(w'_{i+1}=c|w'_i=b).
$$
Thus, for all $n\geq 0$, we have $\Prob(x_i=a)=\Prob(w_i=a)$ and $\Prob(x'_i=a)=\Prob(w'_i=a).$
Note that $x_n$ and $x'_n$ are equal, unless for some $0\leq i \leq n-1$, we have that $x_i=x'_{i}$, but  $x_{i+1}\neq x'_{i+1},$
and in particular, $\phi_{i,x_i}(x)\ne \phi'_{i,x_i}(x).$  To see how likely this is for a given $i$, we sum over all $b\in S_i$, and have
\begin{align*}
\P(x_i=x'_{i} \land x_{i+1}\neq x'_{i+1} )\leq &\sum_{b\in S_i} \P(x_i=b)\P(x_i=x'_{i}=b \land x_{i+1}\neq x'_{i+1} |x_i=b)\\
\leq &\sum_{b\in S_i} \P(x_i=b)\P(\phi_{i,b}(x)\ne \phi'_{i,b}(x)  |x_i=b)\\
\leq &\sum_{b\in S_i}  \P(x_i=b) \frac{1}{2}\sum_{c\in S_{i+1}} |\delta(i,b,c)|.
\end{align*}
So 
$$\P(x_n\neq x_n') \leq \frac{1}{2} \sum_{i=0}^{n-1} \sum_{b\in S_i} \sum_{c\in S_{i+1}} \P(x_i=b)  |\delta(i,b,c)|,
$$
and from this the result follows, since 
$$|\P(w_n\in A)-\P(w'_n\in A)|=|\P(x_n\in A)-\P(x'_n\in A)|\leq \max(\P(x_n\in A \land x'_n\not\in A),\P(x_n\not\in A \land x'_n\in A)).$$
\end{proof}

Theorem~\ref{T:allerror} will follow by applying Lemma~\ref{L:allgooderrorwalk} to a modified sequence.

\begin{proof}[Proof of Theorem~\ref{T:allerror} ]
We insert half-steps $x_{i+1/2}=(x_i,g_i)$ and
$y_{i+1/2}=(y_i,1)$.  We compare the transitional probabilities as follows, first for $i$ integral:
$$
\P(y_{i+1/2}=(r,1)|y_i=r)- \P(x_{i+1/2}=(r,1)|x_i=r)=\P(g_i\ne 1 |x_i=r)
$$
and
$$
\P(y_{i+1/2}=(r,0)|y_i=r)- P(x_{i+1/2}=(r,1)|x_i=r)=-\P(g_i\ne 1 |x_i=r).
$$
Also for $i$ integral, we have
$$
\P(y_{i+1}=s|y_{i+1/2}=(r,1))- \P(x_{i+1}=s|x_{i+1/2}=(r,1))=\delta(i,r,s).
$$
Then applying Lemma~\ref{L:allgooderrorwalk}, we have
\begin{align*}
&|\P(x_n\in A)-\P(y_n\in A)|\\
&\leq \frac{1}{2}\sum_{i=0}^{n-1} \sum_{r} 2 \P(g_i\ne 1 |x_i=r) \P(x_i =r)  + 
\frac{1}{2}\sum_{i=0}^{n-1} \sum_{r} \sum_s |\delta(i,r,s)| \P(x_i =r)  \\
&= \sum_{i=0}^{n-1}  \P(g_i\ne 1)  + 
\frac{1}{2}\sum_{i=0}^{n-1} \sum_{r} \sum_s |\delta(i,r,s)| \P(
i =r)  .
\end{align*}
\end{proof}

\section{Inverse theorem: proof of Theorem \ref{theorem:ILO}}\label{section:ILO} 





We first introduce a more general structure in finite additive groups.

\begin{definition} A set $P$ in a given finite additive group $G$ is a {\it coset progression} of rank $r$ if it can be expressed as in the form of
$$H+Q,$$
where $H$ is a finite subgroup of $G$, and $Q= \{a_0+ x_1a_1 + \dots +x_r a_r| M_i \le x_i \le M_i' \hbox{ and $x_i\in \Z$ for all } 1 \leq i \leq r\}$ is a GAP of rank $r$.  
\begin{itemize}
\item We say that $P$ with this presentation (i.e. choice of $H$, $a_i$, $M_i$, $M_i'$) is \emph{proper} if  the sums $h+ a_0+ x_1a_1 + \dots +x_r a_r, h\in H,  M_i \le x_i \le M_i' $ are all distinct. 
\vskip .05in
\item More generally, given a positive integer $t$ we say that $P$ is $t$-proper with this presentation if $H+tQ$ is proper.
\vskip .05in
\item If $-M_i=M_i'$ for all $i\ge 1$ and $a_0=0$, then we say that $P$ with this presentation is {\it symmetric}.
\end{itemize}
\end{definition}

To prove Theorem \ref{theorem:ILO} we will make use of two results from \cite{TVjohn} by Tao and Vu. The first result allows one to pass from coset progressions to proper coset progressions without any substantial loss.

\begin{theorem}\cite[Corollary 1.18]{TVjohn}\label{thm:proper} There exists a positive integer $C_1$ such hat the following statement holds. Let $Q$ be a symmetric coset progression of rank $d\ge 0$ and let $t\ge 1$ be an integer. Then there exists a $t$-proper symmetric coset progression $P$ of rank at most $d$ such that we have
$$Q \subset P \subset Q_{{(C_1d)}^{3d/2}t}.$$ 
We also have the size bound
$$|Q| \le |P| \le t^d {(C_1d)}^{3d^2/2} |Q|.$$
\end{theorem}

The second result, which is directly relevant to us, says that as long as $|kX|$ grows slowly compared to $|X|$, then it can be contained in a structure. This is a long-ranged version of the Freiman-Ruzsa theorem.

\begin{theorem}\label{thm:longrange}\cite[Theorem 1.21]{TVjohn} There exists a positive integer $C_2$ such hat the following statement holds: whenever $d,k\ge 1$ and $X \subset G$ is a non-empty finite set such that 
$$k^d|X| \ge 2^{2^{C_2 d^2 2^{6d}}} |kX|,$$
then there exists a proper symmetric coset progression $H+Q$ of rank $0\le d'\le d-1$ and size $|H+Q| \ge 2^{-2^{C_2 d^2 2^{6d}}} k^{d'}|X|$ and $x,x' \in G$ such that 

$$x + (H+Q) \subset kX \subset x' + 2^{2^{C_2 d^2 2^{6d}}} (H+Q).$$

\end{theorem}

Note that any GAP $ Q=\{a_0+ x_1a_1 + \dots +x_r a_r|  -N_i \le x_i \le N_i \hbox{ for all } 1 \leq i \leq r\}$ is contained in a symmetric GAP  $ Q'= \{x_0 a_0+ x_1a_1 + \dots +x_r a_r| -1\le x_0 \le 1,  -N_i \le x_i \le N_i \hbox{ for all } 1 \leq i \leq r\}$. Thus, by combining Theorem \ref{thm:longrange} with Theorem \ref{thm:proper} we obtain the following

\begin{corollary}\label{cor:longrange} Whenever $d,k\ge 1$ and $X \subset G$ is a non-empty finite set such that 
$$k^d|X| \ge 2^{2^{C_2 d^2 2^{6d}}} |kX|,$$
then there exists a 2-proper symmetric coset progression $H+P$ of rank $0\le d'\le d$ and size $|H+P| \le 2^d (C_1 d)^{3d^2/2} 2^{d2^{C_2 d^2 2^{6d}}} |kX|$ such that 
$$ kX \subset H+P.$$

\end{corollary}


As for Theorem \ref{theorem:ILO}, the explicit constants in Corollary \ref{cor:longrange} will not be important. (Although a more careful analysis would allow $\a_n$ to be as small as $n^{-1+O(\frac{1}{\log \log \log n})}$ here, and hence in our main theorems. But in order to keep our presentation simple we will not work with this technical assumption, only staying with $\a_n \ge n^{-1+\eps}$.) Now we give a detailed proof of Theorem~\ref{theorem:ILO}. In general our method follows that of \cite{NgV}, but the details are more complicated because we have to obtain an actual inverse result in $\Z/p\Z$, as well as we need to take into account the almost sharp sparsity of the randomness. 

\begin{proof}(of Theorem \ref{theorem:ILO})
 First, for convenience we will pass to symmetric distributions. Let $\psi = \nu-\nu'$ be the symmetrization and let $\psi'$ be a lazy version of $\psi$ that 
\[
\P(\psi'=x) = \begin{cases}
\frac{1}{2}\P(\psi=x) \mbox{ if } x\neq 0\\ 
\P(\psi'=x)= \frac{1}{2}\P(\psi=x) + \frac{1}{2}, \mbox{ if } x=0.
\end{cases}
\]
Notice that $\psi'$ is symmetric as $\psi$ is symmetric. Similarly to \eqref{eqn:symmetrization}, we can check that $\max_x \P(\psi =x)  \le 1 -\alpha_n$, and so
$$ \sup_x \P(\psi'=x) \le 1 -\alpha_n/2.$$
We assume that $\P(\psi' =t_j) = \P(\psi' =-t_j) = \beta_j/2$ for $1\le j\le l$, and that $\P(\psi'=0) = \beta_0$, where $t_{j_1} \pm t_{j_2} \neq 0 \mod p$ for all $j_1 \neq j_2$.

Consider $a  \in \Z/p\Z$ where the maximum is attained, $\rho= \rho(w)=\P( S=a)$, here $S=\xi_1 w_1 +\dots + \xi_n w_n=a$. Using the standard notation $e_p(x)$ for $\exp(2\pi \sqrt{-1} x/p )$, we have
\begin{equation}\label{eqn:fourier1} \rho= \P(S=a)= \E \frac{1}{p} \sum_{x \in \Z/p\Z} e_p (x (S-a)) = \E \frac{1}{p} \sum_{x \in \Z/p\Z} e_p (\xi S) e_p(-x a) \le \frac{1}{p} \sum_{  x\in \Z/p\Z} |\E e_p (x S)|.
\end{equation}
By independence
\begin{equation*} \label{eqn:fourier2}  |\E e_p(x S)| = \prod_{i=1}^n |\E e_p(x \eta_i w_i)|\le \prod_{i=1}^n (\frac{1}{2}(|\E e_p(x \eta_i w_i)|^2+1)) = \prod_{i=1}^n |\E e_p( x \psi'  w_i)|  = \prod_{i=1}^n ( \beta_0 +\sum_{j=1}^l \beta_j  \cos \frac{2\pi x t_j w_i}{p}).  \end{equation*}
It follows that
\begin{equation} \label{eqn:fourier3} \rho  \le \frac{1}{p} |\sum_{x \in \Z/p\Z}  \prod_{i=1}^n (\beta_0 + \sum_{j=1}^l \beta_j  \cos \frac{2\pi x t_j w_i}{p}) |  \le \frac{1}{p} \sum_{x \in \Z/p\Z}  \prod_{i=1}^n (\beta_0 + \sum_{j=1}^l \beta_j  |\cos \frac{\pi x t_j w_i}{p} |)  , \end{equation}
where we made the change of variable $x \rightarrow x /2$ (in $\Z/p\Z$) and used the triangle inequality. 

By convexity, we have that   $|\sin  \pi z | \ge 2 \|z\|$ for any $z\in \R$, where $\|z\|:=\|z\|_{\R/\Z}$ is the distance of $z$ to the nearest integer. Thus,
\begin{equation} \label{eqn:fourier3-1}| \cos \frac{\pi x}{p}|  \le  1- \frac{1}{2} \sin^2 \frac{\pi x}{p}  \le 1 -2 \|\frac{x}{p} \|^2. \end{equation}
Hence for each $w_i$
$$\beta_0 + \sum_{j=1}^l \beta_j  |\cos \frac{\pi x t_j w_i}{p} | \le 1 - 2 \sum_{j=1}^l \beta_j \|\frac{x t_j w_i}{p} \|^2 \le \exp(-2 \sum_{j=1}^l \beta_j \|\frac{x t_j w_i}{p} \|^2).$$
Consequently, we obtain a key inequality
\begin{equation} \label{eqn:fourier4}
\rho \le  \frac{1}{p} \sum_{x \in \Z/p\Z}  \prod_{i=1}^n (\beta_0 + \sum_{j=1}^l \beta_j  |\cos \frac{\pi x t_j w_i}{p} |)   \le  \frac{1}{p} \sum_{x \in F_p} \exp( - 2 \sum_{i=1}^n \sum_{j=1}^l \beta_j \|\frac{x t_j w_i}{p} \|^2).
\end{equation}

{\it Large level sets.}  Now we consider the level sets $S_m:=\{\xi| \sum_{i=1}^n \sum_{j=1}^l \beta_j \|\frac{x t_j w_i}{p} \|^2  \le m  \} $.  We have
$$n^{-C} \le \rho  \le   \frac{1}{p} \sum_{x \in F_p} \exp( - 2 \sum_{i=1}^n \sum_{j=1}^l \beta_j \|\frac{x t_j w_i}{p} \|^2)  \le \frac{1}{p} + \frac{1}{p} \sum_{m \ge 1} \exp(-2(m-1)) |S_m| .$$
As $p$ is assumed to be much larger than $n^C$, and as $\sum_{m\ge 1} \exp(-m) < 1$, there must be  is a large level set $S_m$ such that
\begin{equation} \label{eqn:level1} |S_m| \exp(-m+2) \ge  \rho  p. \end{equation}
In fact, since $\rho \ge n^{-C}$, 
we can assume that $m=O(\log n)$.
\vskip .1in

{\it Double counting and the triangle inequality.} By  double counting we have
$$ \sum_{i=1}^n \sum_{x \in S_m} \sum_{j=1}^l \beta_j \|\frac{x t_j w_i}{p} \|^2 =   \sum_{x \in S_m} \sum_{i=1}^n  \sum_{j=1}^l \beta_j \|\frac{x t_j w_i}{p} \|^2  \le m |S_m |.$$
So, for most $v_i$
\begin{equation} \label{eqn:double1} \sum_{x \in S_m}\sum_{j=1}^l \beta_j \|\frac{x t_j w_i}{p} \|^2  \le \frac{m }{n'} |S_m| \end{equation} 
 for some large constant $C_0$.

By averaging, the set of $w_i$ satisfying \eqref{eqn:double1}
 has size at least $n-n'$.  We call this set $W'$. The set $\{w_1,\dots, w_n\}\backslash W'$ has size at most $n'$ and this is the
exceptional set that appears in Theorem \ref{theorem:ILO}. In the rest of the proof, we are going to show that $W'$ is a dense subset of a proper GAP.

Since $\|\cdot \|$ is a norm, by the triangle inequality, we have  for any $a  \in k W' $
\begin{equation} \label{eqn:double2} \sum_{x \in S_m} \sum_{j=1}^l \beta_j \|\frac{x t_j a}{p} \|^2  \le  k^2 \frac{m}{n'} |S_m| . \end{equation}
More generally, for any $k'  \le k $ and $a \in k'V'$
\begin{equation} \label{eqn:double3} \sum_{x \in S_m} \sum_{j=1}^l \beta_j \|\frac{x t_j a}{p} \|^2   \le  {k'}^2 \frac{m}{n'} |S_m| . \end{equation}

{\it Dual sets.} Set
$$\alpha_n' := \sum_{j=1}^l \beta_j =1-\beta_0.$$ 
Then by definition of $\xi$, we have 
$$\alpha_n' \ge \al_n/2 \ge n^{-1+\eps}.$$
Define  
$$S_m^{\ast} :=\{ a | \sum_{x\in S_m}  \sum_{j=1}^l \beta_j \|\frac{x t_j a}{p} \|^2  \le \frac{ \al_n' }{200} |S_m |\}$$ 
where the constant $200$ is ad hoc and any sufficiently large constant would do.
We have
\begin{equation} \label{eqn:dual1} |S_m^{\ast} |  \le \frac{8 p}{|S_m |} . \end{equation}
To see this, define $T_a :=\sum_{x \in S_m} \sum_{j=1}^l \beta_j \cos \frac{2\pi a t_j x}{p}$. Using the fact that $\cos 2\pi z \ge 1 -100 \|z\|^2 $ for any $z \in \R$, we have, for any $a \in S_m^{\ast}$
$$T_a \ge  \sum_{x \in S_m} (1- 100 \sum_{j=1}^l \beta_j \|\frac{x t_j a}{p} \|^2) \ge \frac{\al_n' }{2} |S_m |. $$
One the other hand, using the basic identity $\sum_{a \in \Z/p\Z} \cos \frac{2\pi  a z}{p} = p\I_{z=0} $, we have (taking into account that $t_{j_1} \neq t_{j_2} \mod p$)
$$\sum_{a \in \Z/p\Z} T_a^2 \le 2p |S_m| \sum_j \beta_j^2 \le  2p |S_m|  \max_{1\le j\le l} \beta_j (\sum_{j=1}^l \beta_j)  \le  2p |S_m|  {\al_n'}^2.$$
Equation \eqref{eqn:dual1} then follows from the last two estimates and averaging.

Next, for a properly chosen  constant $c_1$ we set
$$k := c_1 \sqrt{\frac{\al_n' n'}{m}}.$$ 
 By \eqref{eqn:double3} we have $\cup_{k'=1}^k  k' W'  \subset  S_m^{\ast} $. Next, set 
$$W^{''} := W' \cup \{0\}.$$ 
We have $k W^{''} \subset S_m^{\ast} \cup \{0\} $. This results in the critical bound
\begin{equation}  \label{eqn:dual2} |k W^{''} |  = O( \frac{p}{|S_m|}) = O(\rho^{-1} \exp(-m+2)) .  \end{equation}

{\it The long range inverse theorem.} We are now in the position to apply Corollary \ref{cor:longrange} with 
$X$ as the set of distinct elements of $W^{''}$. As $k = \Omega (\sqrt {\frac{\al_n' n'}{m}}) =\Omega(\sqrt{\frac{ \al_n' n'}{\log n}})$, 
\begin{equation}\label{eqn:LO:grow}
\rho^{-1} \le n^{C} \le k^{4C/\eps+1}.
\end{equation}

It follows from Corollary \ref{cor:longrange}  that $kX$ is a subset of
a 2-proper symmetric coset progression $H+P$ of rank $r=O_{C, \epsilon_0} (1)$ and cardinality
$$|H+P| \le O_{C,\eps} |kX|.$$
Now we use the special property of $\Z/p\Z$ that it has only trivial proper subgroup. As $|kX| =O(n^C)$, and as $p\gg n^C$, the only possibility that $|kX| \gg |H+P|$ is that $H=\{0\}$. 
Consequently,  $kX$ is now a subset of $P$, a 2-proper symmetric GAP of rank $r=O_{C, \epsilon_0} (1)$ and cardinality 
\begin{equation}\label{eqn:P:size}
|P| \le O_{C,\eps} |kX|.
\end{equation}
To this end, we apply the following dividing trick  from \cite[Lemma A.2]{NgV}.
\begin{lemma}\label{lemma:divide} 
Assume that $0 \in X$ and that $P=\{ \sum_{i=1}^r
x_ia_i: |x_i|\le N_i\}$ is a  2-proper  symmetric GAP that contains $kX$. Then 
$$X\subset \{\sum_{i=1}^r x_ia_i: |x_i|\le
2N_i/k\}.$$
\end{lemma}

\begin{proof}(of Lemma \ref{lemma:divide})
Without loss of generality, we  can assume that $k=2^l$. It is enough to
show that $2^{l-1}X \subset  \{\sum_{i=1}^r x_ia_i: |x_i|\le
N_i/2\}$. Since $0 \in X$, $2^{l-1}X \subset 2^lX\subset P$, any element
$x$ of $2^{l-1}X$ can be written as $x= \sum_{i=1}^r x_ia_i$, with
$|x_i|\le N_i$. Now, since $2x\in P \subset 2P$ and   $2P$ is proper GAP (as $P$ is 2-proper), we must
have $0\le |2x_i|\le N_i$.
\end{proof}
Combining \eqref{eqn:P:size} and Lemma \ref{lemma:divide} we thus obtain a GAP $Q$ that contains $X$ and
\begin{align*}
|Q| = O_{C, \epsilon_0}(k^{-r}|kX|)= O_{C, \epsilon_0} (k^{-r} |kW^{''}|)&= O_{C, \epsilon_0} \left( \rho^{-1} \exp(-m) (\sqrt {\frac{ \alpha_n' n'}{m}})^{-r}\right)\\
&= O_{C, \epsilon_0} (  \rho^{-1} (\alpha_n' n')^{-r} ),
\end{align*}
concluding the proof.
\end{proof}

\section{Sparse subspaces for the Laplacian case: proof of Lemmas~\ref{L:supersparse:L} and ~\ref{lemma:sparse:moderate:L}}\label{appendix:sparse}
Our methods are almost identical to those of Lemma~\ref{L:supersparse} and ~\ref{lemma:sparse:moderate}, with a few minor exceptions.

\begin{proof}(of Lemma \ref{L:supersparse:L})
Argue similarly as in the proof of Lemma~\ref{L:supersparse}, it suffices to show that the following holds with probability at least $1- e^{-\co \beta_n  n/2}$. For any $1\le t \le 144  \beta_n^{-1}$, and any $\sigma \in \binom{[n]}{t}$, there are at least two columns $X_i,X_{j}$ with $i, j \notin \sigma$ whose restriction $(X_{j}-X_i)|_\sigma$ has exactly one non-zero entry. 

For a given $\sigma$ of size $t$, assume that $\{i_1,\dots, j_{n_0-t}\} \subset [n_0]\bs \sigma$. For $i\in \{1, 3,\dots, 2\lfloor (n_0-t)/2 \rfloor -1\}$, consider the vectors $Y_i = X_{j_{i+1}}|_\sigma-X_{j_i}|_\sigma$. Note that as $j_i, j_{i+1} \notin \sigma$ and $X_{j_i}\in \CT_{j_i}$ and $X_{j_{i+1}}\in \CT_{j_{i+1}}$, the entries of $Y_i$ are iid copies of the symmetrized random variable $\psi =\xi-\xi'$, where $\xi',\xi$ are independent and have distribution $\xi_n$. Recall that with $1-\beta_n'=\P(\psi=0)$, then $\beta_n \le \beta_n'  \le 2\beta_n$. Now let $p_\sigma$ be the probability that all $Y_i|_\sigma, i\in \{1, 3,\dots, 2\lfloor (n_0-t)/2 \rfloor -1 \}$ fail to have exactly one non-zero entry (in $\Z$), then by independence of the columns and of the entries
$$p_\sigma = (1- t  \beta_n' (1- \beta_n')^{t-1})^{\lfloor (n_0-t)/2 \rfloor} \le (1 - t  \beta_n' e^{- (t-1) \beta_n'})^{\lfloor (n_0-t)/2 \rfloor} \le e^{-n t \beta_n' e^{-(t-1)  \beta_n'}/4},$$
where we used $n_0-t > n/2$ because $\alpha_n\geq 6 \log n/n$ and $t \le 144  \beta_n^{-1}$. The rest of the proof is similar to that of Lemma~\ref{L:supersparse}.
\end{proof}

\begin{proof}(of Lemma \ref{lemma:sparse:moderate:L}) 
For $\sigma \subset [n]$ with $144\beta_n^{-1}\le t=|\sigma| \le \cz n$, consider the event that 
 $W_{n_0}/p$ is normal to a vector $w$ with $\supp(w) =\sigma$ but not with any other vector of smaller support size.
With a loss of a multiplicative factor $\binom{n}{t}$ in probability, we assume that $\sigma=\{1,\dots,t\}$. 
Consider the submatrix $L_{t \times n_0}$ of $L_{n\times n}$ consisting of the first $t$ rows and first $n_0$ columns of $L_{n\times n_0}$. Since the restriction $w|_\sigma$ of $w$ to the first $t$ coordinates is normal to all the columns of 
$L_{t \times n_0}/p$, the matrix $L_{t \times n_0}/p$ has rank $t-1$ (if $p=0$, we mean rank over $\R$).  We assume that the column space of $L_{t \times n_0}/p$ is spanned the columns $\{X_{i_1},\dots, X_{i_{t-1}}\}$ for some $\{i_1,\dots, i_{t-1}\} \subset [n_0]$.  

Note that for $p>2n^T$, the value of $\xi_n$ is determined by its value mod $p$, and so $\beta_n=1- \max_{x\in \Z/p\Z} \P(\xi_n/p=x)$.   If we fix $X_{i_1}|_{\sigma},\dots,X_{i_{t-1}}|_{\sigma}$ such that the subspace $W_{i_1,\dots,i_{t-1}}|_\sigma/p$ generated by these vectors has a normal vector with all $t$ coordinates non-zero, then  
by Theorem \ref{theorem:LO} , the probability that $X_i|_\sigma/p\in W_{i_1,\dots,i_{t-1}}|_\sigma/p$ for all $i \in [n_0]\bs (\sigma \cup \{i_1,\dots, i_{t-1}\})$ (as for these vectors the entries of $\CT_i$ restricted to $\sigma$ are independent) is at most
$$(\frac{1}{p} + \frac{2}{\sqrt{ \beta_n  t}})^{n_0-2t-1}\le ( \frac{1}{p} + \frac{2}{\sqrt{ \beta_n t}})^{(1-3c_0) n} \le  (\frac{2}{3})^{n/2} $$
as long as $\Cz$ is sufficiently large and $\cz$ is sufficiently small. Thus the total probability of the event in the lemma is at most
$$ \sum_{144  \beta_n^{-1} \le t \le c_0 n}\binom{n}{t} \binom{n_0}{t-1} (\frac{2}{3})^{n/2} \le (\frac{2}{3})^{n/4}.$$
\end{proof}


\end{document}